\documentclass[11pt]{article}
\usepackage[utf8]{inputenc}
\usepackage{amssymb, graphicx, amsthm, amsmath, mathtools, fancyhdr, lastpage,verbatim}
\usepackage{mathrsfs}
\usepackage{svg, mathtools, extarrows, afterpage, booktabs, longtable, float, multirow, dashbox, xurl, hyperref, diagbox, subcaption, systeme, hyperref}
\usepackage{xcolor}
\usepackage[margin = 1 in]{geometry}
\usepackage[scr]{rsfso}
\usepackage[T1]{fontenc}
\usepackage[inline]{enumitem}
\usepackage[utf8]{inputenc}
\usepackage[backend=biber,style=alphabetic]{biblatex}
\usepackage[normalem]{ulem}
\svgsetup{inkscapelatex = false}
\hypersetup{colorlinks = true, urlcolor = blue, linkcolor = black, citecolor = black}
\bibliography{Bibliography}

\theoremstyle{definition}
\newtheorem{definition}{Definition}[section]
\theoremstyle{definition}

\theoremstyle{plain}
\newtheorem{theorem}[definition]{Theorem}
\theoremstyle{plain}
\newtheorem{assumption}[definition]{Assumption}
\theoremstyle{plain}
\newtheorem{lemma}[definition]{Lemma}
\theoremstyle{plain}
\newtheorem{remark}[definition]{Remark}
\theoremstyle{definition}

\theoremstyle{definition}

\theoremstyle{plain}
\newtheorem{corollary}[definition]{Corollary}
\theoremstyle{definition}
\newtheorem{proposition}[definition]{Proposition}
\theoremstyle{definition}

\allowdisplaybreaks[1]

\numberwithin{equation}{section}

\newcommand{\equi}{\Leftrightarrow}

\let\bg\gg

\makeatletter
\newcommand{\boldletter}[1]{%
  \expandafter\def\csname#1#1\endcsname{\mathbf{#1}}%
}
\@tfor\@letter:=abcdefghijklmnopqrstuvwxyzABCDEFGHIJKLMNOPQRSTUVWXYZ0123456789\do{%
  \expandafter\boldletter\expandafter{\@letter}%
}

\newcommand{\scriptletter}[1]{%
  \expandafter\def\csname#1#1#1\endcsname{\mathscr{#1}}%
}
\@tfor\@letter:=ABCDEFGHIJKLMNOPQRSTUVWXYZ\do{%
  \expandafter\scriptletter\expandafter{\@letter}%
}
\makeatother

\DeclareMathOperator{\di}{d\!}
\DeclareMathOperator{\dee}{d}
\DeclarePairedDelimiter\abs{\lvert}{\rvert}
\DeclarePairedDelimiter\bigabs{\big\lvert}{\big\rvert}

\DeclarePairedDelimiter\Bigabs{\Big\lvert}{\Big\rvert}

\DeclarePairedDelimiter\inner{\langle}{\rangle}
\DeclarePairedDelimiter\biginner{\big\langle}{\big\rangle}

\DeclarePairedDelimiter\Biginner{\Big\langle}{\Big\rangle}

\DeclarePairedDelimiter\lnorm{\lVert}{\rVert}

\DeclarePairedDelimiter\parens{(}{)}
\DeclarePairedDelimiter\bigparens{\big(}{\big)}

\DeclarePairedDelimiter\Bigparens{\Big(}{\Big)}
\DeclarePairedDelimiter\Biggparens{\Bigg(}{\Bigg)}
\DeclarePairedDelimiter\sqbr{[}{]}
\DeclarePairedDelimiter\bigsqbr{\big[}{\big]}

\DeclarePairedDelimiter\Bigsqbr{\Big[}{\Big]}

\DeclarePairedDelimiter\curlbr{\{}{\}}

\DeclareMathOperator{\Law}{Law}
\DeclareMathOperator{\bigO}{O}

\DeclareMathOperator{\smallo}{o}

\newcommand{\nat}{\mathbb{N}}
\newcommand{\borel}{\mathscr{B}}

\newcommand{\trace}{\mathrm{tr}}

\newcommand{\iid}{independent and identically distributed }

\newcommand{\e}{\mathrm{e}}

\newcommand{\norm}{\mathscr{N}}
\newcommand{\prob}{\mathbb{P}}
\newcommand{\expc}{\mathbb{E}}
\newcommand{\ind}{\text{\usefont{U}{bbold}{m}{n}1}}

\newcommand{\expo}{\operatorname{Exp}}

\newcommand{\pois}{\operatorname{Pois}}

\newcommand{\real}{\mathbb{R}}
\newcommand{\realn}[1]{\mathbb{R}^{#1}}

\newcommand{\sspa}{\text{ }}
\newcommand{\rational}{\mathbb{Q}}

\newcommand{\eqd}{\overset{\mathrm{d}}{=}}
\newcommand{\eqas}{\overset{\mathrm{a.s.}}{=}}

\newcommand*{\tran}{^{\mkern-1.5mu\mathsf{T}}}

\newcommand{\casinf}[1]{\xrightarrow[\sspa {#1}\to\infty \sspa]{\mathrm{a.s.}}}

\newcommand{\cpinf}[1]{\xrightarrow[\sspa {#1}\to\infty \sspa]{\prob}}

\newcommand{\cdinf}[1]{\xrightarrow[\sspa {#1}\to\infty \sspa]{\mathrm{d}}}

\newcommand{\toinf}[1]{\xrightarrow{\sspa {#1}\to\infty \sspa}}

\renewcommand{\dfrac}{\displaystyle \frac}

\newcommand{\dsum}{\displaystyle\sum}

\newcommand{\dsqrt}{\displaystyle\sqrt}

\newcommand{\ols}[1]{\mskip.5\thinmuskip\overline{\mskip-.5\thinmuskip {#1} \mskip-.5\thinmuskip}\mskip.5\thinmuskip}
\newcommand{\olsi}[1]{\,\overline{\!{#1}}}
\makeatletter
\newcommand\closure[1]{
  \tctestifnum{\count@stringtoks{#1}>1}
  {\ols{#1}}
  {\olsi{#1}}
}
\long\def\count@stringtoks#1{\tc@earg\count@toks{\string#1}}
\long\def\count@toks#1{\the\numexpr-1\count@@toks#1.\tc@endcnt}
\long\def\count@@toks#1#2\tc@endcnt{+1\tc@ifempty{#2}{\relax}{\count@@toks#2\tc@endcnt}}
\def\tc@ifempty#1{\tc@testxifx{\expandafter\relax\detokenize{#1}\relax}}
\long\def\tc@earg#1#2{\expandafter#1\expandafter{#2}}
\long\def\tctestifnum#1{\tctestifcon{\ifnum#1\relax}}
\long\def\tctestifcon#1{#1\expandafter\tc@exfirst\else\expandafter\tc@exsecond\fi}
\long\def\tc@testxifx{\tc@earg\tctestifx}
\long\def\tctestifx#1{\tctestifcon{\ifx#1}}
\long\def\tc@exfirst#1#2{#1}
\long\def\tc@exsecond#1#2{#2}
\makeatother

\makeatletter

\title{Flocking with Multiple Types: Competition, Fluid Limits and Traveling Waves}

\begin{document}
\author{Sayan Banerjee and Andrew Nguyen}

\maketitle



\begin{abstract}
    We study a class of interacting particle systems on $\mathbb{R}$ with two types. Particles evolve by independent jumps sampled from a fixed distribution, with type-dependent jump rates $v_+$ and $v_-$, and stochastic type switching driven by non-local order-based interactions. The switching rates depend on the empirical distribution through the proportion of opposite-type particles located ahead, leading to a nonlinear and discontinuous dependence on the empirical measure outside the standard Lipschitz McKean-Vlasov framework. Our first main result is a law of large numbers for the empirical measure process: we prove convergence, along subsequences, to a deterministic measure-valued process characterized by a McKean-Vlasov equation. The proof combines tightness in Wasserstein space with a martingale characterization of limit points. A uniqueness argument based on a Kolmogorov-Smirnov-type distance adapted to the ordering structure yields convergence of the full empirical measure sequence and, in turn, propagation of chaos on finite time intervals. We then study the long-time behavior of the limiting dynamics. Because the system has persistent drift, invariant distributions do not arise; instead, we analyze traveling waves, corresponding to stationary profiles in a moving frame. For exponential jump distributions, the associated non-local integro-differential system admits a local description. In the regime $v_+>v_-=0$, this further reduces to a coupled system of non-linear ODEs, allowing a phase-plane analysis that yields a traveling wave as a heteroclinic orbit connecting two equilibria. We also identify the wave speed and mass partition, and derive tail asymptotics by spectral analysis of the linearized system.
    \vspace{0.15in}

    \vspace{0.15in}

\end{abstract}

\section{Introduction}
Understanding the emergence of collective motion and spatial structure in large interacting particle systems is a central theme across probability, statistical physics, and applied fields such as biology, social dynamics, and distributed systems. Classical models of flocking and alignment—such as Cucker--Smale dynamics \cite{cucker2007emergent}, interacting diffusions, and kinetic models (see references later)—typically assume a \emph{homogeneous population} with smooth mean-field interactions. In contrast, many real-world systems involve \emph{heterogeneous populations} with competing types, whose interactions depend on spatial ordering and non-local effects.

In this paper, we introduce and analyze a class of multi-type interacting particle systems on $\mathbb{R}$, where each particle carries a type $\sigma \in \{+,-\}$ and evolves through jump dynamics combined with stochastic type switching. The jump dynamics are independent across particles, with type-dependent rates $v_+$ and $v_-$, and jump distribution $J$. The key feature of the model is the \emph{non-local mean-field interaction through spatial ordering}: the rate at which a particle changes type depends on the proportion of particles of the opposite type located ahead of it. More precisely, by letting $\mu_n(t)$ denote the empirical measure of the $n$-particle system at time $t$ on the space $\mathbb{R} \times \{+,-\}$ (encoding particle spatial locations and types), the rate at which a particle of type $\pm$ located at $x$ changes to $\mp$ is given by
\begin{equation}\label{alpha_intro}
\alpha^\pm(x,\mu_n(t)) = \alpha^\pm \mp \varphi\big(\mu_n^{\mp}((x,\infty))\big),
\end{equation}
where $\varphi$ is a bounded and Lipschitz function, $\alpha^{\pm}$ are the `intrinsic' rates, and $\mu_n^{\pm}(A) := \mu_n(A \times \{\pm\})$ for $A \in \borel(\real)$. This interaction mechanism models the \emph{competition} between the two types. Thinking of $+$ as the `fitter' type with $v_+>v_-$ (although this assumption is not required in our mean-field analysis), having too many $-$ types ahead of a particle of type $+$ reduces its rate of `slowing down' (changing to $-$), and having too many $+$ types ahead of a $-$ type particle motivates it to switch to the fitter $+$ type at a higher rate. The goal of this article is to study the large-scale and long-time behavior of this system. We analyze the large-scale behavior through mean-field McKean--Vlasov type limits, and the long-time behavior by exhibiting traveling wave solutions to the limiting equation, which characterize persistent shapes associated with `flock formation.'

At a technical level, the above interaction introduces a \emph{nonlinear} and \emph{discontinuous dependence} on the empirical measure, since the map
\begin{align*}
    \nu \mapsto \int \ind_{\{z>x\}}  \nu(\di z)
\end{align*}
is \emph{not necessarily continuous} in the Wasserstein topology. As a result, our model lies outside the standard framework of Lipschitz McKean--Vlasov dynamics, requiring new techniques to analyze both the mean-field limit and long-time behavior.

\subsection{Related work}
Our model is closely related in spirit to the class of \emph{two-velocity Boltzmann-type models} studied in the literature on agent-based flocking, in particular the work of \cite{hongler2014local}. In that work, particles move along continuous trajectories with velocities that switch between $v_+$ and $v_-$ at rates that depend on the surrounding population through a \emph{non-local observation kernel} (along with additional barycentric interactions). In particular, when the observation kernel is supported on the entire particle configuration ahead of the given particle (and with no barycentric interaction), we obtain the \emph{follow-all-the-leaders} model where the velocity switching rates are given by $\alpha^{\pm}(\cdot,\cdot)$ as in \eqref{alpha_intro} with $\varphi$ being the identity map. This model is explicitly solvable owing to its reduction to a Ruijgrok–Wu type system (see \cite{ruijgrok1982completely}) that can be linearized via a logarithmic transform which, in particular, also gives rise to explicit traveling wave solutions. In comparison, our model is of a pure-jump type, which adds another level of non-locality to the system, that leads to more involved large-scale and long-time behavior. In particular, our model is not explicitly solvable, which makes the analysis of traveling wave solutions delicate.

There are several other works on follow-the-leader type systems, both in terms of mean-field and nearest-neighbor type interactions \cite{gazis1961nonlinear,ben2007toy,holden2017follow,ridder2018traveling,banerjee2026long}. Among the few mathematical works on \emph{pure-jump mean-field models for flocking}, we mention \cite{BRT14,BBI24}, whose approach outline is closely followed in proving the McKean--Vlasov limit for our system. A closely related (single-type) particle system was studied in \cite{greenberg1996asynchronous,stolyar2023particle,stolyar2023large,baryshnikov2025large}, motivated by applications in \emph{distributed parallel simulation}. In this model, each particle in the $n$-particle system jumps forward by a random amount, independently sampled from a common distribution $\theta$, with rate given by a non-increasing function of its \emph{quantile} in the empirical distribution of the system.

More generally, the study of particle systems with \emph{topological interactions}, where a particle's motion depends on its relative position in the system, is an active and rapidly growing area. This includes \emph{rank-based diffusions} \cite{palpit, sar, sartsa, DJO,banerjee2022domains}, particle systems motivated by \emph{evolutionary biology} \cite{bruder, durrem}, and \emph{Gaussian pursuit-evasion} models \cite{banerjee2016brownian}, to name a few.

\subsection{McKean--Vlasov Limit}
 Our first objective is to study the mean-field (fluid) limit of the measure-valued process $\mu_n(\cdot)$ as $n$ tends to infinity over compact time intervals $[0,T]$ for any $T > 0$. Formally, one expects convergence to a deterministic measure-valued process $\mu(\cdot)$ satisfying the McKean--Vlasov equation (MVE)
 \begin{equation}\label{eq:MVE_intro}
 \langle f, \mu(t)\rangle=\langle f, \mu(0)\rangle+\int_0^t \langle \LLL f(\cdot,\mu(s)), \mu(s)\rangle \di s,
 \end{equation}
 for all test functions $f$ in a suitable class, where we denote $\langle f, \nu\rangle := \int_{\mathbb{R} \times \{+,-\}} f(\mathbf{z})\nu(\di \mathbf{z})$, and
 \begin{align*}
    \LLL f(x, \pm, \nu) := v_\pm \expc_Z\bigsqbr{f(x+Z, \pm) - f(x, \pm)} + \alpha^\pm(x, \nu) \bigparens{f(x, \mp) - f(x, \pm)},
\end{align*}
where $\alpha^\pm(\cdot,\cdot)$ is defined in \eqref{alpha_intro} with $\mu_n(t)$ replaced by $\nu$.
McKean--Vlasov limits and associated propagation of chaos results are classical (see, e.g., Sznitman, Graham--Méléard, and subsequent developments) and we follow the broad approach in \cite{BRT14,BBI24}. However, the discontinuous interaction structure discussed above prevents the direct use of standard Wasserstein contraction arguments. 
We outline the main framework used here:
\begin{itemize}[leftmargin=*]
\item \emph{Tightness:} Wasserstein tightness of the measure-valued process $\mu_n(\cdot)$ is established via equicontinuity estimates, implying existence of subsequential limits;
\item \emph{Characterizing subsequential limits:} Subsequential limits are identified as solutions to the MVE \eqref{eq:MVE_intro}, which requires \emph{`a priori' continuity estimates} for $\mu_n(\cdot)$ (see Lemma \ref{uniformLemma}) to circumvent the discontinuity described above;
\item \emph{Uniqueness:} Uniqueness of solutions to the MVE from a given initial condition is established via a Kolmogorov--Smirnov type distance adapted to the ordering structure. This `metric change' from Wasserstein to Kolmogorov--Smirnov is again due to the discontinuity which makes the associated Gr\"onwall-type estimates hard to obtain purely in Wasserstein distance.
\end{itemize}
Theorem \ref{fluidTheorem} records our main result in this vein, while
Corollary \ref{POC} gives a \emph{propagation-of-chaos} result, which (roughly) states that the laws of the particle trajectories started from independent and identically distributed (i.i.d.) initial locations remain close to i.i.d. for large $n$ over compact time intervals.

\subsection{Traveling Waves and Long-Time Behavior}
Our next goal is to study the long-time behavior of solutions $\mu(t)$ to \eqref{eq:MVE_intro}, the MVE. To make the analysis somewhat simpler, we assume throughout that $\varphi$ is the identity map, although our techniques have natural extensions to the more general case. As the whole system drifts rightward, $\mu(t)$ does not converge to a stationary distribution. Instead, we study the emergence of \emph{traveling wave solutions}, corresponding to stationary profiles in a moving frame. A traveling wave is a pair of cumulative distribution functions $H=(H_+,H_-)$ and a speed $\gamma>0$ such that\[\mu^\pm(t,(-\infty,x]) = \rho^\pm H_\pm(x - \gamma t),\]for some mass partition $(\rho^+,\rho^-)$ with $\rho^+ + \rho^- = 1$. Substituting into \eqref{eq:MVE_intro}, one obtains a system of coupled non-local integro-differential equations for $H_\pm$ (see Lemma \ref{integrodiff}). The analysis of these equations is challenging due to  their non-local structure, the coupling between types, and the absence of a priori monotonicity or convexity in the dynamics. This places our system outside the realm of standard tools for obtaining traveling wave solutions (e.g., those used for Fisher--KPP equations \cite{an1937etude,baryshnikov2025large}), making its long-time analysis highly involved.

To make the problem more tractable, we restrict attention to exponential jump distributions. This choice leads to certain structural simplifications that allow reduction to a local dynamical system of coupled second-order ordinary differential equations (ODEs). We then carry out a detailed phase-space analysis of this system, including the identification of its equilibria and the study of the spectra and eigenspaces of the corresponding linearizations. The main steps of the analysis are outlined below:
\begin{itemize}
    \item \emph{Characterizing mass-partition and wave speed:} For an exponential jump distribution, traveling wave solutions, if they exist, admit explicit formulas for the mass partition $(\rho^+,\rho^-)$ and wave speed $\gamma$ (Theorem \ref{travelwave}(a)).
    \item \emph{Dynamical systems reduction:} When $v_+> v_-=0$, the traveling wave equations further simplify to a coupled pair of non-linear first order ODEs. This enables phase-space analysis and identification of a heteroclinic orbit connecting equilibria, furnishing existence and characterization of traveling waves (Theorem \ref{travelwave}\ref{b0}).
    \item \emph{Tail asymptotics:} The left and right tail exponents of the traveling wave are explicitly computed using spectral analysis of the associated linearized system near the two equilibria (Theorem \ref{travelwave}\ref{b0}\ref{b2}--\ref{b0}\ref{b3}).
\end{itemize}

Here, we would like to note that, beyond the specific choices of exponential jump distributions and $v_-=0$ made above, the traveling wave analysis holds for \emph{all values} of the parameters $\alpha^{\pm}$ and $v_+>0$, and the wave tail exponents exhibit a rich complex algebraic dependence on these parameters.

\textbf{Qualitative phenomena.} Our analysis reveals several non-trivial features of the traveling wave, some of which we highlight below.
\begin{itemize}
\item Faster particles (of type $+$) accumulate toward the \emph{back} of the wave, driving the `flock' forward, despite being the fitter species (see Theorem \ref{travelwave}\ref{b0}\ref{b1}).
\item The wave speed depends nonlinearly on the competition parameters ($\alpha^+$ and $\alpha^-$), and linearly on the type-specific speeds ($v_+$ and $v_-$), in an explicit way.
\item Both types share the same wave speed and identical tail exponents, indicating strong mixing and absence of spatial segregation.
\end{itemize}

Although the wave speed and tail exponents are algebraically involved, some simplifications are achieved when $\alpha^+ \bg \alpha^-$ or $\alpha^- \bg \alpha^+$.
Corollary \ref{tailclean} and Remark \ref{twalpharem} describe these results and their qualitative implications. Figure \ref{fig:TWplots} provides visualizations of some traveling waves.

\medskip
Overall, our work provides a new class of \emph{multi-type} interacting particle systems with \emph{non-smooth mean-field interactions}, for which both the fluid limit and long-time behavior can be rigorously characterized. The combination of discontinuous interactions, multi-type competition, and traveling wave phenomena places this model at the intersection of several active areas in probability and statistical physics.

\subsection{Future work}

In the present article, we obtained the MVE as a mean-field limit of the particle empirical measure over compact time intervals, and then studied the long-time behavior of the limiting MVE as a proxy for the long-time behavior of the particle system for finite large $n$. To formally make this connection, the next step in our program will be to obtain \emph{interchange of limits, and long-time and uniform-in-time propagation of chaos} results for the particle system. This will help control the discrepancy between the pre-limiting empirical measure $\mu_n(t)$ and the McKean-Vlasov limit $\mu(t)$ over larger time horizons (possibly depending on $n$), and possibly over all time, to make rigorous connections between the long-time behavior of $\mu_n(t)$ and $\mu(t)$ for large $n$. We will further study generalizations of our system with more than two (possibly infinite) types, and with jump distributions having heavier tails.

In future work, we also aim to extend our traveling wave analysis to more general jump distributions and to the case $v_- \neq 0$. This extension is expected to require new ideas and to reveal qualitatively new phenomena.

\section{Problem Setup}
\label{Intro}
Consider a two-type population of $n$ particles, where each particle is of either type $+$ or type $-$. Particles of type $+$ (respectively, $-$) jump at rate $v_+$ (respectively, $v_-$), where $v_+$ and $v_-$ are non-negative real numbers, with jumps independent and identically distributed, sampled from a distribution $J$. Particles are ``tagged'' by their initial locations. Denote the trajectory/path of the $i^\text{th}$ particle using $\{x^n_i(t)\!: t\!\geq\!0\}$, and the type of the same particle using $\{\sigma^n_i(t)\!: t\geq 0\}$. For $t\geq 0$, let
  \begin{align*}
    \xx^n(t) := \begin{bmatrix} x^n_1(t) & x^n_2(t) & \dots & x^n_n(t) \end{bmatrix}\tran & & \text{and} & & \pmb{\sigma}^n(t) := {\begin{bmatrix} \sigma^n_1(t) & \sigma^n_2(t) & \dots & \sigma^n_n(t) \end{bmatrix}}\tran
\end{align*}
denote the vectors containing the locations (and types, respectively) of all particles at time $t$. Lastly, for each $t\geq 0$, let $\XX^n_i(t) := \begin{bmatrix} x^n_i(t) & \sigma^n_i(t) \end{bmatrix}$, and let
\begin{align*}
    \XX^n(t) := \begin{bmatrix} \xx^n(t) & \pmb{\sigma}^n(t)  \end{bmatrix}
\end{align*}
capture the ``state'' of the whole system at time $t$. For shorthand, define $[n] := \{1, 2, \dots, n\}$.

The type of particle $i$ may change to the opposite type at a rate, given as follows:
\begin{itemize}
    \item  To specify the flip rates at configurations containing spatial ties, fix the following deterministic tie-breaking convention. For any $\xx\in\realn{n}$ and $(i, j)\in [n]^2$, write $j \succ_\xx i$ if and only if either $x_j > x_i$, or $x_j = x_i$ and $j > i$. That is,
    \begin{align*}
        \ind_{\{j \succ_\xx i\}} = \ind_{\{x_j > x_i\}} + \ind_{\{x_j=x_i\}} \ind_{\{j>i\}}.
    \end{align*}
    For each $n\in\nat$, constants $\alpha^+ > 0$, $\alpha^- > 0$, $\xx\in\real^n$, $\pmb{\sigma}\in\{+,-\}^n$, and a propensity function $\varphi\!: [0, 1]\to\real^+$ (all of which will be specified later), define
    \begin{align*}
        \alpha^{\sigma_i = +}_n(\xx,\pmb{\sigma}) := \alpha^+ - \varphi\parens*{\dfrac{1}{n} \dsum_{j=1}^n \ind_{\{j \succ_\xx i\}} \ind_{\{\sigma_j = -\}}}
    \end{align*}
    and
    \begin{align*}
        \alpha^{\sigma_i = -}_n(\xx,\pmb{\sigma}) := \alpha^- + \varphi\parens*{\dfrac{1}{n} \dsum_{j=1}^n \ind_{\{j \succ_\xx i\}} \ind_{\{\sigma_j = +\}}}.
    \end{align*}

    \item  If $\sigma^n_i(t^-) = +$ (that is, if just before time $t$, particle $i$ is of type $+$), then the rate of change to type $-$ at time $t$ is given by $\alpha^{\sigma_i = +}_{n}\big(\xx^n(t^-), \pmb{\sigma}^n(t^-)\big)$. Conversely, if $\sigma^n_i(t^-) = -$, then the corresponding rate of type change is $\alpha^{\sigma_i = -}_{n}\big(\xx^n(t^-), \pmb{\sigma}^n(t^-)\big)$.
\end{itemize}
We assume $\varphi\!: [0, 1]\to\real^+$ is bounded and $1$-Lipschitz; that is, $\lnorm*{\varphi}_\infty := \sup_{x\in [0, 1]}\varphi(x) < \infty$, and for any $x\in[0, 1]$ and $y\in [0, 1]$, $\abs*{\varphi(x) - \varphi(y)} \leq |x-y|$. Also, so that $\alpha^+_{n}$ makes sense, we necessarily have $\lnorm*{\varphi}_\infty \leq \alpha^+$. In addition, in order to facilitate our discussions about martingales and Markov processes later on, let us define the ``natural'' filtration $\curlbr*{\FFF(t)\!: t\geq 0}$, where, for each $t\geq 0$, $\FFF(t) := \sigma\curlbr*{\XX^n(s)\!: s\in [0, t]}$.

For $t\geq 0$, we capture the particle locations and types at time $t$ by the empirical measure
\begin{align*}
    \mu_n(t) := \dfrac{1}{n} \dsum_{j=1}^n \delta_{\XX^n_j(t)},
\end{align*}
and encode the particle locations for a given type by
\begin{align*}
    \mu^\pm_n(t) := \dfrac{1}{n} \dsum_{j=1}^n \delta_{\XX^n_j(t)} \ind_{\curlbr*{\sigma^n_j(t) = \pm}}.
\end{align*}
Note that $\mu^+_n(t) + \mu^-_n(t) = \mu_n(t)$; in addition, $\mu^+_n(\cdot)$ and $\mu^-_n(\cdot)$ are RCLL processes, and $\mu_n(\cdot) \in \DDD\big([0, \infty), \PPP_1(\real \times \{+,-\})\big)$. For convenience, for each $n\in\nat$ and $A\in\borel(\real)$, we also define the spatial marginals of $\mu_n$ and $\mu^\pm_n$ by
\begin{align*}
    \mu^\circ_n(t, A) := \mu_n\bigparens{t, A\times \{+,-\}} & & \text{and} & & \mu^\pm_n(t, A) := \mu_n\bigparens{t, A\times \{\pm\}},
\end{align*}
respectively.

Next, let us consider the \textbf{generator} of this process. For any (suitable) function $\ff:\real^n \times \curlbr*{+, -}^n\rightarrow\real$, define
\begin{align*}
    \pmb{\LLL} \ff(\xx, \pmb{\sigma}) := \dsum_{i=1}^n v_{\sigma_i} \expc_Z\bigsqbr{\ff(\xx + \ee_i Z, \pmb{\sigma}) - \ff(\xx, \pmb{\sigma})} + \dsum_{i=1}^n \alpha^{\sigma_i}_n(\xx, \pmb{\sigma}) \Bigparens{\ff\big(\xx, t_i(\pmb{\sigma})\big) - \ff(\xx, \pmb{\sigma})},
\end{align*}
where  $\ee_i$ is the $i^\text{th}$ vector in the standard basis of $\real^n$,  and for $\pmb{\sigma} \in \{+,-\}^n$, $t_i$ ``flips'' the type of the $i^\text{th}$ coordinate of $\pmb{\sigma}$; that is,
\begin{align*}
    t_i(\pmb{\sigma}) =
    \setlength{\nulldelimiterspace}{0pt}
    \left\{
    \begin{aligned}
        &\parens*{\sigma_1,\dots,\sigma_{i-1},-,\sigma_{i+1},\dots,\sigma_n}, && \text{if }\sigma_i=+ \\
        &\parens*{\sigma_1,\dots,\sigma_{i-1},+,\sigma_{i+1},\dots,\sigma_n}, && \text{if }\sigma_i=-
    \end{aligned}
    \right..
\end{align*}
Similarly, we define $t\!: \{+,-\} \to \{+,-\}$ as
\begin{align*}
    t(\sigma) = \setlength{\nulldelimiterspace}{0pt}
    \left\{
    \begin{aligned}
        &-, && \text{if }\sigma = + \\
        &+, && \text{if }\sigma = -
    \end{aligned}
    \right..
\end{align*}
For reasons related to the $1$-Wasserstein metric, we choose $\HHH$ to be the collection of $1$-Lipschitz functions from $\real\times \{+,-\}$ to $\real$, and we say that $f\!: \real\times \{+,-\} \to \real$ is $1$-Lipschitz if, for any
$(x_1, \sigma_1) \in  \real\times \{+,-\}$ and $(x_2, \sigma_2) \in  \real\times \{+,-\}$,
\begin{align*}
    \abs*{f(x_1, \sigma_1) - f(x_2, \sigma_2)} \leq |x_1 - x_2| + \ind_{\{\sigma_1 \neq \sigma_2\}}.
\end{align*}
Now, let us define a ``projected'' version of the generator. For $f\in \HHH$, define
\begin{align*}
    \LLL f\!: \real\times \{+,-\} \times \PPP_1(\real \times \{+,-\}) \to \real
\end{align*}
by
\begin{align*}
    \LLL f(x, \pm, \nu) := v_\pm \expc_Z\bigsqbr{f(x+Z, \pm) - f(x, \pm)} + \alpha^\pm(x, \nu) \bigparens{f(x, \mp) - f(x, \pm)}.
\end{align*}
 We will investigate conditions under which
\begin{align*}
    \mu_n(\cdot) \cdinf{n} \mu(\cdot)
\end{align*}
in $\DDD\big([0, \infty), \PPP_1(\real\times \{+,-\})\big)$ (possibly along a subsequence), where $\mu(\cdot)$ is a solution to the \emph{McKean--Vlasov equation} (MVE)
\begin{align}
\label{originalMVE}
    A_{t,f}\bigparens{\mu(\cdot)} = 0,
\end{align}
with initial condition $\nu$, and for $t \geq 0$ and $f\in\HHH$, $A_{t,f}\bigparens{\mu(\cdot)}$ is defined to be
\begin{align*}
    A_{t,f}\bigparens{\mu(\cdot)} &:= \biginner{f, \mu(t)} - \biginner{f, \mu(0)} - \int_0^t \int_{\real\times \{+,-\}}\LLL f\bigparens{\mathbf{x}, \mu(s)} \mu(s,\dee \mathbf{x}) \di s \\
    &= \biginner{f, \mu(t)} - \biginner{f, \mu(0)} - \int_0^t \Biginner{\LLL f\bigparens{\cdot, \mu(s)}, \mu(s)} \di s.
\end{align*}
We will also investigate \emph{traveling wave solutions} to \eqref{originalMVE}, which are essentially stationary solutions in a moving reference frame, to quantify stable shapes of the flock formed by the moving particles over time. The functional $\alpha^{\pm}(\cdot,\cdot)$, where $\alpha^\pm(x, \nu) := \alpha^\pm \mp \varphi\parens*{\int \ind_{\{z>x\}} \nu^{\mp}(\di \zz)}$, and $\nu^{\pm}(A) := \nu\bigparens{A\times\{\pm\}}$ for $A \in \borel(\real)$, makes the analysis challenging by introducing non-linearity in the measure-valued dynamics and possible  discontinuities as the map $\nu \mapsto \int\limits_{\{z>x\}}\nu^{\mp}(\di \zz)$, for any given $x$, is not necessarily continuous on $\PPP_1(\real \times \{+,-\})$.  Also, for convenience,  we will abbreviate
  \begin{align}
  \label{alphandef}
    \alpha^\pm_n(t, x) := \alpha^\pm\bigparens{x, \mu_n(t)} = \alpha^\pm \mp \varphi\parens*{\dfrac{1}{n} \dsum_{j=1}^n \ind_{\curlbr*{x^n_j(t) > x}} \ind_{\{\sigma^n_j(t)=\mp\}}}.
\end{align}
 \section{Main Results}
\label{sec2}
\setcounter{definition}{-1}
\subsection{The McKean--Vlasov Fluid Limit}
Before stating our result, we shall specify some standing conditions for the particle system. First, regarding the distribution of the jump steps, let us make the following assumption:
\begin{assumption}
    \label{assumption0}
    For $Z\sim J$, $\expc(Z) = 1$, and $\expc\parens*{Z\displaystyle^2} < \infty$.
\end{assumption}
Next, we make an assumption regarding the initial profiles of the particles:
\begin{assumption}
    \label{assumption2}
    For any $\eta > 0$,
      \begin{align}
        \lim_{L\to\infty} \limsup_{n\to\infty} \prob\parens*{\int \abs*{x} \ind_{\curlbr*{|x| \geq L}} \mu^\circ_n(0, \di x) \geq \eta} = \lim_{L\to\infty} \limsup_{n\to\infty} \prob\parens*{\dfrac{1}{n}\dsum_{i=1}^n |x^n_i(0)| \ind_{\curlbr*{|x^n_i(0)| \geq L}} \geq \eta} = 0.
    \end{align}
\end{assumption}
\begin{remark}
    The following is a helpful consequence of Assumption \ref{assumption2}:
      \begin{align}
    \label{moreTightAssumption}
        \lim_{L\to\infty} \limsup_{n\to\infty} \prob\parens*{\dfrac{1}{n} \dsum_{i=1}^n \abs*{x^n_i(0)} > L} = 0.
    \end{align}
\end{remark}
In addition, we make two other assumptions, one about the initial profile of the particles, and the other about the distribution of the jump steps:
 \begin{assumption}
    \label{assumption3}
      Suppose that for every $\varepsilon > 0$,
    \begin{align*}
        \lim_{\eta\to 0} \limsup_{n\to\infty} \prob\parens*{\sup_{x\in\real} \mu^\circ_n\bigparens{0, (x-\eta, x+\eta)} > \varepsilon} = 0.
    \end{align*}
    Furthermore, suppose that the initial spatial configuration is almost-surely tie-free; that is,
    \begin{align*}
        \prob\bigparens{x^n_i(0) = x^n_j(0) \text{ for some } i\neq j} = 0.
    \end{align*}
 \end{assumption}

\begin{remark}
In particular, the first statement of Assumption \ref{assumption3} holds if the initial locations $\{x_i^n(0)\!: i\in[n]\}$ are independent, with cumulative distribution functions $F_{n,i}$, and the family $\curlbr*{F_{n,i}:n\in\nat,\ i\in[n]}$ is uniformly equicontinuous, in the sense that 
\begin{align*}
    \lim_{\eta\downarrow0} \sup_{n\in\nat}\max_{i\in[n]}\sup_{x\in\real} \curlbr*{F_{n,i}(x+\eta) - F_{n,i}(x-\eta)} = 0.
\end{align*}
The second statement holds, in particular, if, for each $n\in\nat$, $\{x_i^n(0):i\in[n]\}$ are \iid with a common continuous cumulative distribution function $F_n$, and the sequence $\{F_n\!: n\in\nat\}$ converges weakly to a continuous cumulative distribution function $F$.
\end{remark}
\begin{assumption}
\label{assumption4}
    The jump length distribution $J$ admits a bounded density $\phi$ on $[0, \infty)$.
\end{assumption}
  Next, define the relevant class of solutions for the McKean--Vlasov equation \eqref{originalMVE} by
\begin{align*}
    \TTT &:= \curlbr*{\mu \in \CCC\bigparens{[0, \infty), \PPP_1(\real \times \{+,-\})}\!: \text{for every } T > 0, \sspa \lim_{\eta\to 0}\sup_{s\in[0,T]} \sup_{x\in\real} \mu^\circ\bigparens{s, (x-\eta, x+\eta)} = 0}.
\end{align*}
Now we state our first main result.
 \begin{theorem}
    \label{fluidTheorem}
    We have the following:
    \begin{enumerate}[label = (\alph*)]
        \item \label{31a} Suppose that Assumptions \ref{assumption0} and \ref{assumption2} hold. Then the sequence $\{\mu_n(\cdot)\!: n\in\nat\}$ is $\mathscr{C}$-tight in the space $\mathscr{D}\big([0, \infty), \mathscr{P}_1(\real\times \{+,-\})\big)$.

        \item \label{31b} In addition, suppose that Assumptions \ref{assumption3} and \ref{assumption4} hold. Then any subsequential weak limit point $\mu(\cdot)$ of $\{\mu_n(\cdot)\!: n\in\nat\}$ in $\DDD\big([0, \infty), \PPP_1(\real \times \{+, -\})\big)$ satisfies \eqref{originalMVE}, the McKean-Vlasov equation, for $t \ge 0$ and $f \in \HHH$. Moreover, almost surely,  $\mu \in \TTT$.

        \item \label{31c} Furthermore, if $\mu_n(0) \cpinf{n} \nu$ in $\PPP_1\bigparens{\real \times \{+, -\}}$, where $\nu$ is a deterministic measure, then $\mu_n(\cdot) \cpinf{n} \mu(\cdot)$ in $\DDD\big([0, \infty), \PPP_1(\real \times \{+, -\})\big)$, where $\mu(\cdot)$ is the unique and deterministic solution to \eqref{originalMVE} in $\TTT$ with initial condition $\mu(0) = \nu$.
    \end{enumerate}
\end{theorem}

The following corollary establishes \emph{propagation of chaos} for finite times $t \geq 0$ when the particles are initiated at i.i.d. locations:
\begin{corollary}[Propagation of Chaos]
\label{POC}
    Suppose Assumptions \ref{assumption0} and \ref{assumption4} hold; in addition, suppose that $\curlbr*{\XX^n_i(0)\!: i\in[n]}$ are independent and identically distributed with some deterministic distribution $\Gamma\in\PPP_1(\real\times\{+,-\})$, whose spatial marginal $\Gamma^\circ$, given by $\Gamma^\circ(A) := \Gamma\bigl(A\times\{+,-\}\bigr)$ for $A\in\borel(\mathbb R)$, has a continuous cumulative distribution function.
    \begin{enumerate}[label = (\alph*)]
        \item There exists a unique $\mu\in\TTT$ such that it solves \eqref{originalMVE} with $\mu(0)=\Gamma$; in addition,   $\mu_n(\cdot)\cpinf{n}\mu(\cdot)$ in $\DDD\bigparens{[0,\infty),\PPP_1(\real\times\{+,-\})}$.

         \item  Furthermore, for any $t\geq0$, the sequence $\curlbr*{\Law\bigparens{\XX^n(t)}\!: n\in\nat}$ is $\mu(t)$-chaotic; that is, for any  $k\in\nat$ and any bounded   continuous function $\ff\!:(\real\times\{+,-\})^k\to\real$,
        \begin{align*}
            &\expc\sqbr*{\ff\bigparens{\XX^n_1(t),\XX^n_2(t),\dots,\XX^n_k(t)}} \\
            &\qquad\toinf{n}\int\limits_{(\real\times\{+,-\})^k}\ff(\xx_1,\xx_2,\dots,\xx_k)\mu(t,\di\xx_1)\mu(t,\di\xx_2)\dots\mu(t,\di\xx_k).
        \end{align*}
        Equivalently,   $\Law\bigparens{\XX^n_1(t),\XX^n_2(t),\dots,\XX^n_k(t)}\cdinf{n}\mu^{\otimes k}(t)$.
    \end{enumerate}
 \end{corollary}
 \subsubsection{Proof Outline and Challenges}
We take the following approach which, at a high level, aligns with that in \cite{BRT14} and \cite{BBI24}. However, several challenges arise, as we shall indicate (briefly) below:
\begin{itemize}
    \item First, under Assumptions \ref{assumption0} and \ref{assumption2}, we will show that $\curlbr*{\mu_n(\cdot)\!: n\in\nat}$ is $\mathscr{C}$-tight in the space $\mathscr{D}\big([0, \infty), \mathscr{P}_1(\real\times \{+,-\})\big)$, which establishes weak convergence of our measure-valued process $\mu_n(\cdot)$ along subsequences. This involves an ``equicontinuity estimate'' in $1$-Wasserstein distance (see Theorem \ref{firstBigThm}) and proving `marginal tightness' (see Theorem \ref{tight2}).

    \item Next, under relevant assumptions, we show that for each $n\in\nat$ and each $f\in\HHH$, the family $\curlbr*{A_{t,f}\big(\mu_n(\cdot)\big)\!: t\geq 0}$ is a martingale that converges (weakly) to zero. In addition, since $\curlbr*{\mu_n(\cdot)\!: n\in\nat}$ is $\CCC$-tight, there exists a subsequence $\curlbr*{\mu_{n_k}(\cdot)\!: k\in\nat}$ that converges weakly to a random measure-valued path $\mu(\cdot)$ in $\mathscr{C}\big([0, \infty), \mathscr{P}_1(\real\times \{+,-\})\big)$. Our next target is to show that
    \begin{align*}
        A_{t,f}\bigparens{\mu_{n_k}(\cdot)} \cdinf{k} A_{t,f}\bigparens{\mu(\cdot)}
    \end{align*}
    for each $t \geq 0$ and $f\in\HHH$. This requires a more careful treatment as the non-linearity and possible discontinuity introduced by $\alpha_n^{\pm}(\cdot,\cdot)$ defined in \eqref{alphandef} require establishing \emph{`a priori' continuity estimates} for $\mu_n(\cdot)$. This is carried out in Lemma \ref{uniformLemma}, and it also shows that subsequential limits have to lie in  $\TTT$. These steps allow us to characterize the behavior of the subsequential limiting measure via the McKean--Vlasov-type fluid limit equation \eqref{originalMVE}.

    \item Lastly, we prove that the solution to the MVE above is unique in the class $\TTT$ under appropriate assumptions. This also turns out to be quite delicate, since attempting to establish a Gr\"onwall type estimate for two solutions of the MVE in terms of the natural $1$-Wasserstein metric seems unfruitful. The reason for this, again, comes from the influence of type change, which makes bounding the rate of change of $1$-Wasserstein distance with respect to itself (the key step in the Gr\"onwall estimate) very challenging. To circumvent this, we obtain a Gr\"onwall estimate in terms of a \emph{Kolmogorov--Smirnov type distance}, which is well-suited to our dynamics. This uniqueness result implies that the limiting measures of all subsequences are the same, and thus the whole sequence converges weakly to the (unique) measure that solves the MVE.
\end{itemize}
 \subsection{Traveling Wave Solutions of the MVE}

Throughout this section, we assume that the function $\varphi$ governing the rate of type change is the identity map. Accordingly, we require $\alpha^+ \ge 1$. Recall that the fluid limit MVE is given by
\begin{align}
\label{originalFluid}
    \inner*{f, \mu(t)} &= \inner*{f, \mu(0)} + \int_0^t \sqbr*{\inner*{v_+ g_f(\cdot, +), \mu^+(s)} + \inner*{v_- g_f(\cdot, -), \mu^-(s)}} \di s \nonumber \\
    &\qquad +\int_0^t\!\Bigsqbr{\biginner{\alpha^+(s, \cdot) h_f(\cdot, -), \mu^+(s)} + \biginner{\alpha^-(s, \cdot) h_f(\cdot, +), \mu^-(s)}} \di s,
\end{align}
for any $f\in \HHH$. As the whole system moves to the right when at least one among $v_+$ and $v_-$ is positive, there is no stable long-time limit of measure-valued solutions to the above equation. However, our interest is in the stabilization of the \emph{shape of the distribution in a moving reference frame}, a so-called \textbf{traveling wave}. First, we formally define a traveling wave:
\begin{definition}[Traveling Wave]
\label{traveldefn}
    We call a $\CCC^2$ function $H = \parens*{H_+, H_-}\!: \real \to [0, 1]^2$ a \textbf{traveling wave} with mass partition $(\rho^+, \rho^-)$, for some $\rho^\pm \in [0,1]$ with $\rho^+ + \rho^- = 1$, and speed $\gamma > 0$, if
    \begin{enumerate}[label = \arabic*.]
        \item $H$ is coordinate-wise non-decreasing;

        \item $\lim_{t\to-\infty} H(t) = (0, 0)$ and $\lim_{t\to\infty} H(t) = (1, 1)$; and

        \item the measure-valued process $\curlbr*{\mu^\pm(s)\!: s\geq 0}$, induced by the (extended) cumulative distribution function processes $\curlbr*{\rho^\pm H_\pm(\cdot - \gamma s)\!: s\geq 0}$, solves the mean-field equation \eqref{originalFluid}.
    \end{enumerate}
\end{definition}
For any $z\in\real$, let us denote $F_{\pm}(z) := \rho^{\pm} H_{\pm}(z)$.

The following is the main result of this section, which addresses the existence of traveling wave solutions to \eqref{originalFluid}, as well as specifies some features of such solutions under suitable assumptions.
\begin{theorem}
\label{travelwave}
    Suppose the jump steps are independent $\expo(1)$ random variables; that is, $Z \sim \expo(1)$.
    \begin{enumerate}[label = (\alph*)]
        \item \label{a0} For a traveling wave solution $H = \parens*{H_+, H_-}$ for \eqref{originalFluid} with mass partition $(\rho^+, \rho^-)$ and speed $\gamma$ to exist, we must necessarily have
        \begin{align*}
            \rho^+ &= \dfrac{1 - \alpha^+ - \alpha^- + \dsqrt{\parens*{1 - \alpha\displaystyle^+ - \alpha\displaystyle^-}^2 + 4\alpha^-}}{2}, \\
            \rho^- &= \dfrac{1 + \alpha^+ + \alpha^- - \dsqrt{\parens*{1 - \alpha\displaystyle^+ - \alpha\displaystyle^-}^2 + 4\alpha^-}}{2}, \\
            \intertext{and}
            \gamma &= v_+\rho^+ + v_-\rho^- \\
            &= \dfrac{\parens*{v_+ + v_-} + \parens*{v_+ - v_-} \parens*{\dsqrt{\parens*{1 - \alpha\displaystyle^+ - \alpha\displaystyle^-}^2 + 4\alpha\displaystyle^-}-\alpha\displaystyle^+ - \alpha\displaystyle^-}}{2}.
        \end{align*}
        In particular, $\rho^+ = \rho^- = \dfrac{1}{2}$ if and only if $\alpha^+ - \alpha^- = \dfrac{1}{2}$.

        \item \label{b0} Suppose $v_+ > v_- = 0$. Then a traveling wave solution $H = (H_+, H_-)$ exists; furthermore, it satisfies the following:
        \begin{enumerate}[label = (\arabic*)]
            \item \label{b1} For each $t\in\real$, $H_+(t) \geq H_-(t)$;

            \item \label{b2} The ``asymptotic left-tail exponent'' for the traveling wave is given by
              \begin{align*}
                \lim_{t\to-\infty} \dfrac{\log H_\pm(t)}{t}&= \dfrac{v_+-\gamma+\alpha^+ +\alpha^- +\rho^+-\rho^-}{2\gamma} \\
                &\quad -\dfrac{1}{2} \sqrt{\parens*{\!\dfrac{v_+ - \gamma +\alpha^+ - \alpha^- -1}{\gamma}\!}^2 + \dfrac{4\parens*{\alpha\displaystyle^+ -\rho\displaystyle^-} \parens*{\rho\displaystyle^+ +\alpha\displaystyle^-+\gamma}}{\gamma\displaystyle^2}} > 0,
            \end{align*}
            where the quantity in the square root is necessarily positive.

            \item \label{b3} The ``asymptotic right-tail exponent'' is given by
              \begin{align*}
                \lim_{t\to\infty} \sqbr*{-\dfrac{\log\bigparens{1 - H_\pm(t)}}{t}} &= -\dfrac{v_+ - \gamma +\alpha^+ + \alpha^-}{2\gamma} \\
                &\qquad + \dfrac{1}{2}\sqrt{\parens*{\dfrac{v_+ - \gamma + \alpha^+ - \alpha^-}{\gamma}}^2 + \dfrac{4\alpha\displaystyle^+ \parens*{\alpha\displaystyle^- + \gamma}}{\gamma\displaystyle^2}} > 0.
            \end{align*}
        \end{enumerate}
    \end{enumerate}
\end{theorem}

\begin{remark}
We note the following:
\begin{enumerate}[label = (\roman*)]
    \item Part \ref{b0}\ref{b1} of Theorem \ref{travelwave} leads to the surprising observation that, although particles of type $+$ are faster, the proportion of type $+$ particles behind any given point is at least as large as that for type $-$ particles, under the traveling wave approximation.

    \item Suppose $v_+>v_-$. For a fixed `net rate of change' $s = \alpha^+ + \alpha^-$, the formulas in Theorem \ref{travelwave}\ref{a0} imply that $\rho^+$, and consequently, the wave speed $\gamma$, are decreasing functions of the rate difference $d = \alpha^+ - \alpha^-$. Thus, by increasing $\alpha^-$, the `intrinsic' rate of change of $-$ to $+$ (or by the same token, decreasing $\alpha^+$), the proportion of $+$ (faster individuals) is also increased, and thus the speed of the traveling wave is also increased. This observation, although natural, is not entirely obvious, as the net rate of change of $-$ to $+$ at time $t$ and location $z + \gamma t$, given by
    \begin{align*}
        \beta^-(z) = \alpha^-(t, z+ \gamma t) = \alpha^- + \rho^+ - F_+(z),
    \end{align*}
    depends on $\alpha^+$ and $\alpha^-$ in a more implicit way.

    By Theorem \ref{travelwave}\ref{b0}\ref{b1}, these $+$ individuals aggregate more towards the back of the wave and push the whole `flock' into moving faster.

    \item Observe that, regardless of the choice of parameters, the wave speed and tail exponents agree between the two types. Thus, the flocking mechanism is strong enough to prevent segregation of types.
\end{enumerate}
\end{remark}

The following corollary quantifies traveling wave properties via Theorem \ref{travelwave} in asymptotic regimes where $\rho^+$ approaches $0$ and $\rho^+$ approaches $1$:

\begin{corollary}
\label{tailclean}
    Suppose that $v_+ > v_- = 0$.
    \begin{enumerate}[label = (\alph*)]
        \item Suppose that $s = \alpha^+ + \alpha^->1$ is fixed. We have
        \begin{align*}
            \lim_{\alpha^- \downarrow 0} \frac{\rho^+}{\alpha^-} = \frac{1}{s-1} & & \text{and} & & \lim_{\alpha^- \downarrow 0} \frac{\gamma}{\alpha^-} = \frac{v_+}{s-1};
        \end{align*}
        furthermore, the traveling wave tail exponents satisfy
        \begin{align*}
            \lim_{\alpha^- \downarrow 0} \lambda_R = \frac{1}{v_+ + s} & & \text{and} & & \lim_{\alpha^- \downarrow 0} \lambda_L = \frac{1}{v_+ + s-1}.
        \end{align*}
        This is referred to as the ``small-$\alpha^-$'' regime.

        \item Suppose that $\alpha^- \uparrow \infty$ and $\alpha^+ = \smallo(\alpha^-)$. Then
        \begin{align*}
            \rho^- = \frac{\alpha^+}{\alpha^-}\bigparens{1 + \smallo(1)} & & \text{and} & & \gamma = v_+ - \frac{v_+\alpha^+}{\alpha^-}\bigparens{1 + \smallo(1)},
        \end{align*}
        and the traveling wave tail exponents satisfy
        \begin{align*}
            \lambda_R = \frac{\alpha^+}{(\alpha\displaystyle^-)\displaystyle^2}\bigparens{1 + \smallo(1)} & & \text{and} & & \lambda_L = \frac{\alpha^+}{(\alpha\displaystyle^-)\displaystyle^2}\bigparens{1 + \smallo(1)}.
        \end{align*}
        This is referred to as the ``large-$\alpha^-$'' regime.
    \end{enumerate}
\end{corollary}

\begin{remark}\label{twalpharem}
    Regarding the two regimes mentioned above, we have the following:
    \begin{enumerate}[label = (\roman*)]
        \item In the ``small-$\alpha^-$'' regime, where $\alpha^- \downarrow 0$, the wave speed approaches $0$ and the tail behavior is dictated by the vanishing fraction of $+$ particles through the parameters $s \approx \alpha^+$ and $v_+$. Increasing $v_+$ increases the proportion of `large' jumps. Increasing $\alpha^+$ decreases the proportion of $+$ particles even further leading to a weakening of the flocking mechanism, which is essentially manifested through the $+$ particles driving the $-$ particles (which form the majority proportion) and keeping the whole population shape localized.

        \item In the ``large-$\alpha^-$'' regime, with $\alpha^- \bg \alpha^+$, the majority of particles are of the $+$ type. This increases the wave speed towards $v_+$, the maximal admissible value. In this faster reference frame, without the separation of types, the effect of the Central Limit Theorem takes over, leading to diminishing flocking and heavy tails.
    \end{enumerate}
\end{remark}

Figure \ref{fig:TWplots} gives kernel density estimator (KDE) plots for the empirical distribution of $500$ particles initiated at i.i.d. $\norm(0,4)$ locations with initial types sampled with probability $\mathbb{P}(\sigma^n_i(0) = \pm) = \rho^{\pm}$ for $i \in [500]$, where $\rho^+$ and $\rho^-$ are set in terms of $\alpha^+$ and $\alpha^-$, as given by Theorem \ref{travelwave}(a). The first plot gives four snapshots, at times $15$, $30$, $45$, and $60$, with parameters $\alpha^+ = 1.3, \alpha^-=0.6,v_+=2, v_-=0$, in the moving reference frame with speed $\gamma \approx 0.892$ (computed using Theorem \ref{travelwave}), exhibiting emergence of a traveling wave front. The second plot gives the overlay plot (combining both types) at times $20$, $40$, $60$, and $80$, with $v_+=2$, $v_-=0$, $\alpha^+=1.25$, and $\alpha^-=0.08$ (which represent the ``small-$\alpha^-$'' regime). The third plot gives the overlay plot (again combining both types) at times $10$, $20$, $30$, and $40$, still with $v_+=2$, $v_-=0$, but this time, $\alpha^+=1.2$ and $\alpha^-=6$ (which represent the ``large-$\alpha^-$'' regime). Clearly, the latter plot has heavier tails than the former.

\begin{figure}[htbp]
    \centering
    \begin{subfigure}[b]{0.62\textwidth}
        \centering
        \includegraphics[width=\textwidth]{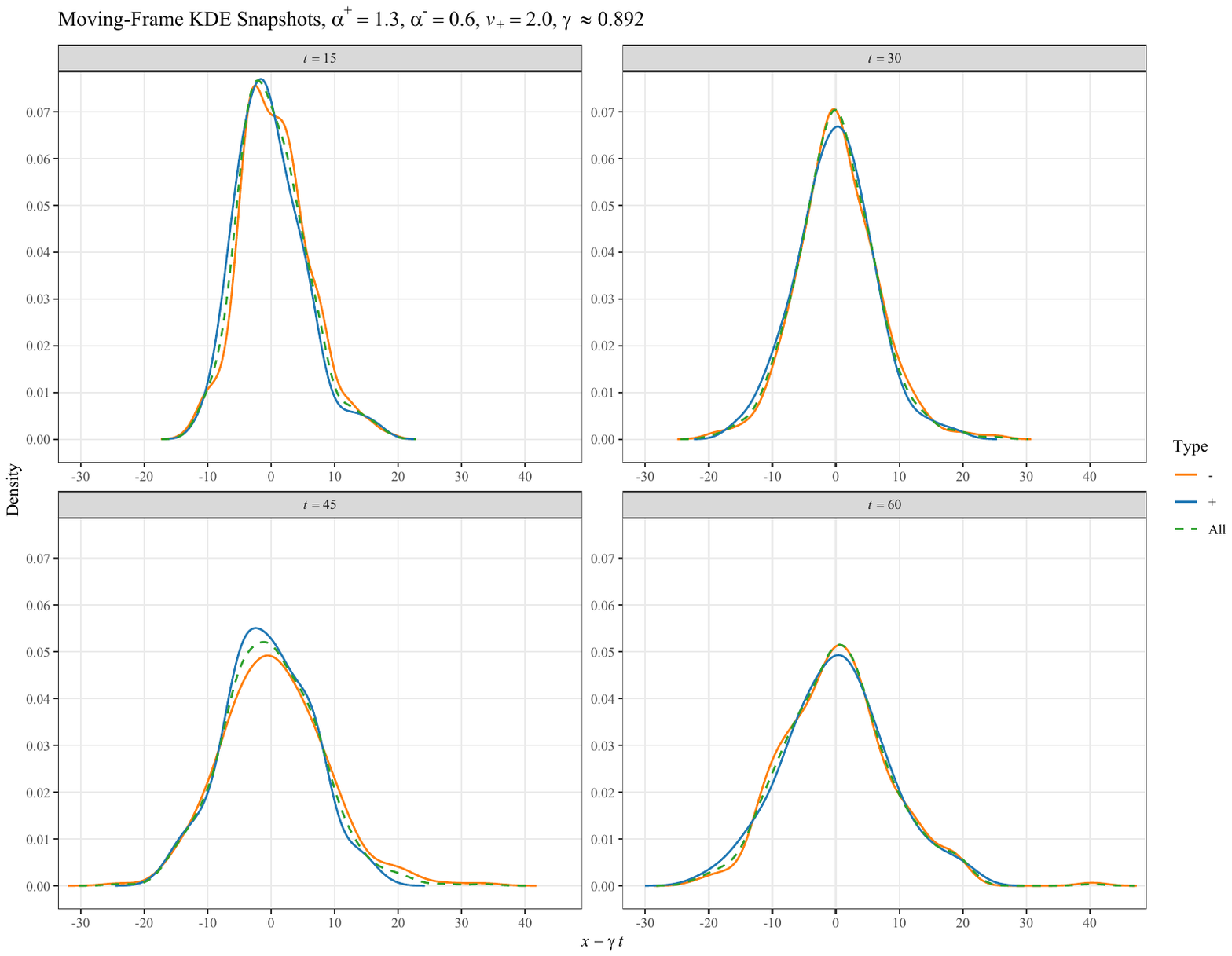}
    \end{subfigure}
    \\
    \begin{subfigure}[b]{0.52\textwidth}
        \centering
        \includegraphics[width=\textwidth]{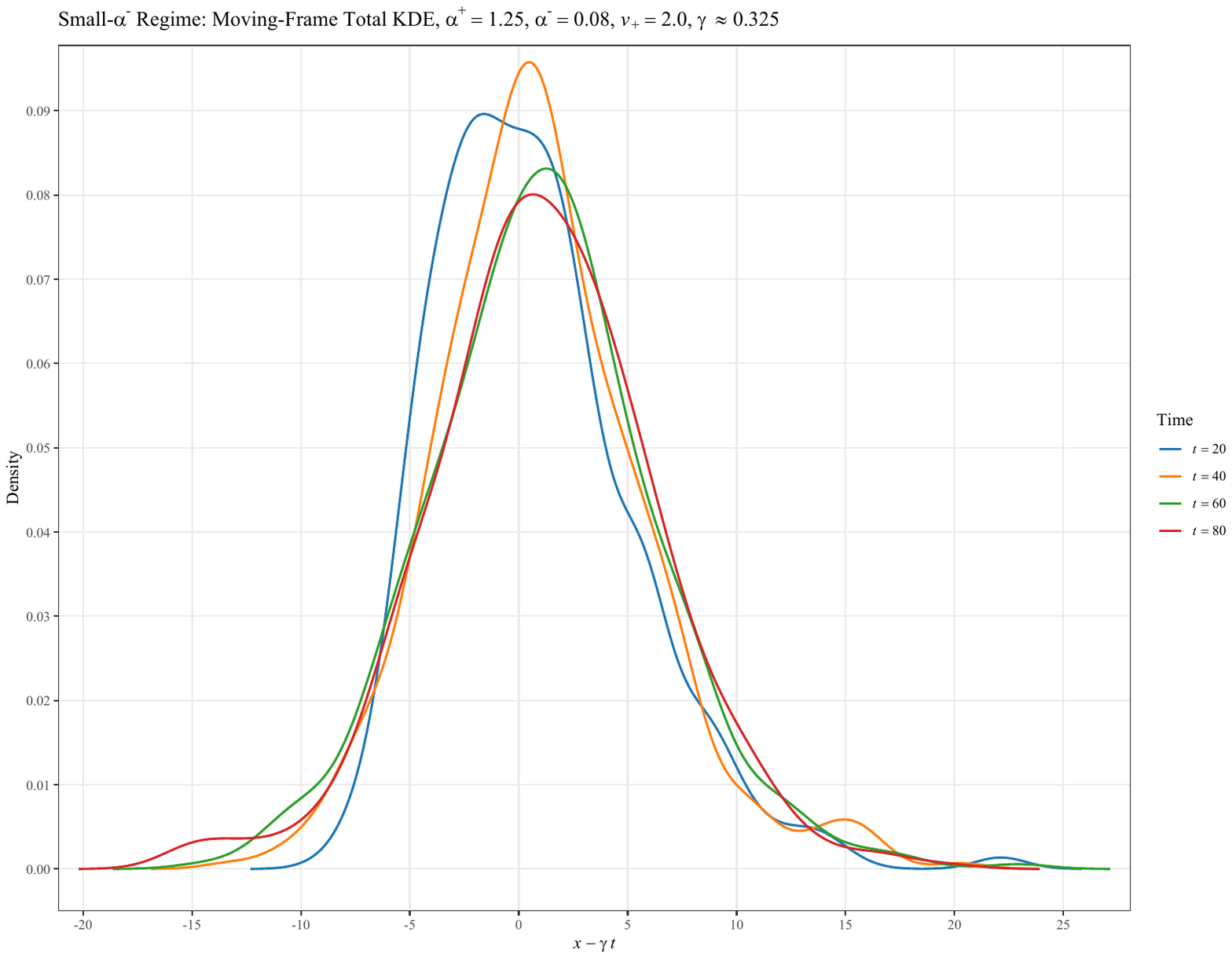}
    \end{subfigure}
    \\
    \begin{subfigure}[b]{0.52\textwidth}
        \centering
        \includegraphics[width=\textwidth]{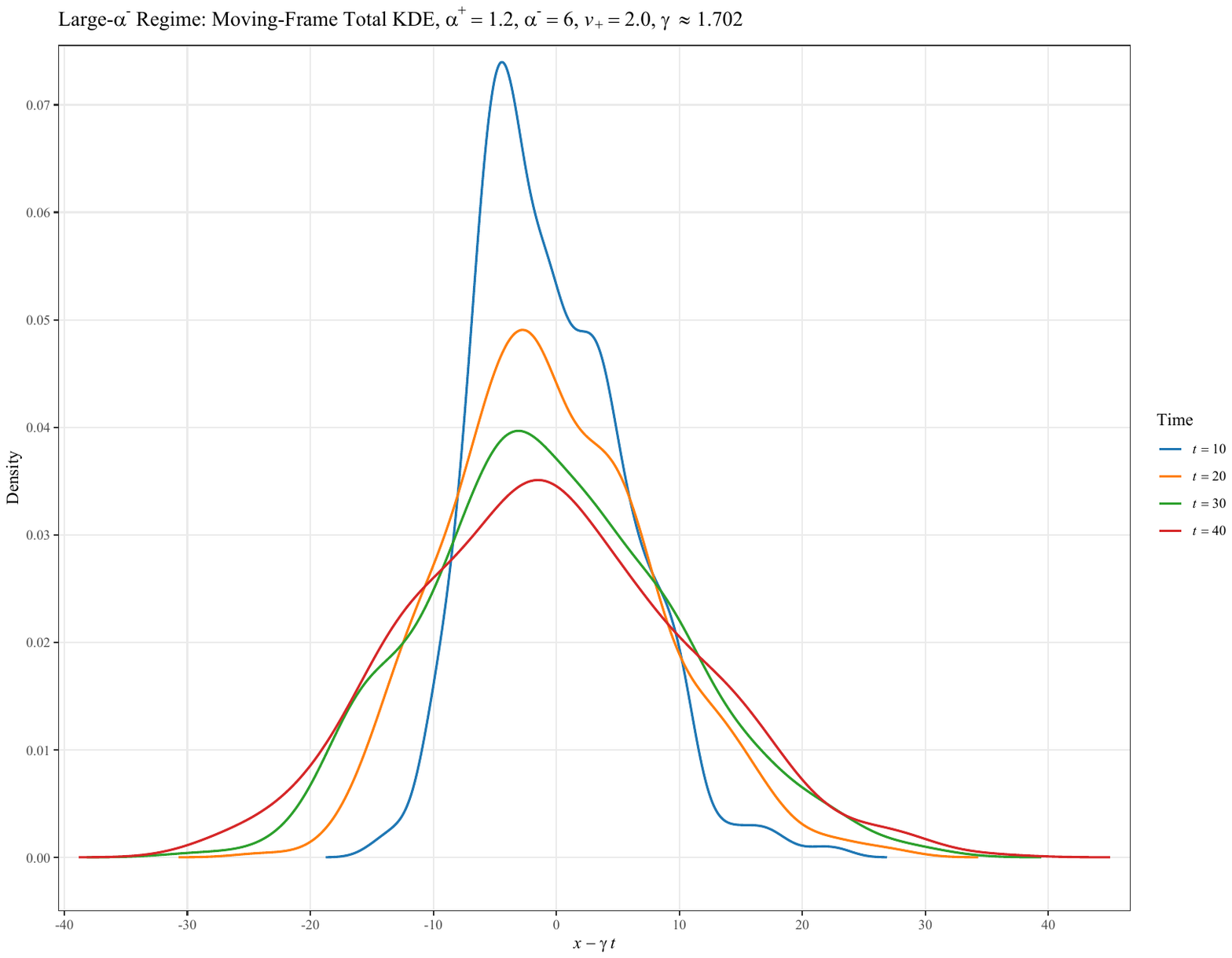}
    \end{subfigure}
    \caption{Traveling wave plots}
    \label{fig:TWplots}
\end{figure}

\subsubsection{Proof Outline and Challenges}
The first step involves obtaining a `strong-formulation' of the traveling wave equation in terms of an autonomous pair of non-linear, non-local integro-differential equations, derived from the `weak-formulation' given by \eqref{originalFluid} [which involves integrals against test-functions in $\HHH$; see \eqref{integrodiff}].

Most examples where traveling waves are shown to exist involve a \emph{dynamical systems} approach where the system is analyzed through coupled ODEs in a (typically) two-dimensional phase space \cite{baryshnikov2025large,an1937etude}. In these references, there is a single type and the second-order ODE governing the dynamics is re-expressed in terms of two coupled \emph{first-order ODEs}, one for the function and the other for its derivative. This is key to analyzing traveling wave solutions via the spectral properties of the associated \emph{linearized system} near equilibrium points.

Obtaining traveling wave solutions for our system of equations \eqref{integrodiff} is challenging owing to (i) its non-local nature which makes ODE-based techniques not readily applicable; and (ii) each equation in the pair \eqref{integrodiff} is second-order, and direct attempts to apply the methods in \cite{an1937etude,baryshnikov2025large} will produce \emph{four} first-order ODEs, leading to an extremely complicated phase space analysis.

These obstacles are circumvented via the following steps:
\begin{itemize}
    \item We circumvent the issue (i) by considering the special case of \emph{exponential jump distribution}, which leads to a \emph{local} description of the non-local equations in terms of second-order coupled ODEs, see \eqref{fluid2}-\eqref{fluid3}. A similar advantage of using exponential jumps is also exploited in \cite{BRT14,BBI24,baryshnikov2025large} for their respective (single-type) models. Analyzing these ODEs leads to the first part of Theorem \ref{travelwave}.

    \item The second obstacle is handled by imposing $v_-=0$. A careful analysis of the ODE system under this additional constraint leads to a \emph{coupled pair of first-order ODEs} for the (complementary) cumulative distribution functions, governing the distribution of particles of each type, obtained in Lemma \ref{twreversed}. This is crucial to the analysis.

    \item In Lemma \ref{linearode}, the only \emph{equilibrium points} of this ODE system are identified as $(0,0)$ and $(1,1)$, and a detailed spectral analysis is performed for the associated \emph{linearized system} around these equilibrium points. In particular, $(0,0)$ is shown to be a stable equilibrium whereas $(1,1)$ is shown to have a one-dimensional unstable manifold.

    \item Finally, in Proposition \ref{monoHeter}, we connect the two equilibria by a \emph{heteroclinic orbit} emerging out of the unstable equilibrium $(1,1)$ along the unstable manifold and eventually getting absorbed in the stable equilibrium $(0,0)$. This exhibits the traveling wave in the second part of Theorem \ref{travelwave}. The asymptotic tail behavior is obtained by a local analysis of the system of ODEs around the equilibrium points.
\end{itemize}
\section{Fluid Limit}
\label{fluidSection}

\subsection{Tightness/Relative Compactness}
 Our first target is to show that $\curlbr*{\mu_n(\cdot)\!: n\in\nat}$ is $\CCC$-tight in $\mathscr{D}\big([0, \infty), \mathscr{P}_1(\real \times \{+,-\})\big)$. To do this, we rely on Theorem 3.7.2 in \cite[132]{EK86}, which breaks this target into two smaller targets. The first task is to show the following:
\begin{theorem}
\label{firstBigThm}
    Suppose Assumption \ref{assumption0} holds. Then for any $T > 0$, the sequence $\curlbr*{\mu_n(\cdot)\!: n\in\nat}$ is Wasserstein-equicontinuous on $[0, T]$. That is, for any $\varepsilon > 0$,
    \begin{align*}
        \lim_{\delta\to 0} \limsup_{n\to\infty} \prob\parens*{\sup_{\substack{0\leq s\leq t\leq T \\ 0< t-s<\delta}} \WWW_1\big(\mu_n(s), \mu_n(t)\big) \geq \varepsilon} = 0.
    \end{align*}
    Here, $\WWW_1$ denotes the $1$-Wasserstein distance on the space $\real \times \{+, -\}$; that is, for $\nu_1\in \PPP_1\bigparens{\real\times \{+,-\}}$ and $\nu_2 \in \PPP_1\bigparens{\real\times \{+,-\}}$, $\WWW_1(\nu_1, \nu_2) := \sup_{f\in\HHH} \curlbr*{\int f \mathrm{d} (\nu_1 - \nu_2)}$.
\end{theorem}

\begin{proof}
    Let $f\in \HHH$ be chosen, and for any $t\geq 0$, define $I_n(t) := \biginner{f, \mu_n(t)} = \dfrac{1}{n} \dsum_{i=1}^n f\big(x^n_i(t), \sigma^n_i(t)\big)$. Observe that for $0\leq s\leq t$,
    \begin{align*}
        \abs*{I_n(t) - I_n(s)} &= \dfrac{1}{n} \abs*{\dsum_{i=1}^n \Bigsqbr{f\big(x^n_i(t), \sigma^n_i(t)\big) - f\big(x^n_i(s), \sigma^n_i(s)\big)}} \\
        &\leq \dfrac{1}{n} \dsum_{i=1}^n \Bigabs{f\big(x^n_i(t), \sigma^n_i(t)\big) - f\big(x^n_i(s), \sigma^n_i(s)\big)} \\
        &\leq \dfrac{1}{n} \dsum_{i=1}^n \big(x^n_i(t) - x^n_i(s)\big) + \dfrac{1}{n} \dsum_{i=1}^n \ind_{\{\sigma^n_i(t) \neq \sigma^n_i(s)\}},
    \end{align*}
    where the first inequality on the second line follows from the Triangle Inequality, while the second inequality follows from our assumption that $f$ is $1$-Lipschitz. Also, since the inequality above holds for any $f\in\HHH$, we have
    \begin{align}
    \label{wassup}
        \WWW_1\big(\mu_n(s), \mu_n(t)\big) = \sup_{f\in\HHH} \abs*{I_n(t) - I_n(s)} \leq \dfrac{1}{n} \dsum_{i=1}^n \big(x^n_i(t) - x^n_i(s)\big) + \dfrac{1}{n} \dsum_{i=1}^n \ind_{\{\sigma^n_i(t) \neq \sigma^n_i(s)\}}.
    \end{align}
    Now, let $\varepsilon > 0$ be chosen. From \eqref{wassup}, for any $T \geq 0$ and any $\delta > 0$,
    \begin{align}
    \label{giantineq1}
        &\prob\parens*{\sup_{\substack{0\leq s\leq t\leq T \\ 0< t-s<\delta}} \WWW_1\big(\mu_n(s), \mu_n(t)\big) \geq \varepsilon} \nonumber \\
        &\qquad \leq \prob\parens*{\sup_{\substack{0\leq s\leq t\leq T \\ 0< t-s<\delta}} \dfrac{1}{n} \dsum_{i=1}^n \big(x^n_i(t) - x^n_i(s)\big) \geq \dfrac{\varepsilon}{2}} + \prob\parens*{\sup_{\substack{0\leq s\leq t\leq T \\ 0< t-s<\delta}} \dfrac{1}{n} \dsum_{i=1}^n \ind_{\{\sigma^n_i(t) \neq \sigma^n_i(s)\}} \geq \dfrac{\varepsilon}{2}}.
    \end{align}
    We treat each term on the right-hand side of \eqref{giantineq1} separately. First, since the system is a non-decreasing pure-jump process,
      \begin{align}
    \label{jumpBound}
        \prob\parens*{\sup_{\substack{0\leq s\leq t\leq T \\ 0< t-s<\delta}} \dfrac{1}{n} \dsum_{i=1}^n \big(x^n_i(t) - x^n_i(s)\big) \geq \dfrac{\varepsilon}{2}} \leq \prob\parens*{\sup_{s\in [0, T]} \dfrac{1}{n} \dsum_{i=1}^n \big(x^n_i(s+2\delta) - x^n_i(s)\big) \geq \dfrac{\varepsilon}{2}}.
    \end{align}
    To bound the probability on the right-hand side of \eqref{jumpBound}, we construct a relevant auxiliary particle system, $\curlbr*{\tilde{x}_i(\cdot)\!: i\in [n]}$, in which the particles are \textbf{of a single type}, as follows:
    \begin{itemize}
        \item For each $i\in [n]$, let $\tilde{x}_i(0) := x^n_i(0)$.

        \item All particles jump with rate $v := v_+ + v_-$, and the jump length is a random variable from the jump length distribution, $J$. Here, if $v = 0$, then every term controlling the spatial location vanishes, so suppose that $v > 0$.
    \end{itemize}
    This auxiliary system has two monotonicity properties, both of which are to our advantage:
    \begin{itemize}
        \item First, for any $i\in [n]$ and $t\geq 0$, $\tilde{x}_i(t) \geq x^n_i(t)$.

        \item Second, for any $i\in [n]$ and $0\leq s\leq t$, $\tilde{x}_i(t) - \tilde{x}_i(s) \geq x^n_i(t) - x^n_i(s)$.
    \end{itemize}

    In addition, conditional on the initial configuration, the jump mechanisms of the particles in this system are mutually independent. With that, define $\Bar{Y}_n(t) := \dfrac{1}{n} \dsum_{i=1}^n\!\tilde{x}_i(t)$. By the monotonicity of $\curlbr*{\tilde{x}_i(\cdot)\!: i\!\in\![n]}$, for $\delta < \min\curlbr*{\dfrac{\varepsilon}{8v}, 1}$,
    \begin{align*}
        &\prob\parens*{\sup_{s\in [0, T]} \dfrac{1}{n} \dsum_{i=1}^n \bigparens{x^n_i(s+2\delta) - x^n_i(s)} \geq \dfrac{\varepsilon}{2}} \\
        &\qquad \leq \prob\parens*{\sup_{s\in [0, T]} \dfrac{1}{n} \dsum_{i=1}^n \bigparens{\tilde{x}_i(s+2\delta) - \tilde{x}_i(s)} \geq \dfrac{\varepsilon}{2}} \\
        &\qquad = \prob\parens*{\sup_{s\in [0, T]} \bigparens{\bar{Y}_n(s+2\delta) - \bar{Y}_n(s)} \geq \dfrac{\varepsilon}{2}} \\
        &\qquad = \prob\parens*{\sup_{s\in [0, T]} \sqbr*{\bigparens{\bar{Y}_n(s + 2\delta) - v(s + 2\delta)} -\!\bigparens{\bar{Y}_n(s) - vs}} \geq \dfrac{\varepsilon}{2} -\!2v\delta } \\
        &\qquad \leq \prob\parens*{\sup_{s\in [0, T]} \sqbr*{\bigparens{\bar{Y}_n(s+2\delta)  - v(s+2\delta)} - \bigparens{\bar{Y}_n(s) -vs}} \geq \dfrac{\varepsilon}{4}}.
    \end{align*}
    Now, for each $n\in\nat$ and $t\geq 0$, let
    \begin{align*}
        K_n(t) := \bar{Y}_n(t) - \bar{Y}_n(0) - vt
    \end{align*}
    so that $\curlbr*{K_n(t)\!: t\geq 0}$ is a martingale (in $t$, with respect to the natural filtration), and observe that
      \begin{align*}
        &\prob\parens*{\sup_{s\in [0, T]} \sqbr*{\bigparens{\bar{Y}_n(s + 2\delta) - v(s + 2\delta)} - \bigparens{\bar{Y}_n(s) - vs}} \geq \dfrac{\varepsilon}{4}} \\
        &\qquad = \prob\parens*{\sup_{s\in [0, T]} \bigparens{K_n(s+2\delta)-K_n(s)}\geq \dfrac{\varepsilon}{4}} \\
        &\qquad \leq \prob\parens*{\sup_{s\in [0, T]} \abs*{K_n(s+2\delta)} + \sup_{s\in [0, T]} \abs*{K_n(s)} \geq \dfrac{\varepsilon}{4}} \\
        &\qquad \leq \prob\parens*{\sup_{s\in [0, T]} \abs*{K_n(s+2\delta)} \geq  \dfrac{\varepsilon}{8}} + \prob\parens*{\sup_{s\in [0, T]} \abs*{K_n(s)} \geq \dfrac{\varepsilon}{8}}.
    \end{align*}
    Since $\curlbr*{K_n(t)\!: t\geq 0}$ is a martingale, by Doob's $\LLL^2$ Inequality,
    \begin{align*}
        \prob\parens*{\sup_{t\in [0, T+2]} \abs*{K_n(t)} \geq \dfrac{\varepsilon}{8}} \leq \dfrac{64\expc\sqbr*{K^2_n(T+2)}}{\varepsilon\displaystyle^2},
    \end{align*}
    and note that since $\mathbb{E}(Z^2) < \infty$, $\expc\sqbr*{K^2_n(T+2)} = \dfrac{v\cdot(T+2)\expc\parens*{Z\displaystyle^2}}{n} \toinf{n} 0$; thus,
    \begin{align*}
        \lim_{\delta \to 0} \limsup_{n\to \infty} \prob\parens*{\sup_{\substack{0\leq s\leq t\leq T \\ 0< t-s<\delta}} \dfrac{1}{n} \dsum_{i=1}^n \bigparens{x^n_i(t) - x^n_i(s)} \geq \dfrac{\varepsilon}{2}} = 0.
    \end{align*}
    Now, for the remaining term on the right-hand side of \eqref{giantineq1}, observe that
    \begin{align*}
        \prob\parens*{\sup_{\substack{0\leq s\leq t\leq T \\ 0< t-s<\delta}} \dfrac{1}{n} \dsum_{i=1}^n \ind_{\{\sigma^n_i(t) \neq \sigma^n_i(s)\}} \geq \dfrac{\varepsilon}{2}} \leq \prob\parens*{\sup_{\substack{0\leq s\leq t\leq T \\ 0< t-s<\delta}} \dfrac{1}{n} \dsum_{i=1}^n W_i(s, t) \geq \dfrac{\varepsilon}{2}},
    \end{align*}
    where $W_i(s, t)$ is the number of times particle $i$ changes type within the interval $(s, t]$. Clearly, by monotonicity,
    \begin{align*}
        \prob\parens*{\sup_{\substack{0\leq s\leq t\leq T \\ 0< t-s<\delta}} \dfrac{1}{n} \dsum_{i=1}^n W_i(s, t) \geq \dfrac{\varepsilon}{2}} \leq \prob\parens*{\sup_{s\in [0, T]} \dfrac{1}{n} \dsum_{i=1}^n W_i(s, s+2\delta) \geq \dfrac{\varepsilon}{2}}.
    \end{align*}
    Since the rate of change of type for any particle is bounded above by $\alpha := \alpha^+ + \alpha^- + \lnorm*{\varphi}_{\infty} < \infty$, for (independent and identically distributed) Poisson processes $\{N_i(\cdot) : i \in [n]\}$ with rate $\alpha$, we have that
    \begin{align*}
        \prob\parens*{\sup_{s\in [0, T]} \dfrac{1}{n} \dsum_{i=1}^n W_i(s, s+2\delta) \geq \dfrac{\varepsilon}{2}} \leq \prob\parens*{\sup_{s\in [0, T]} \dfrac{1}{n} \dsum_{i=1}^n \bigparens{N_i(s+2\delta) - N_i(s)} \geq \dfrac{\varepsilon}{2}}.
    \end{align*}
    For each $n\in\nat$ and $t \geq 0$, let $K'_n(t) := \frac{1}{n}\dsum_{i=1}^n N_i(t) - \alpha t$ so that $\curlbr*{K'_n(t)\!: t\geq 0}$ is also a martingale (in $t$, with respect to the natural filtration). Then for $\delta < \dfrac{\varepsilon}{8\alpha}$,
      \begin{align*}
        \prob\parens*{\sup_{s\in [0, T]} \dfrac{1}{n} \dsum_{i=1}^n \bigparens{N_i(s+2\delta) - N_i(s)} \geq \dfrac{\varepsilon}{2}} &= \prob\parens*{\sup_{s\in [0, T]} \bigparens{K'_n(s+2\delta) - K'_n(s)} \geq \dfrac{\varepsilon}{2} - 2\alpha \delta} \\
        &\leq \prob\parens*{\sup_{s\in [0, T]} \bigparens{K'_n(s+2\delta) - K'_n(s)} \geq \dfrac{\varepsilon}{4}} \\
        &\leq \prob\parens*{\sup_{s\in [0, T]} \abs*{K_n'(s\!+\!2\delta)} \geq \dfrac{\varepsilon}{8}}\!+\!\prob\parens*{\sup_{s\in [0, T]} \abs*{K'_n(s)} \geq \dfrac{\varepsilon}{8}}.
    \end{align*}
    Again, since $\curlbr*{K'_n(t)\!: t\geq 0}$ is a martingale (in $t$, with respect to the natural filtration), by a parallel argument to that above using Doob's $\LLL^2$ Inequality to  $\curlbr*{K'_n(t)\!: t\geq 0}$, we also have
      \begin{align*}
        \lim_{\delta \to 0} \limsup_{n\to \infty} \prob\parens*{\sup_{\substack{0\leq s\leq t\leq T \\ 0< t-s<\delta}} \dfrac{1}{n} \dsum_{i=1}^n \ind_{\{\sigma^n_i(t) \neq \sigma^n_i(s)\}} \geq \dfrac{\varepsilon}{2}} = 0,
     \end{align*}
    and our result follows.
\end{proof}
 The next step is to show that our sequence of measures satisfies a certain form of \emph{uniform integrability}:
\begin{theorem}
\label{tight2}
    Suppose that Assumptions \ref{assumption0} and \ref{assumption2} hold. Then for any $t > 0$ and any $\eta > 0$,
      \begin{align*}
        \lim_{L\to\infty} \limsup_{n\to\infty} \prob\parens*{\int \abs*{x} \ind_{\curlbr*{|x| \geq L}} \mu^\circ_n(t, \di x) \geq \eta} = \lim_{L\to\infty} \limsup_{n\to\infty} \prob\parens*{\dfrac{1}{n}\dsum_{i=1}^n |x^n_i(t)| \ind_{\curlbr*{|x^n_i(t)| \geq L}} \geq \eta} = 0.
    \end{align*}
\end{theorem}
\begin{proof}
    To begin, we re-employ one of the auxiliary particle systems we used in the proof of Theorem \ref{firstBigThm}. Recall the particle system $\curlbr*{\tilde{x}_i(\cdot)\!: i\in [n]}$, and let $t > 0$ be chosen. If $v = 0$, then the result follows immediately from Assumption \ref{assumption2}, so suppose that $v > 0$. For each $i\in [n]$, let
      \begin{align*}
        T_i(t) := x^n_i(t) - x^n_i(0) & & \text{and} & & \tilde{T}_i(t) := \tilde{x}_i(t) - \tilde{x}_i(0).
    \end{align*}
     Observe that $T_i(t) \leq \tilde{T}_i(t)$. Also, let $N_i(t)$ and $\tilde{N}_i(t)$ denote the number of jumps that particle $i$ makes up to time $t$ in its respective system. Then since $\tilde{N}_i(t) \sim \pois(vt)$, we have
      \begin{align*}
        \expc\sqbr*{\tilde{T}_i(t)} = \expc\sqbr*{\tilde{N}_i(t)} \expc(Z) = vt < \infty,
     \end{align*}
     and by Markov's Inequality,
    \begin{align*}
        \prob\parens*{\tilde{T}_i(t) \geq \dfrac{L}{2}} \leq \dfrac{2\expc\sqbr*{\tilde{T}_i(t)}}{L} \toinf{L} 0.
    \end{align*}
    Naturally, since $\tilde{T}_i(t)$ is integrable,
    \begin{align}
    \label{someUIthingy}
        \expc\sqbr*{\tilde{T}_i(t) \ind_{\curlbr*{\tilde{T}_i(t) \geq L/2}}} \toinf{L} 0,
    \end{align}
    and since $\ind_{\curlbr*{|x^n_i(t)| \geq L}} \leq \ind_{\curlbr*{|x^n_i(0)| \geq L/2}} + \ind_{\curlbr*{T_i(t) \geq L/2}}$, and $|x^n_i(t)| \leq |x^n_i(0)| + T_i(t)$, we have
      \begin{align*}
        |x^n_i(t)| \ind_{\curlbr*{|x^n_i(t)| \geq L}} &\leq |x^n_i(0)| \ind_{\curlbr*{|x^n_i(0)| \geq L/2}} + T_i(t) \ind_{\curlbr*{|x^n_i(0)| \geq L/2}} \\
        &\qquad + |x^n_i(0)| \ind_{\curlbr*{T_i(t) \geq L/2}} + T_i(t) \ind_{\curlbr*{T_i(t) \geq L/2}},
     \end{align*}
    and so
    \begin{align*}
        \dfrac{1}{n} \dsum_{i=1}^n |x^n_i(t)| \ind_{\curlbr*{|x^n_i(t)| \geq L}} &\leq \dfrac{1}{n} \dsum_{i=1}^n|x^n_i(0)| \ind_{\curlbr*{|x^n_i(0)| \geq L/2}} + \dfrac{1}{n} \dsum_{i=1}^n T_i(t) \ind_{\curlbr*{|x^n_i(0)| \geq L/2}} \\
        &+ \dfrac{1}{n} \dsum_{i=1}^n |x^n_i(0)| \ind_{\curlbr*{T_i(t) \geq L/2}} + \dfrac{1}{n} \dsum_{i=1}^n T_i(t) \ind_{\curlbr*{T_i(t) \geq L/2}}.
    \end{align*}
    Let $\eta > 0$ be chosen. By the union bound,
      \begin{align}
      \label{eta4}
        &\prob\parens*{\dfrac{1}{n} \dsum_{i=1}^n |x^n_i(t)| \ind_{\curlbr*{|x^n_i(t)| \geq L}} \geq \eta} \nonumber \\
        &\leq \prob\parens*{\dfrac{1}{n} \dsum_{i=1}^n|x^n_i(0)| \ind_{\curlbr*{|x^n_i(0)| \geq L/2}} \geq \dfrac{\eta}{4}} + \prob\parens*{\dfrac{1}{n} \dsum_{i=1}^n T_i(t) \ind_{\curlbr*{|x^n_i(0)| \geq L/2}} \geq \dfrac{\eta}{4}} \nonumber \\
        &\qquad + \prob\parens*{\dfrac{1}{n} \dsum_{i=1}^n |x^n_i(0)| \ind_{\curlbr*{T_i(t) \geq L/2}} \geq \dfrac{\eta}{4}} + \prob\parens*{\dfrac{1}{n} \dsum_{i=1}^n T_i(t) \ind_{\curlbr*{T_i(t) \geq L/2}} \geq \dfrac{\eta}{4}}.
     \end{align}
    Our goal now is to bound each of the four terms on the right-hand side of \eqref{eta4} as both $n$ and $L$ grow. First, from Assumption \ref{assumption2}, we know that
      \begin{align*}
        \lim_{L\to\infty} \limsup_{n\to\infty} \prob\parens*{\dfrac{1}{n} \dsum_{i=1}^n|x^n_i(0)| \ind_{\curlbr*{|x^n_i(0)| \geq L/2}} \geq \dfrac{\eta}{4}} = 0,
      \end{align*}
     and the first term on the right-hand side of \eqref{eta4} is bounded. Next, since $T_i(t) \leq \tilde{T}_i(t)$, we have
    \begin{align*}
        \dfrac{1}{n} \dsum_{i=1}^n T_i(t) \ind_{\curlbr*{T_i(t) \geq L/2}} \leq \dfrac{1}{n} \dsum_{i=1}^n \tilde{T}_i(t) \ind_{\curlbr*{\tilde{T}_i(t) \geq L/2}},
    \end{align*}
    and taking expectation yields
    \begin{align*}
        \expc\sqbr*{\dfrac{1}{n} \dsum_{i=1}^n T_i(t) \ind_{\curlbr*{T_i(t) \geq L/2}}} &\leq \expc\sqbr*{\dfrac{1}{n} \dsum_{i=1}^n \tilde{T}_i(t) \ind_{\curlbr*{\tilde{T}_i(t) \geq L/2}}} = \expc\sqbr*{\tilde{T}_1(t) \ind_{\curlbr*{\tilde{T}_1(t) \geq L/2}}},
    \end{align*}
    where the last equality follows from the fact that the $\tilde{T}_i(t)$'s are independent and identically distributed. Since this bound is uniform in $n$, by Markov's Inequality  and from \eqref{someUIthingy},
    \begin{align*}
        \sup_{n\in\nat} \prob\parens*{\dfrac{1}{n} \dsum_{i=1}^n T_i(t) \ind_{\curlbr*{T_i(t) \geq L/2}} \geq \dfrac{\eta}{4}} \leq \dfrac{4\expc\sqbr*{\tilde{T}_1(t) \ind_{\curlbr*{\tilde{T}_1(t) \geq L/2}}}}{\eta} \toinf{L} 0.
    \end{align*}
    With that, the last term on the right-hand side of \eqref{eta4} is   also bounded. For the second term on the right-hand side of \eqref{eta4}, observe that
      \begin{align*}
        \prob\parens*{\dfrac{1}{n} \dsum_{i=1}^n T_i(t) \ind_{\curlbr*{|x^n_i(0)| \geq L/2}}\geq\dfrac{\eta}{4}} & \leq \prob \parens*{\dfrac{1}{n} \dsum_{i=1}^n \tilde{T}_i(t) \ind_{\curlbr*{|x^n_i(0)| \geq L/2}} \geq \dfrac{\eta}{4}} \\
        &\leq \dfrac{4}{\eta} \expc\sqbr*{\dfrac{1}{n} \dsum_{i=1}^n \tilde{T}_i(t) \ind_{\curlbr*{|x^n_i(0)| \geq L/2}}} \leq \dfrac{4vt}{\eta} \expc\sqbr*{\dfrac{1}{n} \dsum_{i=1}^n \ind_{\curlbr*{|x^n_i(0)| \geq L/2}}}.
     \end{align*}
     Now, for any $\delta \in (0, 1)$, we have
      \begin{align*}
        &\expc\sqbr*{\dfrac{1}{n} \dsum_{i=1}^n \ind_{\curlbr*{|x^n_i(0)| \geq L/2}}} \\
        &\qquad \leq \delta + \prob\parens*{\dfrac{1}{n} \dsum_{i=1}^n \ind_{\curlbr*{|x^n_i(0)| \geq L/2}} \geq \delta} \leq \delta + \prob\parens*{\dfrac{1}{n} \dsum_{i=1}^n |x^n_i(0)|\ind_{\curlbr*{|x^n_i(0)| \geq L/2}} \geq \dfrac{L\delta}{2}},
    \end{align*}
    and by Assumption \ref{assumption2} (again),
    \begin{align*}
        \lim_{L\to\infty} \limsup_{n\to\infty} \prob\parens*{\dfrac{1}{n} \dsum_{i=1}^n |x^n_i(0)|\ind_{\curlbr*{|x^n_i(0)| \geq L/2}} \geq \dfrac{L\delta}{2}} = 0.
    \end{align*}
    Let $\varepsilon > 0$ be chosen, and let $\delta \leq \min\curlbr*{\dfrac{\varepsilon \eta}{8vt}, 1}$ be chosen so that $\dfrac{4vt}{\eta}\cdot \delta \leq \dfrac{\varepsilon}{2}$. With this (fixed) $\delta \in (0, 1)$, let $L > 0$ be so large that
    \begin{align*}
        \limsup_{n\to\infty} \prob\parens*{\dfrac{1}{n} \dsum_{i=1}^n |x^n_i(0)|\ind_{\curlbr*{|x^n_i(0)| \geq L/2}} \geq \dfrac{L\delta}{2}} \leq \dfrac{\varepsilon\eta}{8vt}.
    \end{align*}
    Then
      \begin{align*}
        \limsup_{n\to\infty} \prob\parens*{\dfrac{1}{n} \dsum_{i=1}^n T_i(t) \ind_{\curlbr*{|x^n_i(0)| \geq L/2}}\geq\dfrac{\eta}{4}} & \leq \dfrac{4vt}{\eta} \limsup_{n\to\infty} \expc\sqbr*{\dfrac{1}{n} \dsum_{i=1}^n \ind_{\curlbr*{|x^n_i(0)| \geq L/2}}} \\
        &\leq \dfrac{4vt}{\eta} \sqbr*{\delta\!+\!\limsup_{n\to\infty} \prob\parens*{\dfrac{1}{n}\!\dsum_{i=1}^n |x^n_i(0)|\ind_{\curlbr*{|x^n_i(0)| \geq L/2}} \geq \dfrac{L\delta}{2}}} \\
        &\leq \dfrac{4vt}{\eta} \cdot \delta + \dfrac{4vt}{\eta} \cdot \dfrac{\varepsilon\eta}{8vt} \leq \dfrac{\varepsilon}{2} + \dfrac{\varepsilon}{2} = \varepsilon,
     \end{align*}
    and indeed, $\lim_{L\to\infty} \limsup_{n\to\infty} \prob\parens*{\dfrac{1}{n} \dsum_{i=1}^n T_i(t) \ind_{\curlbr*{|x^n_i(0)| \geq L/2}}\geq\dfrac{\eta}{4}} = 0$. Lastly, to bound the remaining (third) term on the right-hand side of \eqref{eta4}, first, note that
    \begin{align*}
        A_n(L) := \dfrac{1}{n} \dsum_{i=1}^n |x^n_i(0)| \ind_{\curlbr*{T_i(t) \geq L/2}} \leq \dfrac{1}{n} \dsum_{i=1}^n |x^n_i(0)| \ind_{\curlbr*{\tilde{T}_i(t) \geq L/2}} =: \tilde{A}_n(L),
    \end{align*}
    so it suffices to show that $\lim_{L\to\infty} \limsup_{n\to\infty} \prob\parens*{\tilde{A}_n(L) \geq \dfrac{\eta}{4}} = 0$. To that end, let
      \begin{align*}
        \mathscr{I}_n(0) := \int |x| \mu^\circ_n(0, \di x) = \dfrac{1}{n} \dsum_{i=1}^n |x^n_i(0)|.
     \end{align*}
      For any $K > 0$, we have
    \begin{align}
    \label{lastEtaBound}
        \prob\parens*{\tilde{A}_n(L) \geq \dfrac{\eta}{4}} \leq \prob\big(\mathscr{I}_n(0) > K\big) + \prob\parens*{\curlbr*{\tilde{A}_n(L) \geq \dfrac{\eta}{4}} \cap \curlbr*{\mathscr{I}_n(0) \leq K}}.
    \end{align}
    The remaining task is to bound each term on the right-hand side of \eqref{lastEtaBound}. First, note that
      \begin{align*}
        \expc\sqbr*{\tilde{A}_n(L)\big|\FFF(0)} = \expc\sqbr*{\dfrac{1}{n} \dsum_{i=1}^n |x^n_i(0)| \ind_{\curlbr*{\tilde{T}_i(t) \geq L/2}}\Bigg|\FFF(0)} & = \prob\parens*{\tilde{T}_i(t) \geq \dfrac{L}{2}\Bigg|\FFF(0)} \cdot \dfrac{1}{n} \dsum_{i=1}^n |x^n_i(0)| \\
        &= \prob\parens*{\tilde{T}_i(t) \geq \dfrac{L}{2}} \cdot \mathscr{I}_n(0) \\
        & \leq \dfrac{2}{L} \expc\sqbr*{\tilde{T}_i(t)} \cdot \III_n(0) = \dfrac{2vt}{L} \cdot \III_n(0),
     \end{align*}
    where the inequality on the last line follows from Markov's Inequality. Next, using the (conditional) Markov bound, we have
    \begin{align*}
        \prob\parens*{\curlbr*{\tilde{A}_n(L) \geq \dfrac{\eta}{4}} \cap \curlbr*{\mathscr{I}_n(0) \leq K}} &= \expc\sqbr*{\ind_{\curlbr*{\tilde{A}_n(L) \geq \eta/4}} \ind_{\curlbr*{\III_n(0) \leq K}}} \\
        &= \expc\sqbr*{\ind_{\curlbr*{\III_n(0) \leq K}} \prob\parens*{\tilde{A}_n(L) \geq \dfrac{\eta}{4} \Bigg|\FFF(0)}} \\
        &\leq \dfrac{4}{\eta} \expc\Big[\ind_{\curlbr*{\III_n(0) \leq K}} \expc\sqbr*{\tilde{A}_n(L)\big|\FFF(0)}\Big] \\
        &\leq \dfrac{4}{\eta} \expc\Big[\ind_{\curlbr*{\III_n(0) \leq K}} \dfrac{2vt}{L} \cdot \III_n(0) \Big] \leq \dfrac{8Kvt}{L \eta}.
    \end{align*}
    To bound the remaining term in \eqref{lastEtaBound}, let $\varepsilon > 0$ be chosen, and note that for any $K > 0$,
    \begin{align*}
        \III_n(0) = \int |x| \mu^\circ_n(0, \di x) &= \int |x| \ind_{\curlbr*{|x| < K/2}} \mu^\circ_n(0, \di x) + \int |x| \ind_{\curlbr*{|x| \geq K/2}} \mu^\circ_n(0, \di x) \\
        &\leq \dfrac{K}{2} + \int |x| \ind_{\curlbr*{|x| \geq K/2}} \mu^\circ_n(0, \di x);
    \end{align*}
    thus,
    \begin{align*}
        \prob\big(\III_n(0) > K\big) \leq \prob\parens*{\int |x| \ind_{\curlbr*{|x| \geq K/2}} \mu^\circ_n(0, \di x) > \dfrac{K}{2}},
    \end{align*}
    so by Assumption \ref{assumption2}, we have $\lim_{K\to\infty} \limsup_{n\to\infty} \prob\big(\III_n(0) > K\big) = 0$. Next, let $K > 0$ be   chosen so that $\limsup_{n\to\infty} \prob\big(\III_n(0) > K\big) < \dfrac{\varepsilon}{2}$, and with this (fixed) $K > 0$, let $L > 0$ be chosen so that $\dfrac{8Kvt}{L \eta} < \dfrac{\varepsilon}{2}$. Then by \eqref{lastEtaBound},
      \begin{align*}
        \limsup_{n\to\infty} \prob\parens*{\tilde{A}_n(L) \geq \dfrac{\eta}{4}} &\leq \limsup_{n\to\infty} \prob\big(\mathscr{I}_n(0) > K\big) + \limsup_{n\to\infty} \prob\parens*{\curlbr*{\tilde{A}_n(L) \geq \dfrac{\eta}{4}} \cap \curlbr*{\mathscr{I}_n(0) \leq K}} \\
        &\leq \dfrac{\varepsilon}{2} + \dfrac{\varepsilon}{2} = \varepsilon.
    \end{align*}
    That gives $\lim_{L\to\infty} \limsup_{n\to\infty} \prob\parens*{\tilde{A}_n(L) \geq \dfrac{\eta}{4}} = 0$, and with that, 
    \begin{align*}
        \lim_{L\to\infty} \limsup_{n\to\infty} \prob\parens*{A_n(L) \geq \dfrac{\eta}{4}} = 0
    \end{align*}
    as well. Our proof is complete.
\end{proof}
To tie the previous two results together, we have the following:
\begin{corollary}
    \label{tightD}
    Suppose that Assumptions \ref{assumption0} and \ref{assumption2} hold. Then the sequence $\curlbr*{\mu_n(\cdot)\!: n\in\nat}$ is $\CCC$-tight in the space $\DDD\bigparens{[0, \infty), \PPP_1(\real\times \{+,-\})}$.
\end{corollary}
\begin{proof}
    From Theorem 7.2 in Chapter 3 of \cite[128]{EK86}, the $\CCC$-tightness of $\curlbr*{\mu_n(\cdot)\!: n\in\nat}$ follows directly from Theorem \ref{firstBigThm} and Theorem \ref{tight2}.
\end{proof}

\subsection{Martingale and MVE}
So far, we know that  $\curlbr*{\mu_n(\cdot)\!: n\in\nat}$ is $\CCC$-tight, so every subsequence of $\curlbr*{\mu_n(\cdot)\!: n\in\nat}$ has a further subsequence that converges to a limit in $\CCC\bigparens{[0, \infty), \PPP_1(\real\times \{+,-\})}$. The natural next step is to characterize any such subsequential limits. We start with the following characterization of the ``martingale'' term:
\begin{theorem}
\label{martingale1}
    Suppose that Assumptions \ref{assumption0}, \ref{assumption3}, and \ref{assumption4} hold. For every $t \geq 0$ and $f\in\HHH$,
    \begin{align*}
        \sup_{s\in [0, t]} \abs*{A_{s, f}\big(\mu_n(\cdot)\big)} \cpinf{n} 0,
    \end{align*}
    where, for $t\geq 0$ and $f\in \HHH$, recall that
      \begin{align*}
        A_{t,f}\big(\mu_n(\cdot)\big) = \biginner{f, \mu_n(t)} - \biginner{f, \mu_n(0)} - \int\limits_0^t\! \biginner{\LLL f(\cdot, \mu_n(s)), \mu_n(s)} \di s.
     \end{align*}
\end{theorem}
 \begin{remark}
    By Assumption \ref{assumption3}, the initial spatial configuration is almost-surely tie-free. Furthermore, from Assumption \ref{assumption4}, since $J$ admits a density, a spatial jump cannot create a tie with positive conditional probability. Indeed, conditional on the pre-jump configuration, particle $i$ can land at the location of particle $j$ only if its jump length equals the single prescribed value $x^n_j(t^-) - x^n_i(t^-)$, while type flips do not affect particle locations. Thus, by non-explosion and a countable-union argument,
    \begin{align*}
        \prob\bigparens{x^n_i(t) \neq x^n_j(t) \text{ for every } t\geq 0 \text{ and every } i\neq j} = 1.
    \end{align*}
    That implies
    \begin{align*}
        \ind_{\curlbr*{j \succ_{\xx^n(t^-)} i}} \eqas \ind_{\curlbr*{x^n_j(t^-) > x^n_i(t^-)}},
    \end{align*}
    and therefore
    \begin{align*}
        \alpha^{\sigma_i=\sigma^n_i(t^-)}_n\bigparens{\xx^n(t^-), \pmb{\sigma}^n(t^-)} \eqas \alpha^{\sigma^n_i(t^-)}_n\bigparens{t^-, x^n_i(t^-)}.
    \end{align*}
\end{remark}
 \begin{proof}
    Let $t\geq 0$ and $f\in\HHH$ be chosen, and for shorthand, define $M_n(t) := A_{t,f}\big(\mu_n(\cdot)\big)$. By the preceding remark, the actual particle flip rates agree almost surely with the projected rates appearing in the generator. Hence, applying Dynkin's formula shows that $\curlbr*{M_n(t)\!: t\geq 0}$ is a martingale (in $t$). By Doob's $\LLL^2$ Inequality, then, for any $c > 0$,
    \begin{align*}
        \prob\parens*{\sup_{s\in [0, t]} \abs*{M_n(s)} \geq c} &\leq c^{-2} \expc\sqbr*{\sup_{s\in [0, t]} \big(M^2_n(s)\big)} \leq 4c^{-2} \expc\sqbr*{M^2_n(t)}.
    \end{align*}
    It suffices to show that $\lim_{n\to\infty} \expc\sqbr*{M\displaystyle^2_n(t)} = 0$. To that end, since $\curlbr*{M_n(t)\!: t\geq 0}$ is a martingale, by the Doob-Meyer decomposition, we have
    \begin{align*}
        \expc\sqbr*{M^2_n(t)} = \expc\sqbr*{\inner*{M_n}(t)},
    \end{align*}
    where $\curlbr*{\inner*{M_n}(t)\!: t\geq 0}$ denotes the (predictable) quadratic variation process of $\curlbr*{M_n(t)\!: t\geq 0}$. More specifically, for each $n\in\nat$ and $t\geq 0$,
      \begin{align*}
        \inner*{M_n}(t) := \int\limits_0^t \dfrac{1}{n\displaystyle^2} \dsum_{i=1}^n \curlbr*{v_{\sigma^n_i(s^-)} \expc_Z\sqbr*{\Big(\Delta^\text{jump}_i(s)\Big)^2} + \alpha^{\sigma^n_i(s^-)}_n \bigparens{s^-, x^n_i(s^-)} \parens*{\Delta^\text{flip}_i(s)}^2}  \di s.
     \end{align*}
    Here, for $i\in [n]$ and $s\in [0, t]$,
    \begin{align*}
        \Delta^\text{jump}_i(s) := f\bigparens{x^n_i(s^-) + Z, \sigma^n_i(s^-)} - f\bigparens{x^n_i(s^-), \sigma^n_i(s^-)},
    \end{align*}
    and
    \begin{align*}
        \Delta^\text{flip}_i(s) := f\bigparens{x^n_i(s^-), t\big(\sigma^n_i(s^-)\big)} - f\bigparens{x^n_i(s^-), \sigma^n_i(s^-).}
    \end{align*}
     Since $f$ is $1$-Lipschitz, we know $\abs*{\Delta^\text{flip}_i} \leq 1$. In addition, $\expc_Z\sqbr*{\Big(\Delta^\text{jump}_i(s)\Big)^2} \leq \expc\sqbr*{(Z+1)^2}$, and from Assumption \ref{assumption0}, we have $\expc\parens*{Z^2} < \infty$. Also, since $v = v_++v_-$ and $\alpha = \alpha^+ + \alpha^- + \lnorm*{\varphi}_\infty$ are the (uniform) upper bounds of $v_{\sigma^n_i(s)}$ and $\alpha^{\sigma^n_i(s)}_n$, respectively, we have
      \begin{align*}
        \inner*{M_n}(t) \leq \int\limits_0^t \dfrac{1}{n\displaystyle^2} \dsum_{i=1}^n \parens*{v \expc\sqbr*{(Z+1)^2} + \alpha}\di s = \dfrac{t}{n}\parens*{v \expc\sqbr*{(Z+1)^2} + \alpha},
     \end{align*}
    so for each (fixed) $t \geq 0$,
    \begin{align*}
        \expc\sqbr*{M^2_n(t)} = \expc\sqbr*{\inner*{M_n}(t)} \leq \dfrac{t}{n}\parens*{v \expc\sqbr*{(Z+1)^2} + \alpha} \toinf{n} 0,
    \end{align*}
    as desired.
\end{proof}

Now, we prove a result that provides an alternate characterization of the weak limit of the sequence of ``martingale'' terms. This, along with the previous result, will characterize any subsequential limit $\mu(\cdot)$ of $\curlbr*{\mu_n(\cdot)\!: n\in\nat}$ as a solution to the fluid limit equation $A_{t, f}\big(\mu(\cdot)\big) = 0$ for $t \geq 0$ and $f \in \HHH$.
\begin{theorem}
\label{martingale2}
    Suppose that Assumptions \ref{assumption0}, \ref{assumption2}, \ref{assumption3}, and \ref{assumption4} hold, which implies that the sequence $\{\mu_n(\cdot) : n \in \mathbb{N}\}$ is $\CCC$-tight on $\mathscr{D}\big([0, \infty), \mathscr{P}_1(\real\times \{+,-\})\big)$, and that there exists a subsequence $\curlbr*{\mu_{n_k}(\cdot)\!: k\in\nat}$ that converges weakly to a random measure-valued path $\mu(\cdot)$ in $\mathscr{C}\big([0, \infty), \mathscr{P}_1(\real\times \{+,-\})\big)$. Then for any $t\geq 0$ and $f\in \HHH$,
    \begin{align*}
        A_{t, f}\big(\mu_{n_k}(\cdot)\big) \cdinf{k} A_{t, f}\big(\mu(\cdot)\big)
    \end{align*}
    in $\real$.
\end{theorem}

Before proving this result, we give the following ``uniform continuity'' result that will be crucial to the proof:
 \begin{lemma}
    \label{uniformLemma}
    Suppose that Assumptions \ref{assumption0}, \ref{assumption2}, \ref{assumption3}, and \ref{assumption4} hold.
    \begin{enumerate}[label = (\alph*)]
        \item \label{preLimit} For each $T\geq 0$ and $\varepsilon > 0$,
        \begin{align*}
            \lim_{\eta\to 0}\limsup_{n\to\infty} \prob\parens*{\sup_{s\in[0, T]} \sup_{x\in\real} \mu^\circ_n\bigparens{s, (x-\eta, x+\eta)} > \varepsilon} = 0.
        \end{align*}

        \item \label{limiting} If $\mu(\cdot)$ is a subsequential weak limit  of $\curlbr*{\mu_n(\cdot)\!: n\in\nat}$, then $\mu\in\TTT$ almost surely.
    \end{enumerate}
\end{lemma}

\begin{proof}
    In what follows, we present the proof of part \ref{preLimit}, and the proof of part \ref{limiting} will follow from that of part \ref{preLimit}.
    \begin{enumerate}[label = (\alph*)]
        \item Let $\curlbr*{Z_i(\ell)\!: i\in[n], \sspa \ell\in\nat}$ be an array of independent copies of $Z\sim J$, and suppose that the $Z_i(\ell)$'s are independent of $\FFF(0)$ as well. Define the embedded walk $\curlbr*{S_i(k)\!: k\in\nat_0}$, where $S_i(0) = x^n_i(0)$, and for each $k\in\nat$,
        \begin{align*}
            S_i(k) := x^n_i(0) + \dsum_{\ell=1}^k Z_i(\ell).
        \end{align*}
        Also, define $\curlbr*{S'_i(k)\!: k\in\nat_0}$, where $S'_i(0) = 0$, and for $k\in\nat$,
        \begin{align*}
            S'_i(k):= S_i(k) - x^n_i(0) = \dsum_{\ell=1}^k Z_i(\ell).
        \end{align*}
        Observe that $\curlbr*{S'_i(k)\!: k\in\nat_0}$ defines a sequence of renewal epochs, with the corresponding renewal measure
        \begin{align*}
            \mu_J(A) := \dsum_{k=1}^\infty \prob\big(S'_1(k) \in A\big)
        \end{align*}
        for $A\in \borel\bigparens{[0, \infty)}$. By Assumption~\ref{assumption4}, $J$ has a bounded density and hence is spread out. Moreover, by Assumption \ref{assumption0}, $\expc(Z)=1$. Thus, by the Renewal Theorem and Stone's decomposition, $\mu_J$ admits a bounded density $u$ on $[0, \infty)$ \big(see \cite[366--368]{F91} and \cite{Sto66}\big). Then with $\omega_J(\eta) := \sup_{b\in\real} \mu_J\bigparens{(b-\eta, b+\eta)}$, since $0\leq \omega_J(\eta) \leq 2\lnorm*{u}_\infty \eta$, we have
        \begin{align}
        \label{bddDensity}
            \omega_J(\eta) \xrightarrow{\sspa \eta\downarrow 0 \sspa} 0.
        \end{align}
         Now, observe that for any $z \in \real$,
        \begin{align}
        \label{trace1}
            \sup_{s\geq 0} \mu^\circ_n\bigparens{s, (z-2\eta, z+2\eta)} \leq \dfrac{1}{n} \dsum_{i=1}^n \ind_{\curlbr*{x^n_i(0) \in (z-2\eta, z+2\eta)}} + \dfrac{1}{n} \dsum_{i=1}^n \ind_{\curlbr*{\exists k\in\nat: S_i(k) \in (z-2\eta, z+2\eta)}}.
        \end{align}
          We treat each term on the right-hand side of \eqref{trace1} separately. First, consider the second term. Conditional on the initial profile, the summands of the second term of the right-hand side of \eqref{trace1} are independent across the $i$'s, with expectation given by
        \begin{align*}
            p_{i,z} &:= \prob\parens*{\exists k\in\nat\!: S_i(k) \in (z-2\eta, z+2\eta) \big|\FFF(0)} \\
            &\leq \dsum_{k=1}^\infty \prob\bigparens{S'_i(k)\in \parens*{z-x^n_i(0)-2\eta, z-x^n_i(0)+2\eta} \big|\FFF(0)} \\
            &= \mu_J\bigparens{\parens*{z-x^n_i(0)-2\eta, z-x^n_i(0)+2\eta}} \leq \omega_J(2\eta).
       \end{align*}
          Now, let $\varepsilon > 0$ be chosen, and   let $\eta \in \parens*{0, \dfrac{1}{2}}$ be chosen so that $\omega_J(2\eta) \leq \dfrac{\varepsilon}{4}$. Then by the (conditional) Hoeffding's Inequality,
        \begin{align}
        \label{condHoeffding}
            &\prob\parens*{\dfrac{1}{n} \dsum_{i=1}^n \ind_{\curlbr*{\exists k\in\nat: S_i(k) \in (z-2\eta, z+2\eta)}} > \dfrac{\varepsilon}{2} \Bigg|\FFF(0)} \nonumber \\
            &\qquad \leq \prob\parens*{\dfrac{1}{n} \dsum_{i=1}^n \ind_{\curlbr*{\exists k\in\nat: S_i(k) \in (z-2\eta, z+2\eta)}} - \dfrac{1}{n} \dsum_{i=1}^n p_{i,z} > \dfrac{\varepsilon}{4} \Bigg|\FFF(0)} \leq \exp\parens*{-\dfrac{n\varepsilon^2}{8}},
       \end{align}
        and after taking expectations, the same unconditional bound holds.

        Next, for each $n\in\nat$   and each $\eta > 0$, let $\ell(n)\in\nat$, with $\ell(n) \leq \dfrac{2n}{\eta} + 1$, be chosen so that there exist   $-n = z_0 < z_1 < \dots < z_{\ell(n)} = n$ that satisfy $z_j - z_{j-1} \leq \eta$ for each $j\in [\ell(n)]$. First, suppose that $x \in [-n,n]$. Then there exists some $j \in \{0,1,\ldots,\ell(n)\}$ such that
        \begin{align*}
            (x-\eta, x+\eta) \subseteq (z_j-2\eta, z_j+2\eta).
       \end{align*}
          Thus, by \eqref{trace1}, \eqref{condHoeffding}, and the union bound, with $\eta \in \parens*{0, \dfrac{1}{2}}$ being chosen so that $\omega_J(2\eta) \leq \dfrac{\varepsilon}{4}$, we have
       \begin{align}
        \label{trace2}
            \prob\parens*{\sup_{s\geq 0}\sup_{x\in[-n,n]} \mu_n^\circ\bigparens{s,(x-\eta,x+\eta)} > \varepsilon} &\leq \prob\parens*{\sup_{z\in\real} \mu_n^\circ\bigparens{0,(z-2\eta,z+2\eta)} >\frac{\varepsilon}{2}} \nonumber \\
            &\qquad + \dsum_{j=0}^{\ell(n)} \prob\parens*{\dfrac{1}{n}\sum_{i=1}^n \ind_{\curlbr*{\exists k\in\nat: S_i(k)\in(z_j-2\eta,z_j+2\eta)}} > \dfrac{\varepsilon}{2}} \nonumber \\
            &\leq \prob\parens*{\sup_{z\in\real} \mu_n^\circ\bigparens{0,(z - 2\eta,z + 2\eta)} > \dfrac{\varepsilon}{2}} \nonumber \\
            &\qquad + \parens*{\dfrac{2n}{\eta} + 2} \exp\parens*{-\dfrac{n\varepsilon^2}{8}}.
       \end{align}
          Now, suppose that $x\notin [-n,n]$, and observe that $n \geq 1 > 2\eta$. If $x > n$, then $(x - \eta, x + \eta) \subseteq \parens*{\dfrac{n}{2}, \infty}$, while if $x < -n$, then $(x - \eta, x + \eta) \subseteq \parens*{-\infty, -\dfrac{n}{2}}$. Furthermore, since every spatial trajectory is non-decreasing,
        \begin{align}
        \label{trace3}
           &\prob\parens*{\sup_{s \in [0,T]} \sup_{x \notin [-n,n]} \mu_n^\circ\bigparens{s,(x-\eta, x+\eta)} > \varepsilon} \nonumber\\
            &\qquad \leq \prob\parens*{\mu_n^\circ\parens*{T,\parens*{\dfrac{n}{2}, \infty}} > \varepsilon} + \prob\parens*{\mu_n^\circ\parens*{0,\parens*{-\infty, -\dfrac{n}{2}}} > \varepsilon}.
        \end{align}
        From Assumption \ref{assumption2}, we know that the second term on the right-hand side of \eqref{trace3} vanishes as $n$ grows, so it remains to bound the first term. To that end, let $L>0$ be chosen, and observe that for sufficiently large $n$,
       \begin{align*}
            \curlbr*{\mu_n^\circ\parens*{T, \parens*{\dfrac{n}{2},\infty}} > \varepsilon} \subseteq \curlbr*{\int\limits |y| \ind_{\curlbr*{|y|\geq L}} \mu_n^\circ(T, \di y) > 1};
       \end{align*}
        hence,
       \begin{align}
        \label{newUI}
            \limsup_{n\to\infty} \prob\parens*{\mu_n^\circ\parens*{T, \parens*{\dfrac{n}{2},\infty}} > \varepsilon} \leq \limsup_{n\to\infty} \prob\parens*{\int\limits |y| \ind_{\curlbr*{|y|\geq L}} \mu_n^\circ(T, \di y) > 1}.
       \end{align}
        Since the left-hand side of \eqref{newUI} does not depend on $L$, from Theorem \ref{tight2}, taking the limit of both sides of \eqref{newUI} as $L$ tends to infinity gives
       \begin{align*}
            \limsup_{n\to\infty} \prob\parens*{\mu_n^\circ\parens*{T, \parens*{\dfrac{n}{2},\infty}} > \varepsilon} = 0,
        \end{align*}
          which, together with \eqref{trace3}, implies that
        \begin{align}
        \label{trace4}
            \limsup_{n\to\infty} \prob\parens*{\sup_{s\in[0,T]} \sup_{x \notin [-n,n]} \mu_n^\circ\bigparens{s,(x-\eta,x+\eta)} > \varepsilon} = 0.
        \end{align}
        Lastly, for each sufficiently small (fixed) $\eta \in \parens*{0, \dfrac{1}{2}}$, chosen so that $\omega_J(2\eta) \leq \dfrac{\varepsilon}{4}$, from \eqref{trace2} and \eqref{trace4},
        \begin{align*}
            &\limsup_{n\to\infty} \prob\parens*{\sup_{s\in[0,T]}\sup_{x\in\real} \mu_n^\circ\bigparens{s,(x-\eta, x + \eta)} > \varepsilon} \\
            &\qquad \leq \limsup_{n\to\infty} \prob\parens*{\sup_{z\in\real} \mu_n^\circ\bigparens{0,(z-2\eta, z+2\eta)} >\frac{\varepsilon}{2}},
        \end{align*}
        and letting $\eta$ tend to zero and invoking Assumption \ref{assumption3} give us part \ref{preLimit}.

        \item By appealing to the Skorohod Representation Theorem, without loss of generality, we may assume that $\mu_{n_k}(\cdot) \casinf{k} \mu(\cdot)$ on $\mathscr{D}\big([0, \infty), \mathscr{P}_1(\real\times \{+,-\})\big)$.   Since the limiting path is continuous, this convergence is uniform
        on every compact time interval. In particular, for every $T > 0$,
        \begin{align*}
            \sup_{s\in[0,T]} \WWW_1\bigparens{\mu_{n_k}(s),\mu(s)} \casinf{k} 0.
        \end{align*}

        Next, let $T>0$, $\eta>0$, $s\in[0,T]$, and $x\in\real$ be chosen. Since the set $(x-\eta,x+\eta)\times\{+,-\}$ is open in $\real \times \{+,-\}$, from the Portmanteau Theorem, almost surely,
        \begin{align*}
            \mu^\circ\bigparens{s,(x-\eta,x+\eta)} \leq \liminf_{k\to\infty} \mu_{n_k}^\circ\bigparens{s,(x-\eta,x+\eta)},
        \end{align*}
        which implies that, almost surely,
        \begin{align}
            \sup_{s\in[0,T]} \sup_{x\in\real} \mu^\circ\bigparens{s,(x-\eta, x+\eta)} \leq \liminf_{k\to\infty} \sup_{s\in[0,T]} \sup_{x\in\real} \mu_{n_k}^\circ\bigparens{s,(x-\eta, x+\eta)}.
        \end{align}
        Therefore, for each $\varepsilon>0$, from Fatou's Lemma and our result in part \ref{preLimit},
        \begin{align*}
            &\prob\parens*{\sup_{s\in[0,T]} \sup_{x\in\real} \mu^\circ\bigparens{s,(x-\eta,x+\eta)} > \varepsilon} \\
            &\qquad \leq \liminf_{k\to\infty} \prob\parens*{\sup_{s\in[0,T]} \sup_{x\in\real} \mu_{n_k}^\circ\bigparens{s,(x-\eta,x+\eta)}
            >\varepsilon} \\
            &\qquad\leq \limsup_{n\to\infty} \prob\parens*{\sup_{s\in[0,T]} \sup_{x\in\real} \mu_n^\circ\bigparens{s,(x-\eta,x+\eta)} >\varepsilon} \xrightarrow{\sspa \eta\downarrow 0 \sspa} 0.
        \end{align*}

        For every fixed $T>0$, $\sup_{s\in[0,T]}\sup_{x\in\real} \mu^\circ\parens{s,(x-\eta, x+\eta)}$ is non-increasing as $\eta$ shrinks to $0$; thus, almost surely,
        \begin{align*}
            \lim_{\eta\downarrow 0} \sup_{s\in[0,T]} \sup_{x\in\real} \mu^\circ\bigparens{s,(x-\eta, x+\eta)}
        \end{align*}
        exists, and the preceding convergence in probability forces that limit to be zero almost surely. Intersecting the resulting probability-one events over $T \in\nat$, and then using monotonicity in $T$, we conclude that almost surely, for every $T > 0$,
        \begin{align*}
            \lim_{\eta\downarrow 0} \sup_{s\in[0,T]}\sup_{x\in\real} \mu^\circ\bigparens{s,(x-\eta, x+\eta)} = 0;
        \end{align*}
          that is, $\mu\in\TTT$ almost surely, as desired.
    \end{enumerate}
    That completes our proof of Lemma \ref{uniformLemma}.
\end{proof}
With that, we are ready to prove the relevant result, Theorem \ref{martingale2}:

\begin{proof}[Proof of Theorem \ref{martingale2}]
    Since $\curlbr*{\mu_n(\cdot)\!: n\in\nat}$ is $\mathscr{C}$-tight on $\mathscr{D}\big([0, \infty), \mathscr{P}_1(\real\times \{+,-\})\big)$, almost surely, any subsequential weak limit $\mu(\cdot)$ belongs to the class $\mathscr{C}\big([0, \infty), \mathscr{P}_1(\real\times \{+,-\})\big)$. Then by appealing to the Skorohod Representation Theorem (again), without loss of generality, we may assume that
    \begin{align*}
        \mu_{n_k}(\cdot) \casinf{k} \mu(\cdot)
    \end{align*}
     on $\mathscr{D}\big([0, \infty), \mathscr{P}_1(\real\times \{+,-\})\big)$, and that $\mu(\cdot) \in \mathscr{C}\big([0, \infty), \mathscr{P}_1(\real\times \{+,-\})\big)$.

    Let $t\geq 0$ and $f\in \HHH$ be chosen. First, observe that for each $t\geq 0$, the canonical projection
    \begin{align*}
        \pi_t\!: \mathscr{D}\big([0, \infty), \mathscr{P}_1(\real\times \{+,-\})\big) \to \mathscr{P}_1(\real\times \{+,-\}),
    \end{align*}
    given by $\pi_t\big(\nu(\cdot)\big) = \nu(t)$, is continuous at $\nu(\cdot)$ if and only if $\nu(\cdot)$ is continuous at $t$. In addition, by the Kantorovich-Rubinstein duality for the $1$-Wasserstein metric, for each $f\in \HHH$, the mapping
    \begin{align*}
     k_f\!: \mathscr{P}_1(\real\times \{+,-\}) \to\real,
    \end{align*}
    given by $k_f(\nu) = \inner*{f, \nu} = \int f\di\nu$, is continuous as well. Thus, for each $t\geq 0$,
      \begin{align}
        \label{weakLimit1}
        \biginner{f, \mu_{n_k}(t)} \casinf{k} \biginner{f, \mu(t)}.
     \end{align}
    Next, for convenience, for $f\in\HHH$, define
    \begin{align*}
        g_f(x, \pm) := \expc_Z\bigsqbr{f(x+Z, \pm) - f(x, \pm)} & & \text{and} & & h_f(x,\pm) := f(x, \pm) - f(x, \mp).
    \end{align*}
    Then it remains to show that for each $t\geq 0$ and for each $f\in\HHH$,
      \begin{align*}
        &\int\limits_0^t \Bigsqbr{\biginner{v_+ g_f(\cdot,+), \mu^+_{n_k}(s)}+\biginner{v_- g_f(\cdot,-), \mu^-_{n_k}(s)}} \di s \\
        &\qquad + \int\limits_0^t\Bigsqbr{\biginner{\alpha^+_{n_k}(s, \cdot) h_f(\cdot,-), \mu^+_{n_k}(s)}+\biginner{\alpha^-_{n_k}(s, \cdot) h_f(\cdot,+), \mu^-_{n_k}(s)}} \di s \cpinf{k} \\
        &\int\limits_0^t \Bigsqbr{\biginner{v_+ g_f(\cdot,+), \mu^+(s)}+\biginner{v_- g_f(\cdot,-), \mu^-(s)}}\di s \\
        &\qquad +\int\limits_0^t\Bigsqbr{\inner*{\alpha^+(s, \cdot) h_f(\cdot,-), \mu^+(s)}+\inner*{\alpha^-(s, \cdot) h_f(\cdot,+), \mu^-(s)}} \di s,
     \end{align*}
      where, along the path $\mu(\cdot)$, for each $t \geq 0$,
     \begin{align*}
        \alpha^\pm(t, x) := \alpha^\pm\bigparens{x, \mu(t)} = \alpha^\pm \mp \varphi\parens*{\int \ind_{\{z> x\}} \mu^\mp(t, \di z)}.
     \end{align*}
     To this end, first, observe that $f^\pm\!: \real\times \{+,-\} \to\real$, given by $f^\pm(x, \sigma) := \ind_{\{\sigma = \pm\}}$, are bounded and Lipschitz. Moreover, for each $f\in \HHH$, $g_f$ is bounded and Lipschitz as well. Thus, for each $f\in \HHH$, the functions $f^\pm g_f$ are bounded and Lipschitz.
    Hence, for each $f\in \HHH$ and $t\geq 0$,
      \begin{align*}
        \int\limits_0^t\inner*{v_\pm g_f, \mu^\pm_{n_k}(s)} \di s= \int\limits_0^t\inner*{v_\pm f^\pm g_f, \mu_{n_k}(s)} \di s\casinf{k} \int\limits_0^t\inner*{v_\pm f^\pm g_f, \mu(s)} \di s= \int\limits_0^t\inner*{v_\pm g_f, \mu^\pm(s)} \di s.
     \end{align*}
    Thus, if we can show that
      \begin{align}
      \label{convprob}
        &\int\limits_0^t\Bigsqbr{\biginner{\alpha^+_{n_k}(s, \cdot) h_f(\cdot,-), \mu^+_{n_k}(s)}+\biginner{\alpha^-_{n_k}(s, \cdot) h_f(\cdot,+), \mu^-_{n_k}(s)}} \di s \nonumber \\
        &\qquad\cpinf{k} \int\limits_0^t\Bigsqbr{\inner*{\alpha^+(s, \cdot) h_f(\cdot,-), \mu^+(s)}+\inner*{\alpha^-(s, \cdot) h_f(\cdot,+), \mu^-(s)}} \di s,
     \end{align}
    then our proof will be complete. This needs a more careful treatment, as $\alpha^\pm(s, \cdot)$ involves an indicator, and it brings a non-linear dependence of the dynamics on the driving measure-valued process.

    To show this, for $s \in [0,t]$, consider the $\WWW_1$-optimal coupling $\Pi_{n_k}(s,\di \xx, \di \yy)$ between the measures $\mu_{n_k}(s,\di \xx)$  and $\mu(s,\di \yy)$ on $\real \times \{+,-\}$. Note that
      \begin{align*}
        \Big|\inner*{\alpha^+_{n_k}(s, \cdot) h_f(\cdot,-), \mu^+_{n_k}(s)}-\inner*{\alpha^+_{n_k}(s, \cdot) h_f(\cdot,-), \mu^+(s)}\Big| \leq \int\Big|G_{n_k}(s,x) - G_{n_k}(s,y)\Big|\Pi_{n_k}(s, \di \xx, \di \yy),
     \end{align*}
    where $G_{n_k}(s,\cdot) := \alpha^+_{n_k}(s, \cdot) h_f(\cdot, -)f^+(\cdot)$. From the explicit forms of these functions, we can find finite positive constants $C_1$ and $C_2$, independent of $s$, such that for any $n_k\in\nat$ and for $\xx = (x, \sigma_1)$ and $\yy = (y, \sigma_2)$ in $\real \times \{+,-\}$ with $x\leq y$,
      \begin{align*}
        \abs*{G_{n_k}(s,\xx) - G_{n_k}(s,\yy)} \leq C_1 d_1(\xx, \yy) + C_2\mu_{n_k}\bigparens{s, [x,y] \times \{+,-\}},
     \end{align*}
    where, for $(z, \sigma)$ and $(z', \sigma')$ in $\real \times \{+,-\}$, $d_1\bigparens{(z,\sigma), (z',\sigma')} := |z - z'| + \ind_{\{\sigma \neq \sigma'\}}$. Then
      \begin{align}
      \label{coupling1}
        &\Big|\inner*{\alpha^+_{n_k}(s, \cdot) h_f(\cdot,-), \mu^+_{n_k}(s)} - \inner*{\alpha^+_{n_k}(s, \cdot) h_f(\cdot,-), \mu^+(s)}\Big| \nonumber \\
        &\qquad \leq C_1\int d_1(\xx, \yy) \Pi_{n_k}(s, \di \xx, \di \yy) + C_2\int \mu_{n_k}\bigparens{s, [x \wedge y, x \vee y] \times \{+,-\}} \Pi_{n_k}(s, \di \xx, \di \yy) \nonumber \\
        &\qquad = C_1\WWW_1\bigparens{\mu_{n_k}(s),\mu(s)} + C_2\int \mu_{n_k}\bigparens{s, [x \wedge y, x \vee y] \times \{+,-\}} \Pi_{n_k}(s, \di \xx, \di \yy).
     \end{align}
    To control the second term on the right-hand side of \eqref{coupling1}, observe that for any $\eta>0$,
      \begin{align*}
        &\int \mu_{n_k}(s, [x \wedge y, x \vee y]\times \{+,-\}) \Pi_{n_k}(s, \di \xx, \di \yy) \\
        &\qquad \leq \Pi_{n_k}\bigparens{s, \{(\xx,\yy): d_1(\xx, \yy) > \eta\}} +\! \int\limits_{\curlbr*{(\xx,\yy): d_1(\xx, \yy) \leq \eta}}\! \mu_{n_k}\bigparens{s, [x \wedge y, x \vee y]\!\times\! \{+,-\}} \Pi_{n_k}(s, \di \xx, \di \yy) \\
        &\qquad \leq \Pi_{n_k}\bigparens{s, \{(\xx,\yy)\!: d_1(\xx, \yy) > \eta\}} + \sup_{z\in\real} \mu_{n_k}\bigparens{s, (z-\eta, z+\eta) \times \{+,-\}} \\
        &\qquad \leq \dfrac{1}{\eta} \WWW_1(\mu_{n_k}(s),\mu(s)) +  \sup_{z\in\real} \mu_{n_k}\big(s, (z-\eta, z+\eta) \times \{+,-\}\big),
     \end{align*}
    where the inequality on the last line follows from Markov's Inequality. Thus, by integrating over $s \in [0,t]$,
      \begin{align*}
        &\int\limits_0^t \Big|\inner*{\alpha^+_{n_k}(s, \cdot) h_f(\cdot,-), \mu^+_{n_k}(s)}-\inner*{\alpha^+_{n_k}(s, \cdot) h_f(\cdot,-), \mu^+(s)}\Big| \di s \\
        &\qquad \leq \parens*{C_1 + \dfrac{C_2}{\eta}}\int\limits_0^t \WWW_1(\mu_{n_k}(s),\mu(s)) \di s + C_2 t \sup_{s\in[0, t]}\sup_{z\in\real} \mu_{n_k}\bigparens{s, (z-\eta, z+\eta) \times \{+,-\}}.
     \end{align*}
    Hence, for any $\varepsilon>0$,
      \begin{align}
      \label{coupling2}
        &\prob\parens*{\int\limits_0^t\Big|\inner*{\alpha^+_{n_k}(s, \cdot) h_f(\cdot,-), \mu^+_{n_k}(s)}-\inner*{\alpha^+_{n_k}(s, \cdot) h_f(\cdot,-), \mu^+(s)}\Big|\di s > \varepsilon} \nonumber \\
        &\qquad \leq \prob\parens*{\parens*{C_1 + \dfrac{C_2}{\eta}} \int\limits_0^t \WWW_1\bigparens{\mu_{n_k}(s),\mu(s)} \di s > \dfrac{\varepsilon}{2}} \nonumber \\
        &\qquad \qquad + \prob\parens*{\sup_{s\in[0, t]} \sup_{x\in\real} \mu_{n_k}\big(s, (x-\eta, x+\eta) \times \{+,-\}\big) > \dfrac{\varepsilon}{2C_2t}}.
     \end{align}
    Since $\mu_{n_k} \casinf{k} \mu$ on $\mathscr{D}\big([0, \infty), \PPP_1(\real\times \{+,-\})\big)$, that implies $\int\limits_0^t \WWW_1\bigparens{\mu_{n_k}(s),\mu(s)}\di s \casinf{k} 0$, and so the first term on the right-hand side of \eqref{coupling2} vanishes as $k$ grows. Moreover, from Lemma \ref{uniformLemma}\ref{preLimit}, the second term on the right-hand side of \eqref{coupling2} also vanishes upon first letting $k$ grow and then letting $\eta$ shrink to zero. As $\varepsilon>0$ is arbitrary, we conclude that
      \begin{align}
    \label{sameAlpha}
        \int\limits_0^t\Big|\inner*{\alpha^+_{n_k}(s, \cdot) h_f(\cdot,-), \mu^+_{n_k}(s)}-\inner*{\alpha^+_{n_k}(s, \cdot) h_f(\cdot,-), \mu^+(s)}\Big|\di s \cpinf{k} 0.
    \end{align}
    Next, by Lemma \ref{uniformLemma}\ref{limiting}, $\mu\in\TTT$ almost surely, and hence $\mu^\circ\bigparens{s,\{x\}} = 0$ for every $s\geq 0$ and $x\in\real$. Thus, $(x,\infty) \times \{-\}$ is a $\mu(s)$-continuity set, so the Portmanteau Theorem and the continuity of $\varphi$ imply that
    \begin{align*}
        \alpha_{n_k}^+(s,x) \casinf{k} \alpha^+(s,x).
    \end{align*}
    Hence, by the Dominated Convergence Theorem,
      \begin{align}
    \label{sameMeas}
        \int\limits_0^t\Bigabs{\inner*{\alpha^+_{n_k}(s, \cdot) h_f(\cdot,-), \mu^+(s)} -\inner*{\alpha^+(s, \cdot) h_f(\cdot,-), \mu^+(s)}}\di s \casinf{k} 0.
    \end{align}
    The Triangle Inequality, in conjunction with \eqref{sameAlpha} and \eqref{sameMeas}, gives
      \begin{align*}
       \int\limits_0^t\Bigabs{\biginner{\alpha^+_{n_k}(s, \cdot) h_f(\cdot,-), \mu^+_{n_k}(s)}-\biginner{\alpha^+(s, \cdot) h_f(\cdot,-), \mu^+(s)}} \di s \cpinf{k} 0.
    \end{align*}
    By a parallel argument,
      \begin{align*}
       \int\limits_0^t\Bigabs{\biginner{\alpha^-_{n_k}(s, \cdot) h_f(\cdot,+), \mu^-_{n_k}(s)}-\biginner{\alpha^-(s, \cdot) h_f(\cdot,+), \mu^-(s)}}\di s \cpinf{k} 0
    \end{align*}
    as well, which yields \eqref{convprob} and completes our proof.
\end{proof}
To tie the previous two results together, consider the following:
\begin{corollary}
\label{initialization}
    Suppose that Assumptions \ref{assumption0}, \ref{assumption2}, \ref{assumption3}, and \ref{assumption4} hold. Then any subsequential weak limit point $\mu(\cdot)$ of the sequence of empirical measures $\curlbr*{\mu_n(\cdot)\!: n\in\nat}$ in the Skorohod space $\DDD\big([0, \infty), \PPP_1(\real\times \{+,-\})\big)$ solves \eqref{originalMVE} almost surely. That is, with $\mu_{n_k}(\cdot) \cdinf{k} \mu(\cdot)$ in $\DDD\big([0, \infty), \PPP_1(\real\times \{+,-\})\big)$, we have, almost surely,
    \begin{align*}
        A_{t, f}\bigparens{\mu(\cdot)} = 0
    \end{align*}
   for every $t\geq 0$ and $f\in\HHH$.
\end{corollary}
\begin{proof}
    Let $t\geq 0$ and $f\in \HHH$ be chosen. Theorem \ref{martingale1} implies that for any subsequence $\curlbr*{\mu_{n_k}\!: k\in\nat}$, we have $A_{t, f}\bigparens{\mu_{n_k}(\cdot)} \cdinf{k} 0$, while Theorem \ref{martingale2} gives us that
    \begin{align*}
        A_{t, f}\bigparens{\mu_{n_k}(\cdot)} \cdinf{k} A_{t, f}\bigparens{\mu(\cdot)}
    \end{align*}
    in $\real$. By the uniqueness of the weak limit (see, for instance, \cite[14]{Billingsley99}), for each $t\geq 0$ and $f\in\HHH$, 
    \begin{align*}
        A_{t,f}\bigparens{\mu(\cdot)} \eqas 0.
    \end{align*}
    To extend this almost-sure equality simultaneously to every $t\geq 0$ and $f\in\HHH$, note that for every $c\in\real$, since $A_{t,f + c} \bigparens{\mu(\cdot)} = A_{t, f}\bigparens{\mu(\cdot)}$, it suffices to consider functions $f\in\HHH$ such that $f(0, +) = 0$. To finish this proof, apply the preceding argument to every $t\in\rational_+$ and $f$ in a countable and locally uniformly dense subset $\HHH_0$ of $\curlbr*{f\in\HHH\!: f(0, +) = 0}$, and intersect the resulting (countable) probability-one events. Continuity in $t$ extends the identity to every $t \geq 0$. Moreover, if $f_m \in \HHH_0$ converges locally uniformly to $f \in \HHH$, then $|f_m(x,\sigma)| \vee |f(x,\sigma)|\leq |x|+1$. Truncating to bounded spatial sets and using the finite first moments of $\mu(0)$ and $\mu(t)$, together with the assumption that $\expc(Z) = 1 < \infty$ for the jump term, therefore gives $A_{t,f_m}\bigparens{\mu(\cdot)} \toinf{m} A_{t,f} \bigparens{\mu(\cdot)}$, and our claim follows.
\end{proof}

\subsection{Uniqueness of Solution to the MVE}
The next step in our proof is to establish conditions under which the MVE admits a unique solution. The argument below proceeds in a cumulative-distribution-function metric, rather than directly in the dual Lipschitz ($\WWW_1$) metric, as manifested in the following result:
\begin{theorem}
\label{grownwallstuff}
    Suppose $\mu^1(\cdot)$ and $\mu^2(\cdot)$ are two elements of $\TTT$ that solve \eqref{originalMVE}, with initial conditions $\mu^1_0 = \mu^1(0)$ and $\mu^2_0 = \mu^2(0)$, respectively. For $i\in \{1,2\}$, define the type-wise (extended) cumulative distribution functions by $F_i^{\pm}(t,x) := \mu^{\pm,i}\bigparens{t, (-\infty, x]}$, where $\mu^{\pm,i}(t,A) := \mu^i\bigparens{t, A\times \{\pm\}}$ for $A\in \borel(\real)$. In addition, define
    \begin{align*}
        \Delta_{\pm}(s,x) := F_1^{\pm}(s,x) - F_2^{\pm}(s,x),
    \end{align*}
    and let $D_{\pm}(t) := \sup_{x\in \real} \abs*{\Delta_{\pm}(t, x)}$. Also, let $D(t) := D_+(t) + D_-(t)$. Then there exists some constant $c > 0$, depending only on $v_+, v_-, \alpha^+, \alpha^-$, and $\lnorm*{\varphi}_{\infty}$, such that for any $t \geq 0$,
    \begin{align}
    \label{babyGronwall}
        D(t) \leq D(0) \cdot \e^{ct}.
    \end{align}
    In particular, if $\mu^1(0)=\mu^2(0)$, then $\mu^1(t)=\mu^2(t)$ for every $t\geq 0$.
\end{theorem}

\begin{proof}
    Let $t\geq 0$ be chosen. For $i\in \{1,2\}$, let
      \begin{align*}
        \alpha^{\pm,i}(s,x) := \alpha^\pm\bigparens{x, \mu^i(s)} = \alpha^{\pm} \mp \varphi\parens*{\int \ind_{\{z>x\}}\mu^{\mp,i}(s, \di z)}.
     \end{align*}
    It suffices to show that there is a constant $c > 0$ such that
      \begin{align}
    \label{eq:preGronwallD}
        D(t) \leq D(0) + c\int\limits_0^t D(s) \di s,
    \end{align}
    from which \eqref{babyGronwall} follows from Gr\"onwall's Inequality.

    First, let $x\in \real$ be chosen, and let $\eta \in (0,1)$. Let $\psi_{x,\eta}\!: \real\to [0,1]$ be a non-increasing, $\dfrac{1}{\eta}$-Lipschitz function such that
    \begin{align*}
        \ind_{(-\infty, x-\eta]} \leq \psi_{x,\eta} \leq \ind_{(-\infty, x+\eta]}.
    \end{align*}
    Also, for each $\sigma\in \{+,-\}$, define $f^\sigma_{x,\eta}\!: \real\times \{+,-\} \to \real$ by
      \begin{align*}
        f_{x,\eta}^{\sigma}(y,\tau) := \psi_{x,\eta}(y)\ind_{\{\tau = \sigma\}}.
    \end{align*}
     Then since $\eta\cdot f_{x,\eta}^{\sigma}\in \HHH$, and since $\mu^i(\cdot)$ solves \eqref{originalMVE},
      \begin{align}
    \label{eq:MVEtestf}
        \biginner{f_{x,\eta}^{\sigma}, \mu^i(t)} &= \biginner{f_{x,\eta}^{\sigma}, \mu^i(0)} + \int\limits_0^t \Biginner{\LLL f_{x,\eta}^{\sigma}\bigparens{\cdot, \mu^i(s)}, \mu^i(s)} \di s.
    \end{align}
    Next, since $\mu^i\in \TTT$,  for each $t\geq0$,
    \begin{align*}
        \lim_{\eta\to 0} \sup_{s\in[0, t]} \sup_{x\in\real} \mu^i\bigparens{s, (x-\eta, x+\eta)\times \{+,-\}} = 0,
    \end{align*}
      which implies that for any $t\geq 0$,
    \begin{align*}
        \sup_{s\in[0,t]}\Bigabs{\biginner{f_{x,\eta}^{\sigma}, \mu^i(s)} - F_i^{\sigma}(s,x)} \xrightarrow{\sspa \eta\to 0 \sspa} 0.
    \end{align*}
    A similar estimate justifies passage to the limit in the jump and flip terms in \eqref{eq:MVEtestf}. Consequently, for each $x\in \real$ and $i\in \{1,2\}$, with $J(\di z)$ denoting the measure associated with the jump distribution $J$,
      \begin{align}
    \label{eq:FplusEquation}
        F_i^+(t,x) &= F_i^+(0,x) + v_+ \int\limits_0^t \int\limits_{[0,\infty)} \bigparens{F_i^+(s, x-z) - F_i^+(s, x)} J(\di z) \di s \nonumber \\
        &\qquad + \int\limits_0^t \int\limits_{(-\infty,x]} \alpha^{-,i}(s,y) \mu^{-,i}(s,\di y) \di s - \int\limits_0^t \int\limits_{(-\infty,x]} \alpha^{+,i}(s,y) \mu^{+,i}(s,\di y) \di s,
    \end{align}
    and similarly,
      \begin{align}
    \label{eq:FminusEquation}
        F_i^-(t,x) &= F_i^-(0,x) + v_- \int\limits_0^t \int\limits_{[0,\infty)} \bigparens{F_i^-(s, x-z) - F_i^-(s, x)} J(\di z) \di s \nonumber \\
        &\qquad + \int\limits_0^t \int\limits_{(-\infty,x]} \alpha^{+,i}(s,y) \mu^{+,i}(s,\di y) \di s - \int\limits_0^t \int\limits_{(-\infty,x]} \alpha^{-,i}(s,y) \mu^{-,i}(s,\di y) \di s.
    \end{align}
    Next, observe that for each $s \in [0,t]$,
      \begin{align*}
    &\sup_{x\in \real} \abs*{v_\pm \int\limits_{[0,\infty)} \Bigsqbr{\bigparens{F_1^\pm(s, x-z) - F_1^\pm(s, x)} - \bigparens{F_2^\pm(s, x-z) - F_2^\pm(s, x)}} J(\di z)} \\
    &\quad \leq (v_+ + v_-) \sup_{x\in \real} \abs*{\int\limits_{[0,\infty)} \bigsqbr{\Delta_{\pm}(s,x-z) - \Delta_{\pm}(s,x)}J(\di z)} \\
    &\quad \leq (v_+ + v_-)\sup_{x\in \real} \int\limits_{[0,\infty)} \Bigabs{\Delta_\pm(s,x-z) - \Delta_\pm(s,x)}J(\di z) \leq 2(v_+ + v_-) D_\pm(s).
    \end{align*}
    Hence, with $c_1 := 2(v_+ + v_-)$,
      \begin{align}
    \label{eq:jumpbound}
       &\sum_{\sigma \in \{+,-\}}\sup_{x\in \real} \abs*{v_\sigma \int\limits_0^t \int\limits_{[0,\infty)} \Bigsqbr{\bigparens{F_1^\sigma(s, x-z) - F_1^\sigma(s, x)} - \bigparens{F_2^\sigma(s, x-z) - F_2^\sigma(s, x)}} J(\di z) \di s} \nonumber \\
       &\qquad \leq c_1 \int\limits_0^t D(s) \di s,
    \end{align}
    which gives us an appropriate bound on the ``jump'' terms of \eqref{eq:FplusEquation} and \eqref{eq:FminusEquation}. It remains to bound the ``gain'' and/or ``loss'' terms (which we refer to as the ``flip'' terms) of such equations. We shall treat the ``loss'' term of \eqref{eq:FplusEquation}, with the understanding that the remaining ``flip'' terms can be bounded in a similar fashion. To that end, for $i\in \{1, 2\}$, let $G_i^{\pm}(s,x) := \int\limits_{(-\infty,x]} \alpha^{\pm,i}(s,y) \mu^{\pm,i}(s,\di y)$, and note that
      \begin{align*}
        G_1^+(s,x) - G_2^+(s,x) &= \int\limits_{(-\infty,x]} \alpha^{+,1}(s,y) \parens*{\mu^{+,1} - \mu^{+,2}}(s,\di y) \\
        &\qquad + \int\limits_{(-\infty,x]} \sqbr*{\alpha^{+,1}(s,y)-\alpha^{+,2}(s,y)} \mu^{+,2}(s, \di y) \\
        &= \int\limits_{(-\infty,x]} \alpha^{+,1}(s,y) \Delta_+(s,\di y) + \int\limits_{(-\infty,x]} \sqbr*{\alpha^{+,1}(s,y)-\alpha^{+,2}(s,y)} \mu^{+,2}(s, \di y) \\
        &=: I_1(s,x) + I_2(s,x).
    \end{align*}

    We first bound $I_1(s, x)$. Let $s\in[0, t]$ be chosen. Since $\mu^1(\cdot) \in \TTT$, it follows that the spatial marginal of $\mu^1(s)$ is atomless; thus, the function $\mu^{-,1}\bigparens{s, (y, \infty)}$ is continuous and non-increasing (in $y$), with total variation bounded above by $\mu^{-,1}(s, \real) \leq 1$. This, in turn, implies that $\alpha^{+,1}(s, y)$ is also a continuous function in $y$, with total variation also bounded above by $1$. Moreover, $\Delta_+(s, y)$ is a continuous function in $y$, with \(\lim_{y\to-\infty} \Delta_+(s,y) = 0\); thus, Stieltjes integration-by-parts gives
    \begin{align*}
        I_1(s, x) = \alpha^{+,1}(s, x) \Delta_+(s, x) - \int\limits_{(-\infty, x]} \Delta_+(s, y) \partial_y\alpha^{+,1}(s, y),
    \end{align*}
      where $\partial_y\alpha^{+,1}(s, y)$ denotes the signed Lebesgue-Stieltjes measure induced by the bounded-variation function $y \mapsto \alpha^{+,1}(s, y)$. Hence, with $\lnorm*{f}_\mathrm{TV}$ denoting the total variation of a real-valued function $f$,
     \begin{align}
    \label{eq:newI1bound}
        \abs*{I_1(s, x)} &\leq \lnorm*{\alpha^{+,1}(s, \cdot)}_\infty D_+(s) + D_+(s) \lnorm*{\alpha^{+,1}(s, \cdot)}_\mathrm{TV} \leq \parens*{\alpha^+ + \lnorm*{\varphi}_\infty + 1} D_+(s),
    \end{align}
    and since the last expression in \eqref{eq:newI1bound} does not depend on $x\in\real$, we have
      \begin{align}
    \label{eq:I1bound}
        \sup_{x\in\real} |I_1(s, x)| \leq \parens*{\alpha^+ + \lnorm*{\varphi}_\infty + 1} D_+(s).
    \end{align}
    Next, we bound $I_2(s, x)$. Note that for $i\in \{1, 2\}$, $s\in [0, \infty)$, and $y\in\real$,
      \begin{align*}
        \alpha^{+, i}(s, y) = \alpha^+ - \varphi\Bigparens{\mu^{-,i}(s, \real)- \mu^{-,i}\bigparens{s, (-\infty, y]}} = \alpha^+ - \varphi\Bigparens{\mu^{-,i}\bigparens{s, (y, \infty)}},
    \end{align*}
    and again, since $\varphi$ is $1$-Lipschitz,
    \begin{align*}
        \abs*{\alpha^{+,1}(s,y)-\alpha^{+,2}(s,y)} &\leq \abs*{\mu^{-,1}\bigparens{s,(y,\infty)} - \mu^{-,2}\bigparens{s,(y,\infty)}} \\
        &= \abs*{\mu^{-,1}(s,\real) - \mu^{-,2}(s,\real) - \Delta_-(s,y)} \\
        &\leq \abs*{\mu^{-,1}(s,\real) - \mu^{-,2}(s,\real)} + \abs*{\Delta_-(s,y)}.
    \end{align*}
    Clearly, $\abs*{\mu^{-,1}(s,\real) - \mu^{-,2}(s,\real)} = \lim_{x\to\infty} \abs*{\Delta_-(s, x)} \leq D_-(s)$, and since $\abs*{\Delta_-(s,y)} \leq D_-(s)$, by the Triangle Inequality, we have $\Bigabs{\alpha^{+,1}(s,y)-\alpha^{+,2}(s,y)} \leq 2D_-(s)$. Since $\mu^{+,2}(s,\real)\leq 1$, it follows that
    \begin{align}
    \label{eq:I2bound}
        \sup_{x\in \real}\bigabs{I_2(s,x)} \leq 2D_-(s).
    \end{align}
    Combining \eqref{eq:I1bound} and \eqref{eq:I2bound}, we see that
      \begin{align}
    \label{eq:Gplusbound}
        \sup_{x\in \real}\Bigabs{G_1^+(s,x) - G_2^+(s,x)} &\leq \sup_{x\in\real} \abs*{I_1(s, x)} + \sup_{x\in\real} \abs*{I_2(s, x)} \nonumber \\
        &\leq \parens*{\alpha^+ + \lnorm*{\varphi}_\infty + 1} D_+(s) + 2D_-(s) \nonumber \\
        &\leq \parens*{\alpha^+ + \lnorm*{\varphi}_\infty + 3} \bigparens{D_+(s) + D_-(s)} = \parens*{\alpha^+ + \lnorm*{\varphi}_\infty + 3} D(s).
    \end{align}
    By a parallel argument,
    \begin{align}
    \label{eq:Gminusbound}
        \sup_{x\in \real}\Bigabs{G_1^-(s,x) - G_2^-(s,x)} \leq \parens*{\alpha^- + \lnorm*{\varphi}_\infty + 3} D(s).
    \end{align}
    Now, note that in subtracting \eqref{eq:FplusEquation} for $i = 1$ from \eqref{eq:FplusEquation} for $i=2$, the ``flip'' contribution is
      \begin{align}
    \label{flipContribute1}
        H^+(t, x) := \int\limits_0^t \Bigsqbr{\bigparens{G^+_1(s, x) - G^+_2(s, x)} - \bigparens{G^-_1(s, x) - G^-_2(s, x)}} \di s,
    \end{align}
    and in subtracting \eqref{eq:FminusEquation} for $i = 1$ from \eqref{eq:FminusEquation} for $i=2$, the ``flip'' contribution is
      \begin{align}
    \label{flipContribute2}
        H^-(t, x) := \int\limits_0^t \Bigsqbr{\bigparens{G^-_1(s, x) - G^-_2(s, x)} - \bigparens{G^+_1(s, x) - G^+_2(s, x)}} \di s = -H^+(t, x),
    \end{align}
    so it suffices to bound, say, $H^+(t, x)$. By applying the Triangle Inequality and taking supremum over all $x\in\real$ of both sides of \eqref{flipContribute1}, we obtain
      \begin{align*}
        \sup_{x\in\real} \abs*{H^+(t, x)} &\leq \int\limits_0^t \Bigsqbr{\sup_{x\in\real} \bigabs{G^+_1(s, x) - G^+_2(s, x)} + \sup_{x\in\real} \bigabs{G^-_1(s, x) - G^-_2(s, x)}} \di s \\
        &\leq  \parens*{\alpha^+ + \alpha^- + 2\lnorm*{\varphi}_\infty + 6} \int\limits_0^t D(s) \di s,
    \end{align*}
    and naturally, $\sup_{x\in\real} \abs*{H^-(t, x)} \leq \parens*{\alpha^+ + \alpha^- + 2\lnorm*{\varphi}_\infty + 6} \int\limits_0^t D(s) \di s$ as well. Thus, the overall contribution of the ``flip'' terms over \eqref{eq:FplusEquation} and \eqref{eq:FminusEquation} can be bounded by a multiple of $D(s)$; more specifically,
      \begin{align}
    \label{eq:flipbound}
        \sup_{x\in\real} |H^+(t, x)| + \sup_{x\in\real} |H^-(t, x)| \leq 2 \parens*{\alpha^+ + \alpha^- + 2\lnorm*{\varphi}_\infty + 6} \int\limits_0^t D(s) \di s.
    \end{align}
    Finally, let $c_2 = 2 \parens*{\alpha^+ + \alpha^- + 2\lnorm*{\varphi}_\infty + 6}$. Then by subtracting \eqref{eq:FplusEquation} for $i=2$ from \eqref{eq:FplusEquation} for $i=1$, subtracting \eqref{eq:FminusEquation} for $i=2$ from \eqref{eq:FminusEquation} for $i=1$, taking supremum in $x$, applying \eqref{eq:jumpbound} and \eqref{eq:flipbound}, and combining the two resulting inequalities, we obtain
      \begin{align*}
        D(t) \leq D(0) + (c_1+c_2)\int\limits_0^t D(s) \di s,
      \end{align*}
    which is exactly \eqref{eq:preGronwallD} with $c:=c_1+c_2$. Applying Gr\"onwall's Inequality gives us that for each $t\geq 0$,
    \begin{align*}
        D(t) \leq D(0)\cdot \e^{ct}.
    \end{align*}
    In particular, if $\mu^1(0)=\mu^2(0)$, then $D(0)=0$, so for each $t\geq 0$, $D(t)=0$ as well. That is, for each $t\geq 0$, $x\in\real$, and $\sigma \in \{+,-\}$, $F_1^{\sigma}(t,x) = F_2^{\sigma}(t,x)$, or equivalently, for each $\sigma\in \{+,-\}$ and $t\geq 0$, $\mu^{\sigma,1}(t)=\mu^{\sigma,2}(t)$, which implies that for each $t\geq 0$, $\mu^1(t)=\mu^2(t)$. That completes our proof.
\end{proof}
  To tie everything together, we present the proof of Theorem \ref{fluidTheorem}:
\begin{proof}[Proof of Theorem \ref{fluidTheorem}]
    From Corollary \ref{tightD}, we know that the sequence $\curlbr*{\mu_n(\cdot)\!: n\in\nat}$ is $\CCC$-tight in the space $\DDD\big([0, \infty), \PPP_1(\real\times \{+,-\})\big)$, which proves part \ref{31a}. Also, from Corollary \ref{initialization}, each subsequential limit point $\mu(\cdot)$ of $\curlbr*{\mu_n(\cdot)\!: n\in\nat}$ must be a solution to \eqref{originalMVE}, and it must lie in $\TTT$ by Lemma \ref{uniformLemma}. This proves part \ref{31b}. Part \ref{31c} follows from part \ref{31a}, the assumed convergence of $\mu_n(0)$ to a deterministic $\nu$, and the uniqueness of solution to \eqref{originalMVE} in the class $\TTT$ proved in Theorem \ref{grownwallstuff}.
\end{proof}

 \subsection{Propagation of Chaos}
 \begin{proof}[Proof of Corollary \ref{POC}]
    We first show that Assumptions \ref{assumption2} and \ref{assumption3} hold under the assumptions of this result. First, since $\Gamma \in \PPP_1(\real\times \{+,-\})$, Assumption \ref{assumption2} follows from Markov's Inequality, while Assumption \ref{assumption3} follows from the continuity (and hence uniform continuity) of the spatial cumulative distribution function of $\Gamma$, the Glivenko--Cantelli Theorem, and the fact that independent, identically distributed locations with an atomless distribution are almost surely tie-free.

    In addition, since $\curlbr*{\XX^n_i(0)\!: i\in[n]}$ are independent and identically distributed with deterministic distribution $\Gamma$, $\mu_n(0) \cpinf{n} \Gamma$ in $\PPP_1(\real\times \{+,-\})$. Then the first assertion is established by applying Theorem \ref{fluidTheorem}\ref{31c}. 

    Also, as noted before the proof of Theorem \ref{martingale1}, the spatial configuration is almost surely tie-free at any time. Hence, the tie-breaking mechanism is invoked only on a null set; on its complement, the dynamics are permutation-equivariant. Since the initial coordinates are independent and identically distributed, $\Law\bigparens{\XX^n(t)}$ is symmetric. Then the second assertion follows from the first assertion and Lemma 3.19 in \cite{CD22}.
\end{proof}
 \section{Long-Time Behavior: Traveling Wave Solution}

Now, we embark on an investigation of the long-time behavior of the MVE \eqref{originalMVE} by obtaining traveling wave solutions. We break the discussion into several smaller parts.

\subsection{An Integro-Differential Formulation of the Traveling Wave Equation}
The following lemma gives a reformulation of the traveling wave equation derived from \eqref{originalFluid} as a  non-linear, non-local integro-differential   equation:
\begin{lemma}
\label{integrodiff}
    Suppose that a traveling wave solution to \eqref{originalFluid} exists. Then the associated traveling wave densities, defined as $\rho_\pm(z) := F_\pm'(z)$, satisfy
    \begin{align}
    \label{newfluid}
        v_\pm \int\limits_{-\infty}^z \phi(z-y) \rho_\pm(y) \di y = \parens*{v_\pm + \beta^\pm(z)} \rho_\pm(z) - \beta^\mp(z) \rho_\mp(z) - \gamma \rho'_\pm(z)
    \end{align}
     for every $z\in\real$, where $\phi$ denotes the density of $J$, and for each $z\in\real$, we define
      \begin{align*}
        \beta^+(z) &:= \alpha^+ - \int\limits_z^\infty \rho_-(x)\di x = \alpha^+ - \rho^- + F_-(z), \\
        \intertext{and}
        \beta^-(z) &:= \alpha^- + \int\limits_z^\infty \rho_+(x)\di x = \alpha^- + \rho^+ - F_+(z).
    \end{align*}
\end{lemma}
\begin{proof}
Let $f\!: \real\to\real$ be a continuously differentiable function with compact support, and let
\begin{align*}
    C_f := \max\curlbr*{1, \lnorm*{f}_\infty, \lnorm*{f'}_\infty}.
\end{align*}
Next, define $f_+(x, \sigma) = \dfrac{f(x)}{C_f} \ind_{\{\sigma = +\}}$ so that $f_+ \in \HHH$. Then
\begin{align*}
    C_f \inner*{f_+, \mu(t)} = \int\limits_\real f(x) \mu^+(t, \di x) = \inner*{f, \mu^+(t)};
\end{align*}
furthermore, since $H_\pm \in \CCC^2$, the maps $t \mapsto \inner*{f, \mu^\pm(t)}$ are continuously differentiable, while the integrand in \eqref{originalFluid} is continuous in $t$. Then, from \eqref{originalFluid}, we obtain
  \begin{align}
\label{newerFuild1}
    &\dfrac{\di}{\di t} \int f(x)\mu^+(t, \di x) \nonumber \\
    &\qquad = v_+ \int \expc[f(x\!+\!Z)\!-\!f(x)] \mu^+ (t, \di x) - \int\alpha^+(t, x) f(x) \mu^+(t, \di x)+ \int\alpha^-(t, x) f(x) \mu^-(t, \di x).
\end{align}
Next, using $\mu^\pm(t, \di x) = \rho_\pm(t, x) \di x$ and the fact that $J$ has density $\phi$, \eqref{newerFuild1} becomes
  \begin{align}
\label{newerFuild2}
    \dfrac{\di}{\di t} \int\limits_\real f(x) \rho_+(t, x)\di x &= v_+ \int \parens*{\int \bigparens{f(x+z) - f(x)} \phi(z) \di z} \rho_+(t, x) \di x \nonumber \\
    &\qquad - \int \alpha^+(t, x) f(x) \rho_+(t, x) \di x + \int \alpha^-(t, x) f(x) \rho_-(t, x) \di x.
\end{align}
Let $y = x+z$ (so $x = y - z$). By Fubini's Theorem, observe that
\begin{align*}
    \int \int f(x+z) \phi(z) \rho_+(t, x) \di z \di x &= \int \int f(y) \phi(z) \rho_+(t, y-z) \di y \di z \\
    &= \int f(y) \parens*{\int \phi(z) \rho_+(t, y - z) \di z} \di y \\
    &=  \int f(x) \parens*{\int \phi(z) \rho_+(t, x - z) \di z} \di x;
\end{align*}
thus,
  \begin{align*}
    \int \parens*{\int \bigparens{f(x+z) - f(x)} \phi(z) \di z} \rho_+(t, x) \di x = \int f(x) \sqbr*{\parens*{\int \phi(z) \rho_+(t, x-z) \di z}- \rho_+(t, x)} \di x.
\end{align*}
In addition, since $\dfrac{\di}{\di t} \int f(x) \rho_+(t, x) \di x = \int f(x) \partial_t \rho_+(t, x) \di x$, \eqref{newerFuild2} is equivalent to
  \begin{align}
\label{newerFluid3}
    \int f(x) A(t, x) \di x = 0,
\end{align}
where
  \begin{align*}
    A(t, x) = \partial_t \rho_+(t, x) - v_+\parens*{\int \phi(z) \rho_+(t, x-z) \di z-\rho_+(t, x)} + \alpha^+(t, x) \rho_+(t, x) - \alpha^-(t, x) \rho_-(t, x).
\end{align*}
The function $A(t, \cdot)$ is continuous, so \eqref{newerFluid3} implies $A(t, x) = 0$ for every $x\in\real$; that is,
  \begin{align}
\label{newerFluid4}
    \partial_t \rho_+(t, x) = v_+\parens*{\int \phi(z) \rho_+(t, x-z) \di z - \rho_+(t, x)} - \alpha^+(t, x) \rho_+(t, x) + \alpha^-(t, x) \rho_-(t, x).
\end{align}
Now, suppose that the traveling wave solution with speed $\gamma$ exists. Then for any $x\in\real$,
\begin{align*}
    \rho_\pm(t, x) = \rho_\pm(x-\gamma t).
\end{align*}
Letting $z := x - \gamma t$, we have $\partial_t \rho_+(t, x) = -\gamma \rho'_+(z)$, and with this, we have $\alpha^\pm(t, x) = \beta^\pm(z)$. Thus, \eqref{newerFluid4} becomes
  \begin{align*}
    -\gamma \rho'_+(z) = v_+ \parens*{\int \phi(y) \rho_+(z-y) \di y - \rho_+(z)} - \beta^+(z) \rho_+(z) + \beta^-(z) \rho_-(z),
\end{align*}
and by rearranging the terms, we get
\begin{align*}
    v_+ \int \phi(y) \rho_+(z-y) \di y = \bigparens{v_+ + \beta^+(z)} \rho_+(z) - \beta^-(z) \rho^-(z) - \gamma \rho'_+(z),
\end{align*}
as stated in \eqref{newfluid}. The other equation from \eqref{newfluid} is obtained via a parallel argument.
\end{proof}

With that, we are now ready to prove the first part of Theorem \ref{travelwave}:
\begin{proof}[Proof of Theorem \ref{travelwave}\ref{a0}]

    Observe that display \eqref{newfluid} gives \textbf{two} equations simultaneously, and we shall work with each equation one at a time. Next, with $Z\sim \expo(1)$, $\phi(w) = \e^{-w}$ for $w\geq 0$, (one equation of) \eqref{newfluid} becomes
      \begin{align}
    \label{fluid2}
        v_+ \int\limits_{-\infty}^z \e^{y-z} F'_+(y) \di y = \parens*{v_+ + \alpha^+ - \rho^- + F_-(z)} F'_+(z) - \parens*{\alpha^- + \rho^+ - F_+(z)}F'_-(z) - \gamma F''_+(z),
    \end{align}
    where we have used the fact that $\rho_\pm(x) = F'_\pm(x)$. Differentiating the left-hand side of \eqref{fluid2} with respect to $z$, we obtain
    \begin{align*}
        \dfrac{\di}{\di z} \sqbr*{v_+ \int\limits_{-\infty}^z \e^{y-z} F'_+(y) \di y} = -v_+ \int\limits_{-\infty}^z \e^{y-z} F'_+(y) \di y + v_+ F'_+(z),
    \end{align*}
    and by using \eqref{fluid2} once again, we have
    \begin{align}
    \label{fluid100}
     &\dfrac{\di}{\di z} \sqbr*{v_+ \int\limits_{-\infty}^z \e^{y-z} F'_+(y) \di y} \nonumber \\
        &\qquad = -\parens*{v_+ +\alpha^+ - \rho^- + F_-(z)} F'_+(z)+ \parens*{\alpha^-+\rho^+ -F_+(z)}F'_-(z)+\gamma F''_+(z)+v_+ F'_+(z) \nonumber \\
        &\qquad = \parens*{\rho^- -\alpha^+ - F_-(z)} F'_+(z) + \parens*{\alpha^- + \rho^+ - F_+(z)}F'_-(z) + \gamma F''_+(z) \nonumber \\
        &\qquad = \parens*{\rho^- -\alpha^+}F'_+(z) + \parens*{\alpha^- + \rho^+}F'_-(z) - \parens*{F'_+(z) F_-(z) + F_+(z) F'_-(z)}\!+\!\gamma F''_+(z).
    \end{align}
    By ``reversing the process'' and anti-differentiating both sides of \eqref{fluid100} (again, with respect to $z$), we get that for some constant $C^+$,
      \begin{align}
    \label{fluid3}
        v_+ \int\limits_{-\infty}^z \e^{y-z} F'_+(y) \di y = \parens*{\rho^- -\alpha^+}F_+(z) + \parens*{\alpha^- + \rho^+}F_-(z) - F_+(z) F_-(z) + \gamma F'_+(z) + C^+.
    \end{align}
    Next, by Stieltjes integration-by-parts,
      \begin{equation}
    \label{c0}
        v_+ \int\limits_{-\infty}^z\!\e^{y-z} F'_+(y) \di y = v_+F_+(z) - v_+\int\limits_{-\infty}^z \e^{y-z} F_+(y) \di y =  v_+F_+(z) - v_+\int\limits_0^{\infty} \e^{-u} F_+(z-u) \di u \xrightarrow{\sspa z\to\pm\infty \sspa} 0
    \end{equation}
    using the Dominated Convergence Theorem. Using this, along with the fact that   $\lim_{z\to-\infty} F_\pm(z) = 0$, we conclude $\lim_{z\to-\infty} \gamma F_+'(z)$ exists and equals $-C^+$, which forces $C^+$ to be $0$. Moreover, by \eqref{fluid3} and \eqref{c0}, we can see that $\lim_{z\to\infty} \gamma F_+'(z)$ also exists, and again, it has to be $0$, since $\lim_{z\to\infty} F_+(z) = \rho^+ \in [0,1]$. Using this observation, taking the limit of both sides of \eqref{fluid3} as $z\to\infty$ and applying \eqref{c0} (once again), we obtain
      \begin{equation}
    \label{pmrho}
        0 = \parens*{\rho^- -\alpha^+}\rho^+ + \parens*{\alpha^- + \rho^+}\rho^- - \rho^+\rho^-  = \rho^+\rho^- - \alpha^+\rho^+ + \alpha^-\rho^-.
    \end{equation}
    Substituting $\rho^+ = 1 - \rho^-$ gives $\alpha^+ \rho^+ - \alpha^- \parens*{1-\rho^+} = \rho^+ \parens*{1-\rho^+}$, or equivalently,
    \begin{align}
    \label{fluid8}
        \parens*{\rho^+}^2 + \parens*{\alpha^+ + \alpha^- - 1}\rho^+ - \alpha^- = 0.
    \end{align}
    By solving for $\rho^+$ and observing that $\rho^+ \in [0, 1]$, we see that \eqref{fluid8} admits a single solution,
    \begin{align}
    \label{rhoplus}
        \rho^+ &= \dfrac{1 - \alpha^+ - \alpha^- + \dsqrt{\parens*{1 - \alpha\displaystyle^+ - \alpha\displaystyle^-}^2 + 4\alpha^-}}{2},
    \end{align}
    and accordingly,
    \begin{align}
    \label{rhominus}
        \rho^- &= 1 - \rho^+ = \dfrac{1 + \alpha^+ + \alpha^- - \dsqrt{\parens*{1 - \alpha\displaystyle^+ - \alpha\displaystyle^-}^2 + 4\alpha^-}}{2}.
    \end{align}
    Next, let us obtain the speed of the traveling wave. In this process, we also obtain an identity [see \eqref{fluid7} below] whose implications will be useful in the proof of Theorem \ref{travelwave}\ref{b0}. From \eqref{fluid2} and \eqref{fluid3},
    \begin{align*}
        &\parens*{v_+ + \alpha^+ - \rho^- + F_-(z)} F'_+(z) - \parens*{\alpha^- + \rho^+ - F_+(z)}F'_-(z) - \gamma F''_+(z) \\
        &\qquad = \parens*{\rho^- -\alpha^+}F_+(z) + \parens*{\alpha^- + \rho^+}F_-(z) - F_+(z) F_-(z) + \gamma F'_+(z),
    \end{align*}
    or equivalently,
      \begin{align}
    \label{fluid4}
        &F_+(z) F_-(z) + \parens*{F_+(z) F'_-(z) + F'_+(z) F_-(z)} \nonumber \\
        &\qquad = \parens*{\rho^- - \alpha^+} F_+(z) + \parens*{\rho^+ + \alpha^-} F_-(z) \nonumber \\
        &\qquad \qquad - \parens*{v_+ + \alpha^+ - \gamma - \rho^-} F'_+(z) + \parens*{\alpha^- + \rho^+} F'_-(z) + \gamma F''_+(z).
     \end{align}
    Similarly, by working with the other equation of \eqref{newfluid}, $\lim_{z \to \pm \infty}F_-'(z)=0$ and
    \begin{align*}
        &\parens*{v_- + \alpha^- + \rho^+ - F_+(z)} F'_-(z) - \parens*{\alpha^+ - \rho^- + F_-(z)}F'_+(z) - \gamma F''_-(z) \nonumber\\
        &\qquad = -\parens*{\rho^+ + \alpha^-}F_-(z) + \parens*{\alpha^+ - \rho^-}F_+(z) + F_+(z) F_-(z) + \gamma F'_-(z).
    \end{align*}
    Again, equivalently,
      \begin{align}
    \label{fluid5}
        &F_+(z) F_-(z) + \parens*{F_+(z) F'_-(z) + F'_+(z) F_-(z)} \nonumber \\
        &\qquad = \parens*{\rho^- - \alpha^+} F_+(z) + \parens*{\rho^+ + \alpha^-} F_-(z) \nonumber \\
        &\qquad \qquad+ \parens*{v_- + \alpha^- - \gamma + \rho^+} F'_-(z) - \parens*{\alpha^+ - \rho^-} F'_+(z) - \gamma F''_-(z).
    \end{align}
    After equating the right-hand sides of \eqref{fluid4} and \eqref{fluid5} and performing some algebraic simplifications, we get
    \begin{align}
    \label{fluid6}
        \parens*{v_+ - \gamma} F'_+(z) - \gamma F''_+(z) = \parens*{\gamma - v_-} F'_-(z) + \gamma F''_-(z).
    \end{align}
    Now, define
    \begin{align}
    \label{fluid17}
        \Phi(z) := F_+(z) F_-(z) + \parens*{\alpha^+ - \rho^-} F_+(z) - \parens*{\rho^+ + \alpha^-} F_-(z).
    \end{align}
    Then from \eqref{fluid4}, \eqref{fluid5}, and \eqref{fluid6},
    \begin{align}
    \label{fluid7}
        \Phi(z) + \Phi'(z) &= -\parens*{v_+ - \gamma} F'_+(z) + \gamma F''_+(z) = -\parens*{\gamma - v_-} F'_-(z) - \gamma F''_-(z).
    \end{align}
    To obtain the speed of the traveling wave, $\gamma$, note that by integrating the two sides of the last equality in \eqref{fluid7}, and using $\lim_{z \to -\infty} F_{\pm}(z) = \lim_{z \to -\infty} F_{\pm}'(z)=0$,
    \begin{align}
    \label{fluid9}
        \parens*{v_+ - \gamma} F_+(z) - \gamma F'_+(z) = \parens*{\gamma - v_-} F_-(z) + \gamma F'_-(z).
    \end{align}
    Taking the limit of both sides of \eqref{fluid9} as $z \to \infty$ and using $\lim_{z \to \infty} F_{\pm}'(z)=0$, we have
    \begin{align}
    \label{fluid103}
        \parens*{v_+ - \gamma} \rho^+ = \parens*{\gamma - v_-} \rho^-,
    \end{align}
    or
    \begin{align}
        \label{speedTW}
        \gamma = v_+ \rho^+ + v_- \rho^-,
    \end{align}
    as desired. The last assertion of Theorem \ref{travelwave}\ref{a0} now follows upon ``solving'' $\rho^+ = \rho^- = \dfrac{1}{2}$ and using the form of $\rho^+$ and $\rho^-$ in terms of $\alpha^+$ and $\alpha^-$.

    This completes the proof of Theorem \ref{travelwave}\ref{a0}.
\end{proof}

\subsection{A Coupled ODE Formulation for Exponential Jumps and \texorpdfstring{$v_-=0$}{v- = 0}}
To prove Theorem \ref{travelwave}\ref{b0}, we will show in the following result that when $Z\sim \expo(1)$ and $v_- = 0$, the traveling wave equations reduce to a system of first-order ODEs in two dimensions. 
\begin{lemma}
\label{twreversed}
Suppose that $Z\sim \expo(1)$ and $v_-=0$. Furthermore, suppose that $H = \parens*{H_+, H_-}$ is a traveling wave with mass partition $(\rho^+,\rho^-)$ and speed $\gamma$. For each $t\in\real$, let $G_\pm(t) := H_\pm(-t)$. Then $G_\pm(t)$ solves the following system of first-order ODEs:
    \begin{align}
    \label{sysFluid1}
        {\setlength{\nulldelimiterspace}{0pt}%
        \left\{
        \begin{aligned}
            &G'_+ = -\dfrac{\rho^-}{\gamma} G_+ G_- - \dfrac{v_+ - \gamma + \alpha^+ - \rho^-}{\gamma} G_+ + \rho^-  \dfrac{(\rho^+ +\alpha^- + \gamma)}{\gamma \rho^+}G_- \\
            &G'_- = \dfrac{\rho^+}{\gamma} G_+ G_- - \dfrac{\rho^+ + \alpha^-}{\gamma} G_- + \rho^+  \dfrac{(\alpha^+ - \rho^-)}{\gamma \rho^-}G_+
        \end{aligned}
        \right.}.
    \end{align}
     Conversely, suppose that $G=\parens*{G_+, G_-}\!: \real\to[0, 1]^2$ is a coordinate-wise non-increasing solution to \eqref{sysFluid1} such that
    \begin{align*}
        \lim_{t\to-\infty} G_\pm(t) = 1& & \text{and} & & \lim_{t\to\infty} G_\pm(t) = 0,
    \end{align*}
    and suppose that the probability measures with distribution functions $H_\pm(t) = G_\pm(-t)$ have finite first moments. Then with $H := (H_+, H_-)$, where, for each $t\in\real$, $H_\pm(t) := G_\pm(-t)$, $H$ determines a traveling wave with mass partition $\parens*{\rho^+, \rho^-}$ and speed $\gamma$.
 \end{lemma}
\begin{proof}
      We take up the argument for the forward implication first. To that end, define
    \begin{align}
    \label{fluid14}
        g(z) := -\parens*{\dfrac{v_+}{\gamma} - 1} F_+(z) + F'_+(z) + \rho^+ \parens*{\dfrac{v_+}{\gamma} - 1}.
    \end{align}
    Note that
      \begin{align*}
        \dfrac{\di}{\di z} \sqbr*{\exp\Biggparens{-\parens*{\dfrac{v_+}{\gamma}-1}z} F_+(z)} &= \exp\Biggparens{-\parens*{\dfrac{v_+}{\gamma}-1}z}\cdot \Biggparens{g(z) - \rho^+ \parens*{\dfrac{v_+}{\gamma}-1}},
    \end{align*}
    which implies
      \begin{align*}
        -\exp\Biggparens{-\parens*{\dfrac{v_+}{\gamma}-1}x} F_+(x) = -\rho^+ \exp\Biggparens{-\parens*{\dfrac{v_+}{\gamma}-1}x} + \int\limits_x^\infty \exp\Biggparens{-\parens*{\dfrac{v_+}{\gamma}-1}z} g(z) \di z,
    \end{align*}
    or, equivalently,
      \begin{align}
    \label{fluid11}
        F_+(x) &= \rho^+ - \int\limits_x^\infty \exp\Biggparens{-\parens*{\dfrac{v_+}{\gamma}-1}(z-x)} g(z) \di z = \rho^+ - \dfrac{\gamma}{v_+ - \gamma} \expc\sqbr*{g(x+E_+)},
    \end{align}
    where $E_+\sim \expo\parens*{\dfrac{v_+}{\gamma}-1}$. Similarly, using \eqref{fluid9} and \eqref{fluid103}, we can also express $g$ as
    \begin{align}
    \label{fluid15}
        g(z) = -\parens*{1-\dfrac{v_-}{\gamma}} F_-(z) - F'_-(z) + \rho^-\parens*{1-\dfrac{v_-}{\gamma}},
    \end{align}
    and by a similar argument, we have
      \begin{align*}
        -\dfrac{\di}{\di z} \sqbr*{\exp\Biggparens{\parens*{1-\dfrac{v_-}{\gamma}}z} F_-(z)} &= \exp\Biggparens{\parens*{1-\dfrac{v_-}{\gamma}}z}\cdot \sqbr*{g(z) - \rho^- \Biggparens{1-\dfrac{v_-}{\gamma}}},
    \end{align*}
    which gives
      \begin{align*}
        -\exp\Biggparens{\parens*{1-\dfrac{v_-}{\gamma}}x} F_-(x) = -\rho^- \exp\Biggparens{\parens*{1-\dfrac{v_-}{\gamma}}x} + \int\limits_{-\infty}^x \exp\Biggparens{\parens*{1-\dfrac{v_-}{\gamma}}z} g(z) \di z,
    \end{align*}
    so
      \begin{align}
   \label{fluid12}
        F_-(x) &= \rho^- - \int\limits_{-\infty}^x \exp\Biggparens{-\parens*{1-\dfrac{v_-}{\gamma}}(x-z)} g(z) \di z = \rho^- - \dfrac{\gamma}{\gamma-v_-} \expc\sqbr*{g(x-E_-)},
    \end{align}
    where $E_-\sim \expo\parens*{1-\dfrac{v_-}{\gamma}}$. Next, from \eqref{fluid7}, note that $\Phi(z) + \Phi'(z) = \gamma g'(z)$, which implies
    \begin{align*}
        \e^z\sqbr*{\Phi(z) + \Phi'(z)} = \gamma \e^z g'(z).
    \end{align*}
    Equivalently,
    \begin{align}
    \label{fluid10}
        \dfrac{\di}{\di z} \sqbr*{\e^z \Phi(z)} = \gamma \e^z g'(z),
    \end{align}
    which yields
      \begin{align}
    \label{fluid101}
        \e^x\cdot \Phi(x) &= \gamma \int\limits_{-\infty}^x \e^z g'(z) \di z = \gamma\cdot\e^z g(z) \Bigg|_{-\infty}^x - \gamma\int\limits_{-\infty}^x \e^z g(z) \di z = \gamma\cdot\e^x g(x) -\gamma \int\limits_{-\infty}^x \e^z g(z) \di z.
    \end{align}
    In turn, \eqref{fluid101} gives
      \begin{align}
    \label{fluid13}
        \Phi(x) &= \gamma g(x) - \gamma \int\limits_{-\infty}^x \e^{-(x-z)} g(z) \di z = \gamma g(x) - \gamma \expc\sqbr*{g(x-E_1)},
    \end{align}
    where $E_1 \sim \expo(1)$. Now, suppose that $v_- = 0$. Then $E_- \eqd E_1$, and from both \eqref{fluid12} and \eqref{fluid13},
    \begin{align*}
        \Phi(x) &= \gamma g(x) - \gamma \expc\sqbr*{g(x-E_1)} = \gamma g(x) - \gamma \parens*{\rho^- - F_-(x)},
    \end{align*}
    or equivalently,
    \begin{align}
    \label{fluid16}
        \gamma g(x) &= \gamma\parens*{\rho^- - F_-(x)} + \Phi(x) \nonumber \\
        &= \gamma \rho^- - \gamma F_-(x) + F_+(x) F_-(x) + \parens*{\alpha^+ - \rho^-}F_+(x) - \parens*{\rho^+ + \alpha^-} F_-(x) \nonumber \\
        &= F_+(x)F_-(x) + \parens*{\alpha^+-\rho^-}F_+(x) -\parens*{\rho^++\alpha^-+\gamma} F_-(x) + \gamma \rho^-,
    \end{align}
    where equality on the second line follows from the definition of $\Phi$ in \eqref{fluid17}. Also, by multiplying both sides of \eqref{fluid14} by $\gamma$, we have
    \begin{align*}
        \gamma g(x) = -\parens*{v_+ - \gamma}F_+(x) + \gamma F'_+(x) + \rho^+\parens*{v_+ - \gamma},
    \end{align*}
    and this, along with \eqref{fluid16}, yields
      \begin{align}
    \label{fluid102}
        &-\parens*{v_+ - \gamma}F_+(x) + \gamma F'_+(x) + \rho^+\parens*{v_+ - \gamma} \nonumber \\
        &\qquad = F_+(x) F_-(x) + \parens*{\alpha^+ - \rho^-}F_+(x) - \parens*{\rho^+ + \alpha^- + \gamma} F_-(x) + \gamma\rho^-.
    \end{align}
    Next, by rearranging the terms of \eqref{fluid102}, we get
      \begin{align*}
        F'_+(x) = \dfrac{1}{\gamma} F_+(x) F_-(x) + \dfrac{v_+ - \gamma + \alpha^+ - \rho^-}{\gamma} F_+(x) - \dfrac{\rho^++\alpha^-+\gamma}{\gamma} F_-(x) + \dfrac{\gamma \rho^- - \rho^+ \parens*{v_+ - \gamma}}{\gamma}.
    \end{align*}
    Now, from \eqref{fluid103} and the assumption that $v_- = 0$, we have  $\rho^+\parens*{v_+ - \gamma} = \gamma \rho^-$; thus,
    \begin{align*}
        F'_+(x) = \dfrac{1}{\gamma} F_+(x) F_-(x) + \dfrac{v_+ - \gamma + \alpha^+ - \rho^-}{\gamma} F_+(x) - \dfrac{\rho^++\alpha^-+\gamma}{\gamma} F_-(x).
    \end{align*}
    Similarly, by multiplying both sides of \eqref{fluid15} by $\gamma$ and noting that $v_- = 0$ by assumption, we obtain
    \begin{align*}
        \gamma g(x) = -\gamma F_-(x) - \gamma F'_-(x) + \gamma \rho^-,
    \end{align*}
    and from \eqref{fluid16}, we also have that
      \begin{align*}
        -\gamma F_-(x) - \gamma F'_-(x) + \gamma \rho^- = F_+(x) F_-(x) + \parens*{\alpha^+ - \rho^-}F_+(x) - \parens*{\rho^+ + \alpha^- + \gamma} F_-(x) + \gamma\rho^-.
    \end{align*}
    Again, we can follow the same line of argument to arrive at
    \begin{align*}
        F'_-(x) = -\dfrac{1}{\gamma} F_+(x) F_-(x) + \dfrac{\rho^+ + \alpha^-}{\gamma} F_-(x) - \dfrac{\alpha^+ - \rho^-}{\gamma} F_+(x).
    \end{align*}
    Lastly, let $H_\pm := \dfrac{F_\pm}{\rho\displaystyle^\pm}$. This guarantees that $H_+$ and $H_-$ are valid distribution functions whose dynamics are described by the system of equations
    \begin{align}
    \label{sysFluid2}
        {\setlength{\nulldelimiterspace}{0pt}%
        \left\{
        \begin{aligned}
            &H'_+ =  \dfrac{\rho^-}{\gamma} H_+ H_- + \dfrac{v_+ - \gamma + \alpha^+ - \rho^-}{\gamma} H_+ - \dfrac{\rho^-\parens*{\rho^+ + \alpha^- + \gamma}}{\gamma \rho^+} H_- \\
            &H'_- =  -\dfrac{\rho^+}{\gamma} H_+ H_- + \dfrac{\rho^+ + \alpha^-}{\gamma} H_- - \dfrac{\rho^+\parens*{\alpha^+ - \rho^-}}{\gamma \rho^-} H_+
        \end{aligned}
        \right.},
    \end{align}
    and with $G_\pm(t) := H_\pm(-t)$, \eqref{sysFluid1} follows. That gives us the forward implication.

    For the backward implication, suppose that $G = \parens*{G_+, G_-}$ is a coordinate-wise non-increasing solution of \eqref{sysFluid1} satisfying
    \begin{align*}
        \lim_{t\to-\infty} G_\pm(t) = 1 & & \text{and}& & \lim_{t\to\infty} G_\pm(t) = 0.
    \end{align*}
    For each $t\in\real$, let $H_\pm(t) := G_\pm(-t)$, and $F_\pm(t) := \rho^\pm H_\pm(t)$. Since $G_\pm$ are non-increasing and satisfy the preceding boundary conditions, $H_\pm$ are non-decreasing, with
    \begin{align*}
        \lim_{t\to-\infty} H_\pm(t) = 0  & & \text{and} & & \lim_{t\to\infty} H_\pm(t) = 1;
    \end{align*}
    that is, $H_+$ and $H_-$ are proper distribution functions. Furthermore, since the vector field in \eqref{sysFluid1} is smooth, $H_+$ and $H_-$ are $\CCC^2$. By reversing time in \eqref{sysFluid1}, then, $H_+$ and $H_-$ satisfy \eqref{sysFluid2}, which implies that
    \begin{align}
    \label{sysFluid2a}
        \gamma F'_+(t) = F_+(t) F_-(t) + \parens*{v_+ - \gamma + \alpha^+ - \rho^-} F_+(t) - \parens*{\rho^+ + \alpha^- + \gamma} F_-(t),
    \end{align}
    and that
    \begin{align}
    \label{sysFluid2b}
        \gamma F'_-(t) = -F_+(t) F_-(t) + \parens*{\rho^+ + \alpha^-} F_-(t) - \parens*{\alpha^+ - \rho^-} F_+(t).
    \end{align}
    Now, differentiating both sides of \eqref{sysFluid2b} gives
    \begin{align*}
        \beta^-(t) F'_-(t) - \beta^+(t) F'_+(t) - \gamma F''_-(t) = 0,
    \end{align*}
    which is one of the two equations in \eqref{newfluid}, as $v_- = 0$. It remains to show that \eqref{sysFluid2} implies the other equation in \eqref{newfluid}. To that end, define
    \begin{align*}
        I_+(t) := v_+ \int\limits_{-\infty}^t \e^{s-t} F'_+(s) \di s,
    \end{align*}
    and observe that $\lim_{t\to-\infty} I_+(t) = 0$ and that $I_+'(t) + I_+(t) = v_+ F'_+(t)$. Next, adding \eqref{sysFluid2a} and \eqref{sysFluid2b} gives
    \begin{align*}
        \gamma\bigparens{F'_+(t) + F'_-(t)} = (v_+ - \gamma) F_+(t) - \gamma F_-(t).
    \end{align*}
    In addition, we also have $\lim_{t\to-\infty} \bigsqbr{(v_+ -\gamma) F_+(t) - \gamma F_-(t)} = 0$. From these two observations, we obtain
    \begin{align*}
        \dfrac{\di}{\di t} \bigsqbr{(v_+ -\gamma) F_+(t) - \gamma F_-(t)} + (v_+ -\gamma) F_+(t) - \gamma F_-(t) = v_+ F'_+(t).
    \end{align*}
    Thus, by the uniqueness of solution to a linear differential equation, it must be the case that
    \begin{align*}
        I_+(t) = (v_+ - \gamma) F_+(t) - \gamma F_-(t) = \gamma\bigparens{F'_+(t) + F'_-(t)}.
    \end{align*}
    Finally, by differentiating \eqref{sysFluid2a} with respect to $t$ and simplifying, we obtain
    \begin{align*}
        \bigparens{v_+ + \beta^+(t)} F'_+(t) - \beta^-(t) F'_-(t) - \gamma F''_+(t) = \gamma\bigparens{F'_+(t) + F'_-(t)} = I_+(t) = v_+\int\limits_{-\infty}^t \e^{s-t} F'_+(s) \di s,
    \end{align*}
    which is the other equation in \eqref{newfluid}. In other words, $\rho_\pm(t) := F'_\pm(t)$ satisfy \eqref{newfluid}.

    To wrap up this proof, note that \eqref{newfluid} gives \eqref{newerFluid4} pointwise for the translated densities $\rho_\pm(t, x) = \rho_\pm(x-\gamma t)$. Integrating \eqref{newerFluid4} in time and space against $f(\cdot, +)$ and $f(\cdot, -)$, and reversing the Fubini calculation in Lemma \ref{integrodiff}, yields \eqref{originalFluid} for every $f\in\HHH$. All terms are integrable because $f$ is $1$-Lipschitz, the profiles have finite first moments, the switching rates are bounded, and from Assumption \ref{assumption0}, $\expc(Z) = 1 < \infty$.
\end{proof}

\subsection{Equilibria and Linearized Stability Analysis}
  To prove Theorem \ref{travelwave}\ref{b0}, we first identify (precisely) $(0, 0)$ and $(1, 1)$ as the equilibria of \eqref{sysFluid1} as well as analyze them. We shall also construct a \textbf{heteroclinic orbit}  from $(1, 1)$ to $(0, 0)$,   whose coordinates are non-increasing. In order to perform those tasks, we need information on the local stability around the equilibrium points. The next lemma analyzes the fixed points and their stability properties for the associated \emph{linearized versions} of the flow, given by \eqref{sysFluid1}:
\begin{lemma}
\label{linearode}
Suppose $Z\sim \expo(1)$ and $v_+>v_- = 0$. The following statements about the system of first-order ODEs \eqref{sysFluid1} are true:
    \begin{enumerate}[label = (\alph*)]
        \item \label{55a} The system of first-order ODEs \eqref{sysFluid1} has exactly two fixed points (or equilibria) given by $(0, 0)$ and $(1, 1)$.

        \item \label{55b} The linearized version of \eqref{sysFluid1} at $(0, 0)$,
        \begin{align}
        \label{linearSysFluid1}
            {\setlength{\nulldelimiterspace}{0pt}%
            \left\{
            \begin{aligned}
                &\dot{x} = -\dfrac{v^+ -\gamma + \alpha^+ - \rho^-}{\gamma} x + \dfrac{\rho^- \parens*{\rho\displaystyle^+ + \alpha\displaystyle^- + \gamma}}{\gamma \rho\displaystyle^+} y \\
                &\dot{y} = \dfrac{\rho\displaystyle^+ \parens*{\alpha\displaystyle^+ - \rho\displaystyle^-}}{\gamma \rho\displaystyle^-} x - \dfrac{\rho\displaystyle^+ + \alpha\displaystyle^-}{\gamma} y
            \end{aligned}
            \right.},
        \end{align}
        has a stable equilibrium at $(0, 0)$, in the sense that both eigenvalues of the associated Jacobian at $(0, 0)$,
          \begin{align}
          \label{originalJ00}
            J(0, 0) = \begin{bmatrix}
                -\dfrac{v_+ -\gamma + \alpha^+ - \rho^-}{\gamma} & \dfrac{\rho\displaystyle^-\parens*{\rho^+ + \alpha^- + \gamma}}{\gamma \rho\displaystyle^+} \\
                \dfrac{\rho^+\parens*{\alpha^+-\rho^-}}{\gamma \rho\displaystyle^-}& -\dfrac{\rho^+ + \alpha^-}{\gamma}
            \end{bmatrix}
        \end{align}
        are real and negative.

         In addition, let $\lambda_1 < \lambda_2 < 0$ denote the two eigenvalues of $J(0, 0)$. Then an eigenvector corresponding to $\lambda_1$ has coordinates of opposite signs, while an eigenvector corresponding to $\lambda_2$ has coordinates of the same sign. (In each case, the corresponding eigenvector has non-zero coordinates.)

        \item \label{55c} The linearized version of \eqref{sysFluid1} at $(1, 1)$,
        \begin{align}
        \label{linearSysFluid2}
            {\setlength{\nulldelimiterspace}{0pt}%
            \left\{
            \begin{aligned}
                &\dot{x} = -\dfrac{v^+ -\gamma + \alpha^+}{\gamma} x + \dfrac{\rho^- \parens*{\alpha\displaystyle^- + \gamma}}{\gamma \rho\displaystyle^+} y \\
                &\dot{y} = \dfrac{\rho\displaystyle^+ \alpha\displaystyle^+}{\gamma \rho\displaystyle^-} x - \dfrac{\alpha\displaystyle^-}{\gamma} y
            \end{aligned}
            \right.},
        \end{align}
        has a one-dimensional unstable manifold, in the sense that the associated Jacobian $J(1, 1)$ has two real eigenvalues $\lambda_u$ and $\lambda_s$ such that $\lambda_s < 0 < \lambda_u$. Moreover, the eigenvector $\vv_u = \begin{bmatrix} x_u & & y_u\end{bmatrix}\tran$ corresponding to $\lambda_u$ satisfies
        \begin{align*}
            \dfrac{v^+ - \gamma + \alpha\displaystyle^+}{v_+ - \gamma + \alpha\displaystyle^+ - \rho\displaystyle^-} < \dfrac{y_u}{x_u} < \dfrac{\alpha\displaystyle^+}{\alpha\displaystyle^+ - \rho\displaystyle^-}.
        \end{align*}
    \end{enumerate}
\end{lemma}
\begin{proof}
    Let us proceed:
     \begin{enumerate}[label = (\alph*)]
        \item   Let $(g_+, g_-)$ be a fixed point of \eqref{sysFluid1}. Then we can multiply the first equation of \eqref{sysFluid1} by $\rho^+$, the second one by $\rho^-$, and add the two resulting equations to obtain
        \begin{align}
        \label{fixPts1}
            -\dfrac{\rho^+(v_+ - \gamma)}{\gamma} g_+ + \rho^- g_- = 0.
        \end{align}
          Also, since $v_- = 0$, \eqref{fluid103} becomes $\rho^+ (v_+ - \gamma) = \gamma \rho^-$, and plugging this back into \eqref{fixPts1} gives
        \begin{align*}
            \rho^-(g_- - g_+) = 0.
        \end{align*}
          Since $\rho^- > 0$, that implies $g_- = g_+$. Denote their common value by $g$. Substituting this back into the second equation of \eqref{sysFluid1} gives
        \begin{align}
        \label{fixPts2}
            \dfrac{\rho^+}{\gamma} g^2 - \dfrac{\rho^+ + \alpha^-}{\gamma} g + \dfrac{\rho^+ \parens*{\alpha^+ - \rho^-}}{\gamma \rho^-} g = 0.
        \end{align}
          Recall that from \eqref{pmrho}, $\rho^+ \parens*{\alpha^+ - \rho^-} = \alpha^- \rho^-$, and plugging this back into \eqref{fixPts2} yields
        \begin{align}
        \label{fixPts3}
            \dfrac{\rho^+}{\gamma} g^2 - \dfrac{\rho^+}{\gamma} g = 0 \equi \dfrac{\rho^+}{\gamma} g(g-1) = 0.
        \end{align}
          Now, since $\rho^+ > 0$ and $\gamma > 0$, \eqref{fixPts3} implies $g = 0$ or $g = 1$, so every fixed point of \eqref{sysFluid1} is either $(0, 0)$ or $(1, 1)$. Conversely, direct substitution into \eqref{sysFluid1} shows that both $(0, 0)$ and $(1, 1)$ are fixed points. We conclude that $(0, 0)$ and $(1, 1)$ are precisely the two fixed points of \eqref{sysFluid1}.

        \item Note that
          \begin{align*}
            \trace\bigparens{J(0,0)} = -\dfrac{v_+-\gamma + \alpha^+ + \alpha^- + \rho^+ - \rho^-}{\gamma}.
        \end{align*}
        From the explicit formulas for $\rho^+$ and $\rho^-$ [see \eqref{rhoplus} and \eqref{rhominus}], we have
        \begin{align*}
            \alpha^+ + \alpha^- + \rho^+ > \dfrac{1+\alpha^++\alpha^-}{2} > \rho^-,
        \end{align*}
        since $\alpha^+ > 0$ and $\alpha^- > 0$, by construction. In addition, $v_+ - \gamma = v_+ \rho^- > 0$, so $\trace\bigparens{J(0,0)} < 0$. Moreover, since $\gamma = v_+ \rho^+$, $v_+ - \gamma = v_+ \rho^-$, and $\alpha^+ \rho^+ - \alpha^- \rho^- = \rho^+ \rho^-$, the determinant $\abs*{J(0, 0)}$  takes the form
        \begin{align*}
            \gamma^2 \abs*{J(0, 0)} &= \parens*{v_+ - \gamma + \alpha^+ - \rho^-}\parens*{\rho^+ + \alpha^-} - \parens*{\rho^+ + \alpha^- + \gamma} \parens*{\alpha^+ - \rho^-} \\
            &= \parens*{\rho^+ + \alpha^-} \parens*{v_+ - \gamma + \alpha^+ - \rho^- - \alpha^+ + \rho^-} - \gamma\parens*{\alpha^+ - \rho^-} \\
            &= \parens*{\rho^+ + \alpha^-} \parens*{v_+ - \gamma} - \gamma\parens*{\alpha^+ - \rho^-} \\
            &= \parens*{\rho^+ + \alpha^-} v_+ \rho^- - v_+ \rho^+\parens*{\alpha^+ - \rho^-} \\
            &= v_+ \Bigparens{\rho^- \parens*{\rho^+ + \alpha^-} - \rho^+ \parens*{\alpha^+ - \rho^-}} \\
            &= v_+ \Bigparens{\rho^+ \rho^- + \parens*{\rho^+ \rho^- - \alpha^+ \rho^+ + \alpha^- \rho^-}} = v_+ \rho^+ \rho^- > 0.
        \end{align*}
        This gives us that both eigenvalues of $J(0, 0)$ have negative real parts.

        Next, to show that both eigenvalues of $J(0, 0)$ are real, we need to show that
        \begin{align*}
            \trace^2\bigparens{J(0, 0)} - 4 \abs*{J(0, 0)} > 0.
        \end{align*}
        To do so, observe that from the original formation of $J(0, 0)$ in \eqref{originalJ00},
          \begin{align}
            \label{realEigen}
            \gamma^2 \sqbr*{\trace^2\bigparens{J(0, 0)} - 4 \abs*{J(0, 0)}} = \parens*{v_+ - \gamma + \alpha^+ - \alpha^- - \rho^+ - \rho^-}^2 + 4\parens*{\alpha^+ - \rho^-}\parens*{\rho^+ + \alpha^- + \gamma},
        \end{align}
        and from \eqref{pmrho}, $\alpha^+ \rho^+ - \alpha^- \rho^- = \rho^+ \rho^-$, or equivalently,
        \begin{align*}
            \rho^+ \parens*{\alpha^+ - \rho^-} = \alpha^- \rho^- \equi
            \alpha^+ - \rho^- = \dfrac{\alpha^- \rho^-}{\rho\displaystyle^+} > 0.
        \end{align*}
         Thus, the right-hand side of \eqref{realEigen} is positive, so our claim  naturally follows.

        To prove the last assertion, note that
        since $\alpha^+ - \rho^- = \dfrac{\alpha^- \rho^-}{\rho^+} > 0$, both off-diagonal entries of $J(0, 0)$ are strictly positive. Now let $\vv = \begin{bmatrix} x & & y \end{bmatrix}\tran$ be an eigenvector corresponding to some eigenvalue $\lambda$ of $J(0, 0)$. Then the first row of the eigenvalue equation gives
          \begin{align}
        \label{eigSlopeJ00}
            \parens*{-\dfrac{v_+ - \gamma + \alpha^+ - \rho^-}{\gamma} - \lambda} x + \dfrac{\rho^- \parens*{\rho^+ + \alpha^- + \gamma}}{\gamma \rho^+} y = 0.
        \end{align}
        Since both off-diagonal entries of $J(0, 0)$ are strictly positive, neither coordinate of an eigenvector can vanish; hence, we may rearrange \eqref{eigSlopeJ00} to obtain
        \begin{align}
        \label{eigSlopeJ00b}
            \dfrac{y}{x} = \dfrac{\dfrac{v_+ - \gamma + \alpha^+ - \rho^-}{\gamma} + \lambda}{\dfrac{\rho^- \parens*{\rho^+ + \alpha^- + \gamma}}{\gamma \rho^+}}.
        \end{align}
        Since the denominator in \eqref{eigSlopeJ00b} is positive, the sign of $\dfrac{y}{x}$ is exactly the sign of
        \begin{align*}
            \dfrac{v_+ - \gamma + \alpha^+ - \rho^-}{\gamma} + \lambda.
        \end{align*}
        Next, let
          \begin{align*}
            a := \dfrac{v_+ - \gamma + \alpha^+ - \rho^-}{\gamma}, & &
            b := \dfrac{\rho^- \parens*{\rho^+ + \alpha^- + \gamma}}{\gamma \rho^+}, & & c := \dfrac{\rho^+ \parens*{\alpha^+ - \rho^-}}{\gamma \rho^-}, & & \text{and}& & d := \dfrac{\rho^+ + \alpha^-}{\gamma},
        \end{align*}
        so that
        \begin{align*}
            J(0, 0) = \begin{bmatrix} -a & & b \\ c & & -d \end{bmatrix},
        \end{align*}
        and note that $a$, $b$, $c$, and $d$ are four positive numbers. Accordingly, the characteristic polynomial is
        \begin{align*}
            p(\lambda) = \abs*{J(0, 0) - \lambda I} = (\lambda + a)(\lambda + d) - bc,
        \end{align*}
        and therefore $p(-a) = -bc < 0$. In addition, since $p$ is a quadratic polynomial with positive leading coefficient and two real roots $\lambda_1$ and $\lambda_2$, where, say, $\lambda_1 < \lambda_2$, it follows that $\lambda_1 < -a < \lambda_2$, or equivalently,
        \begin{align*}
            a + \lambda_1 < 0 < a + \lambda_2.
        \end{align*}
        Returning to \eqref{eigSlopeJ00b}, we conclude that for $\lambda = \lambda_1$, $\dfrac{y}{x} < 0$, while for $\lambda = \lambda_2$, $\dfrac{y}{x} > 0$. Thus, the eigenvector corresponding to the smaller eigenvalue $\lambda_1$ has coordinates of opposite signs, whereas the eigenvector corresponding to the larger eigenvalue $\lambda_2$ has coordinates of the same sign.

        \item Lastly, observe that
        \begin{align}
        \label{detJ11}
            \abs*{J(1, 1)} &= \dfrac{1}{\gamma\displaystyle^2} \sqbr*{\parens*{v_+ - \gamma + \alpha^+} \cdot \alpha^- - \parens*{\alpha^- + \gamma} \cdot \alpha^+} \nonumber \\
            &= \dfrac{1}{\gamma\displaystyle^2} \sqbr*{v_+ \alpha^- - \gamma \parens*{\alpha^- + \alpha^+}} \nonumber \\
            &= \dfrac{1}{\gamma\displaystyle^2} \sqbr*{v_+ \alpha^- - v_+ \rho^+ \parens*{\alpha^- + \alpha^+}} = \dfrac{v_+}{\gamma\displaystyle^2} \sqbr*{\alpha^- - \rho^+ \parens*{\alpha^- + \alpha^+}}.
        \end{align}
        Now, from \eqref{fluid8}, recall that $\rho^+$ satisfies $\parens*{\rho^+}^2 + \parens*{\alpha^+ + \alpha^- - 1}\rho^+ - \alpha^- = 0$. Equivalently, we can rearrange the terms to obtain
        \begin{align*}
            \alpha^- - \rho^+ \parens*{\alpha^- + \alpha^+} = \parens*{\rho^+}^2 - \rho^+,
        \end{align*}
        and since $\rho^+ \in (0, 1)$, $\parens*{\rho^+}^2 - \rho^+ < 0$. Thus, $\alpha^- - \rho^+ \parens*{\alpha^- + \alpha^+} < 0$ as well, and so from \eqref{detJ11}, $\abs*{J(1, 1)} < 0$. This implies that both eigenvalues of $J(1, 1)$ are real; furthermore, one eigenvalue is negative and the other one is positive.

        Next, to prove the bounds on the ``slope'' of the eigenvector, observe that from the characteristic equation of $J(1, 1)$,
        \begin{align}
        \label{slopeEigenv}
            \dfrac{y_u}{x_u} = \dfrac{\dfrac{\rho\displaystyle^+ \alpha\displaystyle^+}{\gamma \rho\displaystyle^-}}{\dfrac{\alpha\displaystyle^-}{\gamma} + \lambda_u} = \dfrac{\dfrac{v_+ - \gamma + \alpha\displaystyle^+}{\gamma} + \lambda_u}{\dfrac{\rho\displaystyle^-\parens*{\alpha\displaystyle^- + \gamma}}{\gamma \rho\displaystyle^+}}.
        \end{align}
        We can exploit each equality above to prove the bound on the eigenvector ``slope.'' First, since $\lambda_u > 0$, from \eqref{slopeEigenv}, and since $\alpha^+ \rho^+ - \alpha^- \rho^- = \rho^+ \rho^-$,
        \begin{align*}
            \dfrac{y_u}{x_u} = \dfrac{\dfrac{\rho\displaystyle^+ \alpha\displaystyle^+}{\gamma \rho\displaystyle^-}}{\dfrac{\alpha\displaystyle^-}{\gamma} + \lambda_u} < \dfrac{\dfrac{\rho\displaystyle^+ \alpha\displaystyle^+}{\gamma \rho\displaystyle^-}}{\dfrac{\alpha\displaystyle^-}{\gamma}} = \dfrac{\rho^+ \alpha^+}{\alpha\displaystyle^- \rho\displaystyle^-} = \dfrac{\rho^+ \alpha^+}{\alpha\displaystyle^+ \rho\displaystyle^+ - \rho\displaystyle^+ \rho\displaystyle^-} = \dfrac{\alpha^+}{\alpha\displaystyle^+ - \rho\displaystyle^-},
        \end{align*}
        and the upper bound on $\dfrac{y_u}{x_u}$ holds. Similarly, using the other formulation of $\dfrac{y_u}{x_u}$ from \eqref{slopeEigenv}, we have that
        \begin{align*}
            \dfrac{y_u}{x_u} = \dfrac{\dfrac{v_+ - \gamma + \alpha\displaystyle^+}{\gamma} + \lambda_u}{\dfrac{\rho\displaystyle^-\parens*{\alpha\displaystyle^- + \gamma}}{\gamma \rho\displaystyle^+}} > \dfrac{\dfrac{v_+ - \gamma + \alpha\displaystyle^+}{\gamma}}{\dfrac{\rho\displaystyle^-\parens*{\alpha\displaystyle^- + \gamma}}{\gamma \rho\displaystyle^+}} = \dfrac{\rho^+ \parens*{v_+ - \gamma + \alpha\displaystyle^+}}{\rho^- \parens*{\alpha\displaystyle^- + \gamma}}.
        \end{align*}
        Here, again, as $\gamma = v_+ \rho^+$, $v_+ - \gamma = v_+ \rho^-$, and $\alpha^+ \rho^+ - \alpha^- \rho^- = \rho^+ \rho^-$, observe that
          \begin{align*}
            \rho^- \parens*{\alpha^- + \gamma} = \alpha^- \rho^- + \gamma \rho^- &= \alpha^- \rho^- + \rho^-\rho^+ v_+ \\
            &= \parens*{\rho^+ \alpha^+ - \rho^+ \rho^-} + \rho^+ \parens*{v_+ - \gamma} = \rho^+ \parens*{v_+ - \gamma + \alpha^+ - \rho^-}.
         \end{align*}
        Then
        \begin{align*}
            \dfrac{y_u}{x_u} > \dfrac{\rho^+ \parens*{v_+ - \gamma + \alpha\displaystyle^+}}{\rho^- \parens*{\alpha\displaystyle^- + \gamma}} = \dfrac{\rho^+ \parens*{v_+ - \gamma + \alpha\displaystyle^+}}{\rho^+ \parens*{v_+ - \gamma + \alpha^+ - \rho^-}} = \dfrac{v_+ - \gamma + \alpha\displaystyle^+}{v_+ - \gamma + \alpha^+ - \rho^-},
        \end{align*}
        and the lower bound on $\dfrac{y_u}{x_u}$ also follows.
    \end{enumerate}
    That completes our proof of Lemma \ref{linearode}.
\end{proof}

\subsection{\texorpdfstring{Monotone Heteroclinic Orbit from $(1,1)$ to $(0,0)$}{Monotone Heteroclinic Orbit from (1,1) to (0,0)}}
The next result utilizes the results above to produce a \emph{heteroclinic orbit}---a solution trajectory of the dynamical system \eqref{sysFluid1} connecting the unstable equilibrium point $(1, 1)$ to the stable equilibrium point $(0, 0)$:
\begin{proposition}
\label{monoHeter}
    There exists a solution $\parens*{G_+, G_-}$ to \eqref{sysFluid1} that satisfies the following two conditions:
    \begin{enumerate}[label = (\alph*)]
        \item First, $G_+$ and $G_-$ are non-increasing in time.

        \item In addition, $\lim_{t\to-\infty} G_\pm(t) = 1$, $\lim_{t\to\infty} G_\pm(t) = 0$. Moreover, $G_+(t) \geq G_-(t)$ for each $t\in\real$, and furthermore,
          \begin{align}
        \label{tailBoundA}
            &\lim_{t\to\infty} \dfrac{\log G_\pm(t)}{t} = -\dfrac{v_+-\gamma+\alpha^+ +\alpha^- +\rho^+-\rho^-}{2\gamma} \nonumber \\
            &\qquad+\dfrac{1}{2} \sqrt{\parens*{\!\dfrac{v_+ - \gamma +\alpha^+ - \alpha^- -1}{\gamma}\!}^2 + \dfrac{4\parens*{\alpha\displaystyle^+ -\rho\displaystyle^-} \parens*{\rho\displaystyle^+ +\alpha\displaystyle^-+\gamma}}{\gamma\displaystyle^2}},
        \end{align}
        where the quantity on the right-hand side is the larger eigenvalue of $J(0, 0)$. Lastly,
          \begin{align}
            \label{tailBoundB}
            \lim_{t\to-\infty} \dfrac{\log \bigparens{1-G_\pm(t)}}{t} &= -\dfrac{v_+-\gamma+\alpha^+ +\alpha^-}{2\gamma} + \dfrac{1}{2}\sqrt{\parens*{\!\dfrac{v_+ - \gamma +\alpha^+ - \alpha^-}{\gamma}\!}^2 + \dfrac{4\alpha\displaystyle^+ \parens*{\alpha\displaystyle^-+\gamma}}{\gamma\displaystyle^2}},
        \end{align}
        where the quantity on the right-hand side is $\lambda_u$, the positive eigenvalue of $J(1, 1)$.
    \end{enumerate}
\end{proposition}
\begin{remark}
    If Proposition \ref{monoHeter} holds, then for each $t\in\real$, let $H_\pm(t) := G_\pm(-t)$. The logarithmic tail asymptotics in \eqref{tailBoundA} and \eqref{tailBoundB}, along with the fact that $\lambda_2 < 0 < \lambda_u$, imply that the distributions induced by $H_\pm$ have finite first moments; hence, Lemma \ref{twreversed} applies, and Theorem \ref{travelwave}\ref{b0} follows immediately.
\end{remark}
\begin{proof}[Proof of Proposition \ref{monoHeter}]
    For convenience, let $x(\cdot) := G_+(\cdot)$ and $y(\cdot) := G_-(\cdot)$. Define
    \begin{align*}
        y_1(x) := \dfrac{x}{1+\dfrac{\rho\displaystyle^-}{\alpha\displaystyle^+-\rho\displaystyle^-}(1-x)} & & \text{and} & & y_2(x) := \dfrac{x}{1 + \dfrac{\rho\displaystyle^-}{v_+ - \gamma + \alpha\displaystyle^+ - \rho\displaystyle^-}(1-x)},
    \end{align*}
    and let $\MMM := \curlbr*{(x, y) \in [0,1]\times [0,1]\!: y_1(x) \leq y \leq y_2(x)}$. It is easy to check that if $(x, y) \in \MMM$,
    then $\dot{x} \leq 0$ and $\dot{y} \leq 0$. [A quick argument using the ODE system \eqref{sysFluid1} shows that $y\geq y_1(x)$ implies $\dot{y} \leq 0$, and $y\leq y_2(x)$ implies $\dot{x} \leq 0$.]

      By the Stable Manifold Theorem \cite[107]{Perko01} and the Hartman-Grobman Theorem \cite[120]{Perko01} (see also \cite[169]{hirsch2013differential}), there exists some $t_0 > 0$, as well as a solution trajectory $\curlbr*{\big(x(t), y(t)\big)\!: t\leq -t_0}$ such that
      \begin{align*}
        \lim_{t\to-\infty} \big(x(t), y(t)\big) = (1, 1), & & \text{and} & & \lim_{t\to-\infty} \dfrac{y(t)-1}{x(t)-1} = \dfrac{y_u}{x_u},
        \end{align*}
    where $x_u$ and $y_u$ are defined as in Lemma \ref{linearode}. Noting that
    $$
    y_1'(1) = \dfrac{\alpha\displaystyle^+}{\alpha\displaystyle^+ - \rho\displaystyle^-}, \qquad y_2'(1) = \dfrac{v_+ - \gamma + \alpha\displaystyle^+}{v_+ - \gamma + \alpha\displaystyle^+ - \rho\displaystyle^-},
    $$
    by Lemma \ref{linearode}\ref{55c}, we can pick $t_0 > 0$ sufficiently large so that $\big(x(t), y(t)\big) \in \MMM$ for all $t\leq -t_0$. In particular, both $x(t)$ and $y(t)$ are non-increasing for $t \leq -t_0$.

    Next, since the vector field is smooth and points inward along the boundary of $\MMM$, the solution starting from $\bigparens{x(-t_0), y(-t_0)}$ remains in $\MMM$; furthermore, since $\MMM$ is compact, this solution extends uniquely to every $t\geq -t_0$. Indeed, first, suppose that there exists some   $t' \geq -t_0$ such that $\big(x(t'), y(t')\big)$ lies on $\curlbr*{(x, y)\!: y = y_1(x)}$. On $y = y_1(x)$, observe that
    \begin{align*}
        \dot{y}(t') = 0 & & \text{and} & & \dot{x}(t') < 0,
    \end{align*}
    and since $y_1$ is increasing, the vector field points above the lower boundary and into $\MMM$. Similarly, if there exists some $t'' \geq -t_0$ such that $\big(x(t''), y(t'')\big)$ lies on $\curlbr*{(x, y)\!: y = y_2(x)}$, then on $y = y_2(x)$,
    \begin{align*}
        \dot{x}(t'') = 0 & & \text{and}& & \dot{y}(t'') < 0,
    \end{align*}
    and so the vector field points below the upper boundary and into $\MMM$ (again). Either way, along this solution trajectory, $x(t)$ and $y(t)$ are non-increasing for all $t \geq -t_0$. More specifically,
    \begin{align*}
        x_0 := \lim_{t\to\infty} x(t) & & \text{and} & & y_0 := \lim_{t\to\infty} y(t)
    \end{align*}
    both exist, with $x_0 < 1$ and $y_0 < 1$. Since $\bigparens{x(t), y(t)} \xrightarrow{\sspa t\to\infty \sspa} (x_0, y_0)$ and the vector field is continuous, $(x_0, y_0)$ must be an equilibrium of \eqref{sysFluid1}, and since the trajectory is non-constant, with $x_0 < 1$ and $y_0 < 1$, from Lemma \ref{linearode}, it must be the case that $x_0 = y_0 = 0$. In addition, as this solution trajectory is ``trapped'' inside $\MMM$, which lies below the line $y = x$, we have that $x(t) \geq y(t)$ for all $t\in\real$.

    The remaining task is to show that the identities for the rate of growth for $G_\pm$ hold. In particular, let us show that
    \begin{align*}
        \lim_{t\to-\infty} \dfrac{\log\bigparens{1-G_+(t)}}{t} = \lambda_u,
    \end{align*}
    and the proof of $\lim_{t\to-\infty} \dfrac{\log\bigparens{1-G_-(t)}}{t} = \lambda_u$ follows by a parallel argument. To this end, we recenter the first equation of \eqref{sysFluid1}. Using the notation we adopt for this proof, we obtain
    \begin{align*}
        \dot{x} = &-\dfrac{\rho^-}{\gamma} (x-1)(y-1) - \dfrac{v_+ -\gamma+ \alpha^+ - \rho^-}{\gamma} (x-1) + \dfrac{\rho^- \parens*{\rho^+ + \alpha^-+\gamma}}{\gamma \rho\displaystyle^+} (y-1) \\
        &- \dfrac{\rho^-}{\gamma} (y-1) - \dfrac{\rho^-}{\gamma} (x-1) - \dfrac{\rho^-}{\gamma} - \dfrac{v_+ -\gamma+ \alpha^+ - \rho^-}{\gamma} + \dfrac{\rho^- \parens*{\rho^+ + \alpha^-+\gamma}}{\gamma \rho\displaystyle^+}.
    \end{align*}
      Next, since $(1,1)$ is an equilibrium point of \eqref{sysFluid1},
    \begin{align*}
        - \dfrac{\rho^-}{\gamma} - \dfrac{v_+ -\gamma+ \alpha^+ - \rho^-}{\gamma} + \dfrac{\rho^- \parens*{\rho^+ + \alpha^-+\gamma}}{\gamma \rho\displaystyle^+} = 0,
    \end{align*}
    so
    \begin{align*}
        \dot{x} =&- \dfrac{\rho^-}{\gamma} (x-1)(y-1) - \dfrac{v_+ -\gamma+ \alpha^+ - \rho^-}{\gamma} (x-1) \\
        &\qquad + \dfrac{\rho^- \parens*{\rho^+ + \alpha^-+\gamma}}{\gamma \rho\displaystyle^+} (y-1) - \dfrac{\rho^-}{\gamma} (y-1) - \dfrac{\rho^-}{\gamma} (x-1).
    \end{align*}
    This implies
      \begin{align*}
        -\dfrac{\dot{x}}{1-x} &= -\dfrac{\rho^-}{\gamma} (y-1) - \dfrac{v_+ -\gamma+ \alpha^+ - \rho^-}{\gamma} + \dfrac{\rho^- \parens*{\rho^+ + \alpha^-+\gamma}}{\gamma \rho\displaystyle^+} \cdot \dfrac{y-1}{x-1} - \dfrac{\rho^-}{\gamma}\cdot \dfrac{y-1}{x-1} - \dfrac{\rho^-}{\gamma} \\
        &= -\dfrac{\rho^-}{\gamma} - \dfrac{v_+ -\gamma+ \alpha^+ - \rho^-}{\gamma} + \dfrac{\rho^- \parens*{\rho^+ + \alpha^-+\gamma}}{\gamma \rho\displaystyle^+} \cdot \dfrac{y_u}{x_u} - \dfrac{\rho^-}{\gamma} \cdot \dfrac{y_u}{x_u} + r(t) \\
        &= - \dfrac{v_+ - \gamma + \alpha^+}{\gamma} + \dfrac{\rho^- \parens*{\alpha^- + \gamma}}{\gamma \rho\displaystyle^+} \cdot \dfrac{y_u}{x_u} + r(t) = \lambda_u + r(t),
    \end{align*}
    where $\lim_{t\to-\infty} r(t) = 0$, and the last equality follows from \eqref{slopeEigenv}. Hence, for any $\varepsilon > 0$, there exists some $t_\varepsilon > 0$ such that for all $t\leq -t_\varepsilon$, $\lambda_u - \varepsilon \leq - \dfrac{\dot{x}(t)}{1-x(t)} \leq \lambda_u + \varepsilon$, and by integrating each of the expressions above from any $t< -t_\varepsilon$ to $-t_\varepsilon$, we obtain
    \begin{align*}
        &\phantom{\equi} \int\limits_{t}^{-t_\varepsilon} (\lambda_u - \varepsilon) \di u \leq \int\limits_{t}^{-t_\varepsilon} - \dfrac{\dot{x}(u)}{1-x(u)} \di u \leq \int\limits_{t}^{-t_\varepsilon} (\lambda_u + \varepsilon) \di u,
    \end{align*}
    or equivalently, as $\int\limits_{t}^{-t_\varepsilon} - \dfrac{\dot{x}(u)}{1-x(u)} \di u = \log\bigparens{1-x(-t_\varepsilon)} - \log\bigparens{1-x(t)}$,
    \begin{align*}
        -(\lambda_u - \varepsilon)(t+t_\varepsilon) \leq \log\bigparens{1-x(-t_\varepsilon)} - \log\bigparens{1-x(t)} \leq -(\lambda_u+\varepsilon)(t+t_\varepsilon).
    \end{align*}
    Dividing each expression by $-t \geq t_\varepsilon > 0$ gives
    \begin{align*}
        \parens*{1+\dfrac{t_\varepsilon}{t}}\parens*{\lambda_u - \varepsilon} \leq \dfrac{\log\bigparens{1-x(t)}}{t} - \dfrac{\log\bigparens{1-x(-t_\varepsilon)}}{t} \leq \parens*{1+\dfrac{t_\varepsilon}{t}}\parens*{\lambda_u + \varepsilon}.
    \end{align*}
    Lastly, observe that $\lim_{t\to-\infty} \dfrac{\log\bigparens{1-x(-t_\varepsilon)}}{t} = 0$, so by letting $t$ tend to $-\infty$, and as $\varepsilon>0$ is arbitrary, we obtain
    \begin{align*}
        \lim_{t\to-\infty} \dfrac{\log\bigparens{1-x(t)}}{t} = \lambda_u,
    \end{align*}
    which proves \eqref{tailBoundB}.

    The rest of the task is to show \eqref{tailBoundA}, which involves obtaining $\lim_{t \to \infty}\dfrac{\log G_{\pm}(t)}{t}$. To that end, first, observe that, by Theorem 2 in \cite[140--141]{Perko01}, $\lim_{t \to \infty} \dfrac{y(t)}{x(t)}$ exists and takes the form $\dfrac{y}{x}$ for an eigenvector of $J(0,0)$. By our previous arguments and Lemma \ref{linearode}\ref{55b}, this ratio has to be positive and hence correspond to an eigenvector for the larger eigenvalue $\lambda_2$.

    Next, let $s := \lim_{t\to\infty} \dfrac{y(t)}{x(t)} \in (0, \infty)$. Since the vector field is differentiable at $(0, 0)$, for large enough $t$,
    \begin{align}
    \label{vecField1}
        \begin{bmatrix} \dot{x}(t) \\ \dot{y}(t) \end{bmatrix} = J(0, 0) \begin{bmatrix} x(t) \\ y(t) \end{bmatrix} + \smallo\parens*{\lnorm*{\begin{bmatrix} x(t) \\ y(t) \end{bmatrix}}}.
    \end{align}
    Since the non-constant trajectory cannot reach the equilibrium $(0, 0)$ in finite time, for each $t\in\real$, $x(t) > 0$, so dividing both sides of \eqref{vecField1} gives
    \begin{align}
    \label{vecField2}
        \begin{bmatrix} \dot{x}(t)/x(t) \\ \dot{y}(t)/x(t) \end{bmatrix} = J(0, 0) \begin{bmatrix} 1 \\ y(t)/x(t) \end{bmatrix} + \dfrac{1}{x(t)} \smallo\parens*{\lnorm*{\begin{bmatrix} x(t) \\ y(t) \end{bmatrix}}}.
    \end{align}
    We want to extract the long-time limit of the first coordinate on the left-hand side of \eqref{vecField2}. To that end, we treat each term on the right-hand side of \eqref{vecField2} separately. For the remainder term, since $x(t) > 0$ for each $t\in\real$,
    \begin{align*}
        \dfrac{1}{x(t)} \lnorm*{\begin{bmatrix} x(t) \\ y(t) \end{bmatrix}} = \dsqrt{1 + \parens*{\dfrac{y(t)}{x(t)}}^2} \toinf{t} \dsqrt{1 + s^2} < \infty,
    \end{align*}
    so 
\mbox{
$\dfrac{1}{x(t)} \smallo\parens*{\lnorm*{\begin{bmatrix} x(t) \\ y(t) \end{bmatrix}}} \toinf{t} 0$}. Then from \eqref{vecField2}, it follows that
    \begin{align*}
        \lim_{t\to\infty} \begin{bmatrix} \dot{x}(t)/x(t) \\ \dot{y}(t)/x(t) \end{bmatrix} = J(0, 0) \begin{bmatrix} 1 \\ s \end{bmatrix} = \lambda_2 \begin{bmatrix} 1 \\ s \end{bmatrix},
    \end{align*}
    and extracting the first coordinate gives $\lim_{t\to\infty} \dfrac{\dot{x}(t)}{x(t)} = \lambda_2$. Thus, for any $r\in\real$,
    \begin{align*}
        \dfrac{\log x(t)}{t} = \dfrac{\log x(r)}{t} + \dfrac{1}{t} \int\limits_r^t \dfrac{\dot{x}(z)}{x(z)} \di z \toinf{t} \lambda_2.
    \end{align*}
    At the same time, since $\lim_{t\to\infty} \dfrac{y(t)}{x(t)} = s \in (0, \infty)$,
    \begin{align*}
        \dfrac{\log y(t)}{t} = \dfrac{\log x(t)}{t} + \dfrac{1}{t} \log\parens*{\dfrac{y(t)}{x(t)}} \toinf{t} \lambda_2.
    \end{align*}
    Hence, $\lim_{t\to\infty} \dfrac{\log G_-(t)}{t} = \lambda_2$, so \eqref{tailBoundA} holds as well. Together with \eqref{tailBoundB}, this gives us Proposition \ref{monoHeter}, as desired.
 \end{proof}
Now, with Theorem \ref{travelwave} proven, we shall examine the special regimes where $\alpha^-$ is very small (or very large), as observed in Corollary \ref{tailclean}:
\begin{proof}[Proof of Corollary \ref{tailclean}]
We shall proceed:
\begin{enumerate}[label = (\alph*)]
    \item First, let us examine the asymptotic behaviors of $\rho^+$ and $\gamma$ as $\alpha^-$ shrinks to zero. Recall that $\gamma=v_+\rho^+$, and let $s:=\alpha^++\alpha^-$ be chosen. (Accordingly, since $\alpha^+ \geq 1$, we also have $s\geq 1$ as well.) With that,
    \begin{align*}
        \rho^+=\frac{1-s+\sqrt{(1-s)\displaystyle^2+4\alpha^-}}{2}.
    \end{align*}
    In addition, let $d := \alpha^+ - \alpha^-$, and note that $\lim_{\alpha^-\downarrow 0} d = s$, which implies that $\varepsilon:=s-d \downarrow 0$. Also, note that $\alpha^- = \frac{\varepsilon}{2}$ and $\alpha^+ = s-\frac{\varepsilon}{2}$; hence,
    \begin{align*}
        \rho^+ = \frac{1-s+\sqrt{(1-s)\displaystyle^2+2\varepsilon}}{2} = \frac{\varepsilon}{\sqrt{(s-1)\displaystyle^2+2\varepsilon}+s-1} =\frac{\varepsilon}{2(s-1)}+ \bigO\parens*{\varepsilon^2}.
    \end{align*}
    This, in turn, implies that
    \begin{align*}
        \gamma=v_+\rho^+=\frac{v_+\varepsilon}{\sqrt{(s-1)\displaystyle^2+2\varepsilon}+s-1}=\frac{v_+}{2(s-1)} \varepsilon+\bigO\parens*{\varepsilon^2},
    \end{align*}
    and the asymptotic behaviors for $\rho^+$ and $\gamma$ thus follow.

    Next, we want to look at the left-tail and the right-tail behavior in the small-$\alpha^-$-regime. We take the case of the right-tail first. Recall that
    \begin{align*}
        \lambda_R = -\frac{v_+-\gamma+\alpha^++\alpha^-}{2\gamma} +\frac{1}{2}\sqrt{ \parens*{\frac{v_+-\gamma+\alpha^+-\alpha^-}{\gamma}}^2 + \frac{4\alpha^+(\alpha^-+\gamma)}{\gamma\displaystyle^2}}.
    \end{align*}
    Now, substituting $\alpha^- = \frac{\varepsilon}{2}$ and $\alpha^+ = s-\frac{\varepsilon}{2}$, and treating $\lambda_R$ as a function of $\varepsilon$ give
    \begin{align*}
        \lambda_R = \lambda_R(\varepsilon) &= -\frac{v_+-\gamma+s}{2\gamma} +\frac{1}{2}\sqrt{\left(\frac{v_+-\gamma+s-\varepsilon}{\gamma}\right)^2 +\frac{4\left(s-\frac{\varepsilon}{2}\right)\left(\frac{\varepsilon}{2}+\gamma\right)}{\gamma\displaystyle^2}} \\
        &= \frac{-(v_+-\gamma+s)+\sqrt{(v_+-\gamma+s-\varepsilon)\displaystyle^2 + 4\left(s-\frac{\varepsilon}{2}\right)\left(\frac{\varepsilon}{2}+\gamma\right)}}{2\gamma}.
    \end{align*}
    For shorthand, define $C_\varepsilon:=v_+-\gamma+s$, and let
    \begin{align*}
        \delta_\varepsilon &:= (C_\varepsilon-\varepsilon)^2-C_\varepsilon^2+4\left(s-\frac{\varepsilon}{2}\right)\left(\frac{\varepsilon}{2}+\gamma\right) = -2C_\varepsilon\varepsilon+\varepsilon^2+4\left(s-\frac{\varepsilon}{2}\right)\left(\frac{\varepsilon}{2}+\gamma\right),
    \end{align*}
    so that $\lambda_R(\varepsilon)$ can be written (more compactly) as $\lambda_R(\varepsilon) = \frac{-C_\varepsilon+\sqrt{C\displaystyle^2_\varepsilon+\delta_\varepsilon}}{2\gamma}$. As $\gamma = \bigO(\varepsilon)$, expanding $\delta_\varepsilon$ gives
    \begin{align*}
        \delta_\varepsilon &= -2C_\varepsilon\varepsilon+\varepsilon^2+2s\varepsilon+4s\gamma-\varepsilon^2-2\varepsilon\gamma \\
        &= \varepsilon\bigl(-2C_\varepsilon - 2 \gamma +2s\bigr)+4s\gamma\\
        &=-2v_+\varepsilon + 4s\gamma = \varepsilon\left(-2v_+ + \frac{2sv_+}{s-1}\right) + \bigO\parens*{\varepsilon^2} = \frac{2v_+}{s-1}\varepsilon + \bigO\parens*{\varepsilon^2},
    \end{align*}
    The last display gives us that $\delta_\varepsilon= \bigO(\varepsilon)$, and so $\sqrt{C\displaystyle^2_\varepsilon+\delta_\varepsilon} = C_\varepsilon + \frac{\delta_\varepsilon}{2C_\varepsilon} + \bigO\parens*{\varepsilon^2}$. Plugging this into $\lambda_R(\varepsilon)$ yields
    \begin{align*}
        \lambda_R(\varepsilon) = \frac{1}{2\gamma}\left(\frac{\delta_\varepsilon}{2C_\varepsilon}+\bigO\parens*{\varepsilon^2}\right) = \frac{\delta_\varepsilon}{4\gamma C_\varepsilon} + \bigO(\varepsilon).
    \end{align*}
    Now, since $\dfrac{\rho^+}{\alpha\displaystyle^-} \xrightarrow{\sspa \alpha^- \downarrow 0 \sspa} \dfrac{1}{s-1}$ and $\dfrac{\gamma}{\alpha\displaystyle^-} \xrightarrow{\sspa \alpha^- \downarrow 0 \sspa} \dfrac{v_+}{s-1}$, we have that $C_\varepsilon \xrightarrow{\sspa \alpha^- \downarrow 0 \sspa} v_++s$ and $\frac{\gamma}{\varepsilon} \xrightarrow{\sspa \alpha^- \downarrow 0 \sspa} \frac{v_+}{2(s-1)}$. That gives us
    \begin{align*}
        \lambda_R(\varepsilon) = \frac{\frac{2v_+}{s-1}\varepsilon + \bigO\parens*{\varepsilon^2}}{4\gamma(v_++s)} + \bigO(\varepsilon) &= \frac{\frac{2v_+}{s-1}\varepsilon + \bigO\parens*{\varepsilon^2}}{4\left(\frac{v_+}{2(s-1)}\varepsilon + \bigO\parens*{\varepsilon^2}\right)(v_+ + s)} + \bigO(\varepsilon) \\
        &= \frac{1}{v_++s} + \bigO(\varepsilon),
    \end{align*}
    which is the appropriate asymptotic behavior of the right-tail exponent.

    Next, we shall take up the asymptotic behavior of the left-tail. To that end, let
    \begin{align*}
        A &:= v_+-\gamma+\alpha^++\alpha^-+\rho^+-\rho^-, \\
        M &:= v_+-\gamma+\alpha^+-\alpha^--1,
        \intertext{and}
        K &:= (\alpha^+-\rho^-)(\rho^++\alpha^-+\gamma).
    \end{align*}
    In addition, $\lambda_L$ can be treated as a function of $d = \alpha^+ - \alpha^-$:
    \begin{align*}
        \lambda_L = \lambda_L(d) = \frac{A-\sqrt{M\displaystyle^2 + 4K}}{2\gamma}.
    \end{align*}
    Now, observe that $A$ can be re-expressed as
    \begin{align*}
        A = v_+-\gamma+s+\rho^+-\rho^- = v_+-\gamma+s+2\rho^+-1 = v_+-\gamma+\sqrt{(s-1)\displaystyle^2+2\varepsilon},
    \end{align*}
    and
    \begin{align*}
        M=v_+-\gamma+d-1 = v_+-\gamma+s-\varepsilon-1.
    \end{align*}
    [In addition, $M \xrightarrow{\sspa \alpha^- \downarrow 0 \sspa} v_++s-1>0$.] Naturally,
    \begin{align*}
        A-M = \sqrt{(s-1)^2+2\varepsilon}-(s-1)+\varepsilon = \frac{s}{s-1}\varepsilon + \bigO\parens*{\varepsilon^2}.
    \end{align*}
    We now look at the asymptotic behavior of $K$. With $\alpha^+-\rho^- \xrightarrow{\sspa \alpha^- \downarrow 0 \sspa} s-1$ and
    \begin{align*}
        \frac{\rho^++\alpha^-+\gamma}{\varepsilon} \xrightarrow{\sspa \alpha^- \downarrow 0 \sspa} \frac{1}{2(s-1)}+\frac{1}{2}+\frac{v_+}{2(s-1)} = \frac{v_++s}{2(s-1)},
    \end{align*}
    we have $K=\frac{v_++s}{2} \varepsilon + \bigO\parens*{\varepsilon^2}$.  Putting the pieces together, we have
    \begin{align*}
        \sqrt{M\displaystyle^2+4K} = M+\frac{2K}{M} + \bigO\parens*{\varepsilon^2},
    \end{align*}
    which implies
    \begin{align*}
        A-\sqrt{M\displaystyle^2+4K} &= (A-M)-\frac{2K}{M} + \bigO\parens*{\varepsilon^2} \\
        &= \parens*{\frac{s}{s-1}-\frac{v_++s}{v_++s-1}}\varepsilon + \bigO\parens*{\varepsilon^2} = \frac{v_+\varepsilon}{(s-1)(v_++s-1)} + \bigO\parens*{\varepsilon^2}.
    \end{align*}
    Lastly, recall that in this regime, $2\gamma = \frac{v_+\varepsilon}{s-1} + \bigO\parens*{\varepsilon^2}$, and the asymptotic behavior of the left-tail follows.

    \item Next, we examine the large-$\alpha^-$ regime, where $\alpha^- \uparrow \infty$ and $\alpha^+ = \smallo\parens*{\alpha^-}$. In this case, with
    \begin{align*}
        \rho^+ = \frac{1-\alpha^+ - \alpha^- +\sqrt{\parens*{1-\alpha^+ - \alpha^-}\displaystyle^2+4} \alpha^-}{2},
    \end{align*}
    we have
    \begin{align*}
        \rho^- = 1-\rho^+ &= \frac{1 + \alpha^+ + \alpha^- - \sqrt{(1- \alpha^+ -\alpha^-)\displaystyle^2 + 4\alpha^-}}{2} \\
        &= \frac{2\alpha^+}{\sqrt{\parens*{\alpha\displaystyle^+ + \alpha\displaystyle^- - 1}\displaystyle^2+4\alpha^-}+(1 + \alpha^+ + \alpha^-)} = \frac{\alpha^+}{\alpha^+ + \alpha^-} \bigparens{1 + \smallo(1)} \\
        &= \frac{\alpha^+}{\alpha^-} \bigparens{1 + \smallo(1)},
    \end{align*}
    and so $\gamma=v_+\rho^+=v_+\parens*{1-\rho\displaystyle^-}=v_+-v_+ \rho^- = v_+ - \frac{v_+\alpha^+}{\alpha^-}\bigparens{1 + \smallo(1)}$, which gives the asymptotic behaviors for $\rho^-$ and $\gamma$.

    We come to the tails next---first, the right tail. With $C_R := v_+-\gamma+ \alpha^+ + \alpha^-$, we have
    \begin{align*}
        \lambda_R = \frac{-C_R+\sqrt{\parens*{v_+-\gamma + \alpha\displaystyle^+ - \alpha\displaystyle^-}\displaystyle^2+4\alpha^+(\alpha\displaystyle^- +\gamma)}}{2\gamma}.
    \end{align*}
    In addition, from the facts that $v_+-\gamma=v_+ \rho^-$ and $\parens*{\rho\displaystyle^+}^2+\parens*{\alpha^+ + \alpha^- - 1}\rho^+ - \alpha^- = 0$, we have (after algebraic manipulations)
    \begin{align*}
        \parens*{v_+-\gamma+\alpha^+ - \alpha^-}^2 + 4\alpha^+(\alpha^-+\gamma)=C^2_R+4\gamma \rho^-.
    \end{align*}
    Then we can rewrite $\lambda_R$ as
    \begin{align*}
        \lambda_R = \frac{\sqrt{C\displaystyle^2_R+4\gamma \rho\displaystyle^-}-C_R}{2\gamma} = \frac{2\rho^-}{\sqrt{C\displaystyle^2_R+4\gamma \rho\displaystyle^-}+C_R}.
    \end{align*}
    Now, since $\rho^- = \smallo(1)$ and $\gamma \xrightarrow{\sspa \alpha^- \uparrow \infty \sspa} v_+$, we have
    \begin{align*}
        C_R = \alpha^+ + \alpha^- + v_+\rho^- = \parens*{\alpha^+ + \alpha^-} \bigparens{1+\smallo(1)},
    \end{align*}
    and putting everything together,
    \begin{align*}
        \lambda_R = \frac{2\rho^-}{\sqrt{C\displaystyle^2_R+4\gamma \rho\displaystyle^-}+C_R} = \dfrac{\rho^-}{C_R} \bigparens{1 + \smallo(1)} =\frac{\alpha^+}{\parens*{\alpha\displaystyle^+ + \alpha\displaystyle^-}\displaystyle^2}\bigparens{1+\smallo(1)} = \frac{\alpha^+}{(\alpha^-)^2} \bigparens{1+\smallo(1)},
    \end{align*}
    which gives the right-tail exponent asymptotics.

    Lastly, we examine the asymptotic behavior of the left tail. With $A$, $M$, and $K$ defined earlier in this proof,
    \begin{align*}
        \lambda_L = \frac{A-\dsqrt{M\displaystyle^2+4K}}{2\gamma}.
    \end{align*}
    More specifically, in this regime, observe that
    \begin{align*}
        A = v_+-\gamma+ \alpha^+ + \alpha^- +1-2\rho^- = \alpha^+ + \alpha^-+1 + o\parens*{\alpha^+ + \alpha^-}.
    \end{align*}
    Again, from the facts that $v_+-\gamma=v_+ \rho^-$ and $\parens*{\rho\displaystyle^+}^2+\parens*{\alpha^+ + \alpha^- - 1}\rho^+ - \alpha^- = 0$, after algebraic manipulations,
    \begin{align*}
        A^2 - \parens*{M\displaystyle^2+4K} = 4\gamma \rho^-.
    \end{align*}
    With this, in light of the asymptotic behavior of $A$, $\alpha^+$, and $\alpha^-$, we have
    \begin{align*}
        \lambda_L = \dfrac{A\!-\!\dsqrt{A\displaystyle^2\!-\!4\gamma \rho\displaystyle^-}}{2\gamma} &= \dfrac{2\rho\displaystyle^-}{A+\dsqrt{A\displaystyle^2-4\gamma \rho^-}} \\
        &= \frac{\rho^-}{A}\bigparens{1\!+\!\smallo(1)}\!=\!\frac{\alpha^+}{\parens*{\alpha\displaystyle^+\!+\!\alpha\displaystyle^-} \parens*{\alpha\displaystyle^+\!+\!\alpha\displaystyle^-\!+\! 1}}\bigparens{1\!+\! \smallo(1)}\!=\!\frac{\alpha^+}{\parens*{\alpha\displaystyle^-}\displaystyle^2} \bigparens{1\!+\!\smallo(1)},
    \end{align*}
    as desired.
\end{enumerate}
That gives us the proof of Corollary \ref{tailclean}.
\end{proof}

\textbf{Acknowledgements:} SB and AN were supported in part by the NSF-CAREER award DMS-2141621. SB was supported in part by the NSF-RTG award DMS-2134107.

\printbibliography[heading = bibintoc]

\vspace{\baselineskip}

\noindent{\scriptsize {\textsc{\noindent S. Banerjee and A. Nguyen,\newline
Department of Statistics and Operations Research\newline
University of North Carolina\newline
Chapel Hill, NC 27599, USA\newline
email: sayan@email.unc.edu
\newline
email: andrwn@unc.edu
\vspace{\baselineskip} } }}
\end{document}